\documentclass[11pt,a4paper,reqno]{amsart}
\usepackage[T1]{fontenc}
\usepackage{lmodern}
\usepackage[margin=27mm,top=25mm,bottom=25mm,headheight=14pt]{geometry}
\usepackage{amsmath,amssymb,amsthm,mathtools,mathrsfs,bm}
\usepackage{graphicx,xcolor,tikz-cd,enumitem,needspace}
\usepackage[section]{placeins}
\usepackage{etoolbox}
\usetikzlibrary{arrows.meta,calc,positioning,bending,decorations.markings}
\usepackage[unicode,hidelinks]{hyperref}
\hypersetup{pdftitle={Exchange graphs and tau-tilting complexes of surface rigid algebras},pdfauthor={Ping He and Bo Le},pdfsubject={Exchange components, spherical homotopy types and endpoint ray sets}}
\usepackage{xcolor}
\usepackage[normalem]{ulem}

\newcommand{\delete}{%
	\bgroup
	\markoverwith{\textcolor{red}{\rule[0.5ex]{2pt}{0.4pt}}}%
	\ULon
}

\DeclareMathOperator{\Hom}{Hom}
\DeclareMathOperator{\End}{End}

\newcommand{\K}{\mathcal K}

\newcommand{\twopsilt}{2\text{-}\operatorname{psilt}}
\newcommand{\indtwopsilt}{\operatorname{ind}\,2\text{-}\operatorname{psilt}}

\numberwithin{equation}{section}
\theoremstyle{plain}
\newtheorem{theorem}{Theorem}[section]
\newtheorem{lemma}[theorem]{Lemma}
\newtheorem{proposition}[theorem]{Proposition}
\newtheorem{corollary}[theorem]{Corollary}
\newtheorem*{mainthma}{Theorem A}
\newtheorem*{mainthmb}{Theorem B}
\newtheorem*{mainthmc}{Theorem C}
\theoremstyle{definition}
\newtheorem{definition}[theorem]{Definition}
\newtheorem{example}[theorem]{Example}
\newtheorem{remark}[theorem]{Remark}

\title[Exchange graphs and $\tau$-tilting complexes]{Exchange graphs and $\tau$-tilting complexes of surface rigid algebras}
\author[Ping He]{Ping He\textsuperscript{*}}
\address{Ping He, Beijing Institute of Mathematical Sciences and Applications, Beijing 101408, People's Republic of China}
\email{pinghe@bimsa.cn}
\author{Bo Le}
\address{Bo Le, Department of Mathematical Sciences, Tsinghua University, Beijing 100084, People's Republic of China}
\email{bole\_math@163.com}
\thanks{* Corresponding author.}
\date{}

\begin{document}

\begin{abstract} Let $A=\Lambda_{\mathbf R}$ be the surface rigid algebra associated with a partial tagged triangulation $\mathbf R$ of a punctured marked surface with nonempty boundary, and let $C_A$ be the Cartan matrix of $A$. We classify the connected components of the support $\tau$-tilting graph of $A$. For connected $A$,  such a graph is disconnected if and only if $A$ is non-local and weakly symmetric with $\operatorname{rank}C_A=1$.  The graph then has exactly two components. The completion graph of every nonzero basic $\tau$-rigid pair over connected $A$ is connected. When $\mathbf R$ is nonempty and all arc endpoints are punctures, we construct an explicit homotopy equivalence from the realization of the support $\tau$-tilting complex to a sphere of dimension $\operatorname{rank}C_A-1$. Under the same hypothesis, $A$ is $\tau$-tilting finite if and only if $C_A$ is positive definite. In the finite case, we give a root-system realization of the $g$-vectors of indecomposable $\tau$-rigid pairs, with types $A$, $D$ and $C$ determined by the tagged surface configuration. \end{abstract}

\maketitle \tableofcontents

\section*{Introduction}

In \cite{AIR}, Adachi, Iyama and Reiten introduced support $\tau$-tilting theory for arbitrary finite-dimensional algebras. They proved that every basic almost complete $\tau$-rigid pair has exactly two basic $\tau$-tilting completions. So replacing one complement by the other defines mutation. The exchange graph has basic $\tau$-tilting pairs as vertices and mutations as edges. Although mutation is possible at every indecomposable summand, the exchange graph need not be connected.

Tagged triangulations describe cluster mutation \cite{FST}, and tagged arcs represent indecomposable rigid objects in the associated surface cluster category \cite{BZ,QZ}. In \cite{FGLZ}, Fu, Geng, Liu and Zhou proved connectedness and the reachable-in-face property for the exchange graphs of gentle algebras. In \cite{HZZ}, He, Zhou and Zhu constructed a geometric model for the endomorphism algebras of arbitrary rigid objects in surface cluster categories. These surface rigid algebras include skew-gentle algebras. Their $\tau$-tilting pairs correspond to dissections relative to a partial tagged triangulation $\mathbf R$, and mutations correspond to flips. They proved that the exchange graph is connected if every connected component of $\mathbf R$ meets the boundary \cite[Theorem~5.16]{HZZ}. We classify the connected components of the exchange graph and of the completion graphs without this boundary assumption.

Let $k$ be an algebraically closed field. Let $\mathbf S$ be a punctured marked surface with nonempty boundary, and $\mathbf R$ a partial tagged triangulation of $\mathbf S$. Write \[ A=\Lambda_{\mathbf R} =\operatorname{End}_{\mathcal C(\mathbf S)}X(\mathbf R), \] where $X(\mathbf R)$ is the corresponding basic rigid object in the cluster category $\mathcal C(\mathbf S)$. Denote the Cartan matrix of $A$ by $C_A$. For a basic $\tau$-rigid pair $u$, let $\mathcal F_A(u)$ be the graph of basic $\tau$-tilting completions of $u$, with edges given by mutations preserving $u$.

\begin{mainthma}[{Theorems~\ref{res:3.5} and~\ref{res:5.8}}] Let $A$ be a connected surface rigid algebra. The exchange graph of $A$ is disconnected if and only if $A$ is non-local and weakly symmetric with $\operatorname{rank}C_A=1$. In this case,  the graph has exactly two components, containing $(A,0)$ and $(0,A)$, respectively. Moreover, $\mathcal F_A(u)$ is connected for every nonzero basic $\tau$-rigid pair $u$. \end{mainthma}

To prove the assertion about $\mathcal F_A(u)$, we realize $\tau$-tilting reduction by cutting along the arcs representing $u$. For connected $A$ and $u\neq(0,0)$, the resulting partial tagged triangulation is connected to the boundary. So \cite[Theorem~5.16]{HZZ} applies. 

For arbitrary $A$, the numbers of exchange and completion components follow by taking products over the blocks (Corollary~\ref{res:6.1}). The same corollary gives the reachable-in-face property: two mutation-reachable completions of a common summand can be joined by mutations preserving that summand.

Geometrically, the disconnected case in Theorem~A occurs precisely when $|\mathbf R|\geq2$, all endpoints are punctures, and the underlying arcs of $\mathbf R$ are pairwise distinct with the same endpoint multiset. Example~\ref{res:6.6} takes two plain-tagged arcs joining the same two punctures in an annulus. Their union separates the boundary components and their endpoint columns are both $(1,1)^t$, so $C_A=\left(\begin{smallmatrix}2&2\\2&2\end{smallmatrix}\right)$. Thus the algebra is connected but the exchange graph of $A$ has two components. This example also shows that the boundary condition cannot be omitted from the first assertion of \cite[Corollary~5.18]{HZZ}, whose proof uses \cite[Theorem~5.16]{HZZ}. The skew-gentle conclusion of that corollary remains valid.

For a connected special biserial algebra $B$, Asai proved that the two-term silting graph has at most two components, each containing $B$ or $B[1]$, and characterized disconnectedness by $g$-vector cones \cite[Corollary~6.6]{Asai}. For certain Brauer graph algebras, Aihara proved that $B$ and $B[1]$ lie in different components of the full tilting quiver \cite[Theorem~2.6]{Ai}. Together, these results give exactly two components in the two-term silting graph for Aihara's algebras. Theorem~A gives a criterion in terms of weak symmetry and $\operatorname{rank}C_A$ for surface rigid algebras.

To study the puncture components of $\mathbf R$, we consider the support $\tau$-tilting complex $\Delta_\tau(A)$. The faces of $\Delta_\tau(A)$ are the basic $\tau$-rigid pairs, ordered by direct summands. Thus the vertices of $\Delta_\tau(A)$ are the indecomposable pairs and the maximal faces are the $\tau$-tilting pairs. We give $|\Delta_\tau(A)|$ the weak topology. If $A$ is $\tau$-tilting finite with $n$ simple modules, then $|\Delta_\tau(A)|$ is homeomorphic to $S^{n-1}$ by \cite[Theorem~5.4]{DIJ}. 

Call $\mathbf R$ \emph{punctured} if every arc in $\mathbf R$ has both endpoints at punctures. Suppose $\mathbf R$ is punctured, and encode plain and notched taggings by $1$ and $-1$. For an $\mathbf R$-standard arc $\gamma$ (Definition~\ref{res:1.8}), let $\eta(\gamma)_p$ be the sum of the taggings at the ends of $\gamma$ incident with $p$. Both ends of a loop are counted. Let $B$ have columns $\eta(\gamma_i)$ for $\gamma_i\in\mathbf R$. Propositions~\ref{res:4.1} and~\ref{prop:endpoint-formula} give \[ C_A=B^tB\qquad\text{and}\qquad Bg_{\mathbf R}(\gamma)=\eta(\gamma). \]

The following result also applies when $A$ is $\tau$-tilting infinite.

\begin{mainthmb}[{Theorem~\ref{top:surface}}] Suppose $\mathbf R\neq\varnothing$ is punctured, and  let $r=\operatorname{rank}C_A$. Then \[ |\Delta_\tau(A)|\simeq S^{r-1}. \] Under the identification with compatible $\mathbf R$-standard arcs, a homotopy equivalence is \[ \Psi_{\mathbf R}:|\Delta_\tau(A)| \longrightarrow S(\operatorname{im}_{\mathbb R}B),\qquad \sum_i t_i\alpha_i\longmapsto \frac{\sum_i t_i\eta(\alpha_i)} {\|\sum_i t_i\eta(\alpha_i)\|}, \] where the $\alpha_i$ are compatible, $t_i\geq0$, $\sum_i t_i=1$, and $S(W)$ denotes the unit sphere of a Euclidean space $W$. \end{mainthmb}

The complex $\Delta_\tau(A)$ has pure dimension $|\mathbf R|-1$, which can exceed $\operatorname{rank}C_A-1$ when $A$ is $\tau$-tilting infinite. Puncture coordinates and a nerve argument also occur in the calculation of the full tagged arc complex in \cite[Theorem~10.2]{FST}.  Theorem~B concerns the complex of compatible $\mathbf R$-standard arcs and gives its homotopy type in terms of $\operatorname{rank}C_A$.

The proof of Theorem~B uses Asai's purely non-rigid region $R_0(A)$. For a tame weakly symmetric algebra, we prove that $R_0(A)=\ker C_A$ whenever the vectors $C_Ag(X)$, with $X$ indecomposable two-term presilting, form a bounded set (Theorem~\ref{top:kernel}).  The identity $Bg_{\mathbf R}(\gamma)=\eta(\gamma)$ in Proposition~\ref{prop:endpoint-formula} gives the required bound for the algebras in Theorem~B. A convex-cover construction then gives
\[
|\Delta_\tau(A)|\simeq
\mathbb R^{|\mathbf R|}\setminus\ker C_A
\simeq S^{\operatorname{rank}C_A-1}.
\]
The same construction, together with Asai's description of $R_0(A)$, gives the sphere homotopy type for Brauer graph algebras with arbitrary positive vertex multiplicities (Corollary~\ref{top:brauer}).

For connected non-local $A$, the connectedness of nonzero completion graphs implies that the exchange graph and $|\Delta_\tau(A)|$ have the same connected components (Lemma~\ref{top:exchange-components}). When $\mathbf R$ is punctured, Theorem~B therefore gives exactly two components when $\operatorname{rank}C_A=1$, and one when $\operatorname{rank}C_A\geq2$. Together with the boundary case in \cite[Theorem~5.16]{HZZ}, this proves the classification in Theorem~A.

For punctured $\mathbf R$, the Cartan matrix also determines $\tau$-tilting finiteness. The finite set of $g$-vectors depends on more than this matrix. For example, a disk triangle and an essential three-cycle in an annulus have the same signed endpoint matrix, but give root systems of types $D_3$ and $C_3$, respectively. Theorem~C describes the surface configurations that distinguish the root systems. Let $\Gamma_{\mathbf R}$ be the endpoint multigraph, with incident punctures as vertices and arcs as edges, retaining loops and parallel edges.

\begin{mainthmc}[{Theorems~\ref{res:4.7}, \ref{res:7.3} and~\ref{res:7.5}}] Suppose $\mathbf R\neq\varnothing$ is punctured. \begin{enumerate} \item $A$ is $\tau$-tilting finite if and only if $C_A$ is positive definite. \item In the finite case, every connected component $\Gamma$ of $\Gamma_{\mathbf R}$ is a tree or is unicyclic. Suppose $\Gamma$ has $n$ edges. Under $B$, the primitive $g$-vector ray generators supported on $\Gamma$ form a root system of type $A_n$ if $\Gamma$ is a tree. If $\Gamma$ is unicyclic, the type is $D_n$ when the unique cycle of $\Gamma$ is a conjugate pair or consists of three distinct underlying arcs bounding an unpunctured disk, and $C_n$ otherwise. \end{enumerate} \end{mainthmc}

The roots in Theorem~C are primitive in the image lattice $B\mathbb Z^{|\mathbf R|}$. Theorem~C concerns the ray generators. Tagged compatibility determines which rays span cones, so the root system alone does not determine the $g$-fan or the convexity of the canonical $g$-polytope. For Brauer graph algebras, the tree and odd-cycle finiteness criterion was proved in \cite[Theorem~6.7]{AAC}, and the type $A$ and $C$ ray sets were determined in \cite[Theorem~8.4 and Proposition~8.16]{AHIKM}. The surface criterion in Theorem~C also treats conjugate pairs and unpunctured disk triangles with arbitrary attached trees. Type-$D$ trivial extensions are described in \cite[Section~1.1]{MH}; here the ray generators are determined directly from the original partial tagged triangulation.

This paper is organized as follows. Section~\ref{sec:geometric-model} recalls the geometric model. Section~\ref{sec:2} realizes reduction by tagged cutting and proves connectedness above nonzero pairs. Section~\ref{sec:4} establishes the endpoint formulas and the finiteness criterion. Section~\ref{sec:topology} determines the homotopy type of the $\tau$-tilting complex. Section~\ref{sec:5} classifies exchange and completion components, and Section~\ref{sec:7} determines the finite endpoint ray sets. Appendix~\ref{sec:cutting-index-projections} identifies the maps on links, index lattices and $g$-fans induced by tagged cutting. Appendix~\ref{sec:appendix-algebra-presentations} proves the tameness needed for Theorem~B and verifies the algebra presentations used in Section~\ref{sec:7}.

\section*{AI Declaration}
This work builds on an earlier manuscript on reduction and handwritten
notes containing some results on the number of connected components of
the exchange graph. These materials were prepared without AI assistance.

In developing the present paper, the authors used ChatGPT (GPT-6 Astra)
to assist with a unified treatment in terms of Cartan matrices and
$g$-vectors, extensions of the earlier results, and the development of
proof strategies and arguments, including those concerning the sphere
homotopy type of the support $\tau$-tilting complex. The tool also assisted
with checking arguments for possible gaps and inconsistencies, literature
searches, the organisation of the manuscript, drafting and revising the
text, and LaTeX editing. The authors take full responsibility for the
final manuscript, including the correctness of the results and proofs
and the accuracy of the references.


\section*{Conventions}

Throughout, $k$ is an algebraically closed field. All additive categories are
$k$-linear, Hom-finite, Krull--Schmidt and idempotent complete. For an object
$T$, write $\operatorname{add}T$ for the full subcategory of direct summands of
finite direct sums of copies of $T$, and $|T|$ for the number of isomorphism
classes of indecomposable summands of $T$. For a basic algebra $A$, the number
$|A|$ is also the number of isomorphism classes of simple $A$-modules. All
algebras are finite-dimensional unless stated otherwise.
The term tame includes representation-finite algebras. For an algebra $A$, we write
$\operatorname{mod}A$ and $\operatorname{proj}A$ for the categories of finitely
generated right $A$-modules and finitely generated projective right
$A$-modules, respectively. We also write $D=\operatorname{Hom}_k(-,k)$.
The Auslander--Reiten translation is denoted by $\tau$. In a composition $fg$, the map $g$ is applied first.

Exchange and completion graphs are undirected. For a graph $G$, let $\pi_0G$
be the set of connected components. For graphs $G_1$ and $G_2$, write
$G_1\square G_2$ for their Cartesian product. We use bold letters for
collections of arcs and the corresponding unbolded letters for their objects
in the cluster category; thus
$R=X(\mathbf R)=\bigoplus_{\gamma\in\mathbf R}X(\gamma)$.
The marked surface, boundary marked points and
punctures are denoted by $\mathbf S$, $\mathbf M$ and $\mathbf P$.

\section{\texorpdfstring{Preliminaries}{Preliminaries}}\label{sec:geometric-model}

\subsection{$\tau$-tilting pairs and completion graphs}
Let $A$ be a finite-dimensional basic $k$-algebra. We use the pair notation
of \cite{AIR}.

\begin{definition}\label{res:1.1}
Let $M\in\operatorname{mod}A$ and $P\in\operatorname{proj}A$ be basic. The
pair $(M,P)$ is called a \emph{$\tau$-rigid pair} if
\[
 \operatorname{Hom}_A(M,\tau M)=0
 \qquad\text{and}\qquad \operatorname{Hom}_A(P,M)=0.
\]
Denote by $|(M,P)|=|M|+|P|$ the number of indecomposable summands of the pair.
A $\tau$-rigid pair $(M,P)$ is called \emph{$\tau$-tilting} if $|(M,P)|=|A|$, and
\emph{almost complete} if $|(M,P)|=|A|-1$. Direct sums and direct summands of pairs
are taken componentwise.
\end{definition}

Denote by $\tau\text{-}\operatorname{rigidp}A$ and
$s\tau\text{-}\operatorname{tilt}A$ the sets of isomorphism classes of basic
$\tau$-rigid pairs and basic $\tau$-tilting pairs, respectively. The algebra
$A$ is \emph{$\tau$-tilting finite} if $s\tau\text{-}\operatorname{tilt}A$ is finite.

Every almost complete $\tau$-rigid pair has exactly two basic $\tau$-tilting completions,
which are mutations of one another, see
\cite[Theorem~2.18 and Definition~2.19]{AIR}. Let
$\operatorname{EG}(s\tau\text{-}\operatorname{tilt}A)$ be the graph whose
vertices are the isomorphism classes of basic $\tau$-tilting pairs and whose
edges join pairs sharing an almost complete direct summand.

\begin{definition}\label{res:1.2}
For a basic $\tau$-rigid pair $u$, the \emph{completion graph} $\mathcal F_A(u)$ is
the induced subgraph of
$\operatorname{EG}(s\tau\text{-}\operatorname{tilt}A)$ on the pairs containing
$u$ as a direct summand.
\end{definition}

Fix representatives $P_1,\ldots,P_n$ of the isomorphism classes of
indecomposable projective $A$-modules. We write $C_A=(c_{ij})_{i,j}$ for
the \emph{Cartan matrix} of $A$, where
\[
 c_{ij}=\dim_k\operatorname{Hom}_A(P_i,P_j).
\]
Let $X=\bigoplus_iX_i$ be a basic object in an additive category $\mathcal C$,
with indecomposable summands $X_i$, and  let
$A=\operatorname{End}_{\mathcal C}(X)$. The projective $A$-modules
$P_i=\operatorname{Hom}_{\mathcal C}(X,X_i)\cong e_iA$ satisfy
\[
 \operatorname{Hom}_A(P_i,P_j)
 \cong e_jAe_i\cong\operatorname{Hom}_{\mathcal C}(X_i,X_j).
\]
The \emph{Nakayama functor} on
$\operatorname{proj}A$ is $\nu_A=D\operatorname{Hom}_A(-,A)$. The algebra
$A$ is called \emph{weakly symmetric} if $A$ is self-injective and the Nakayama
permutation is the identity, that is, $\nu_A(P_i)\cong P_i$ for every
indecomposable projective $P_i$.

\subsection{Marked surfaces and tagged arcs}
We recall the surface conventions of \cite[Section~3.1]{HZZ}. A \emph{punctured
marked surface} is a triple $\mathbf S=(S,\mathbf M,\mathbf P)$, where $S$ is a
connected, compact oriented surface with nonempty boundary, and
$\mathbf M\subseteq\partial S$ and $\mathbf P\subseteq S\setminus\partial S$ are
finite sets of points. Every boundary component contains at least one point of $\mathbf M$. Let
$S^\circ=S\setminus(\partial S\cup\mathbf P)$. The closures of the
components of $\partial S\setminus\mathbf M$ are called \emph{boundary segments}.
Denote their set by $\mathbf B$. As in \cite[Definition~2.1]{FST}, we exclude
unpunctured disks with one, two or three boundary marked points and
once-punctured disks with one boundary marked point.

An \emph{ordinary arc} is an immersion $\gamma:[0,1]\longrightarrow S$ with
endpoints in $\mathbf M\cup\mathbf P$ and interior in $S^\circ$ such that
$\gamma$ has no self-intersections in $S^\circ$ and is neither null-homotopic nor
homotopic to a boundary segment. Arcs are considered up to homotopy relative
to their endpoints, with the other marked points and punctures fixed. An \emph{ideal
triangulation} is a maximal collection of ordinary arcs with no interior
intersections. A triangle is called \emph{self-folded} if two sides coincide. The
\emph{folded side} of a self-folded triangle joins the base to the puncture enclosed
by the \emph{non-folded side}. Following \cite[Definition~3.2]{HZZ}, a \emph{partial ideal
triangulation} is a subset of an ideal triangulation that contains the folded
side whenever it contains the corresponding non-folded side.

\begin{definition}[{\cite[Definition~7.1]{FST}}]\label{res:1.3}
A \emph{tagged arc} is an ordinary arc $\gamma$ that does not cut out a
once-punctured monogon, equipped with a map
\[
 \kappa_\gamma:\{t\in\{0,1\}:\gamma(t)\in\mathbf P\}
 \longrightarrow\{1,-1\}
\]
such that $\kappa_\gamma(0)=\kappa_\gamma(1)$ whenever
$\gamma(0)=\gamma(1)\in\mathbf P$. The value $\kappa_\gamma(t)$ is called
the \emph{tagging} of $\gamma$ at $\gamma(t)$. The tagging is \emph{plain} if
$\kappa_\gamma(t)=1$ and \emph{notched} if $\kappa_\gamma(t)=-1$.
In figures, a plain end is unmarked and a notched end
carries a cross $\times$. Denote by $\mathcal A^\times(\mathbf S)$ the set
of tagged arcs. We also use $\gamma$ to denote the underlying ordinary arc
when the taggings are not involved.
\end{definition}

The \emph{tagged rotation} $\rho(\gamma)$ is obtained from $\gamma$ by moving each
marked endpoint anticlockwise along the boundary to the next marked point and
reversing the taggings at puncture ends, see \cite[Definition~3.4]{BQ}.
If both endpoints of $\gamma$ are punctures, then $\rho(\gamma)$ has the same
underlying arc as $\gamma$.
In particular, $\rho^2(\gamma)=\gamma$ and
$\rho(\gamma)=\rho^{-1}(\gamma)$.

For ordinary arcs $\gamma,\eta$ in transverse minimal position, an
\emph{interior intersection} is a pair $(s,t)\in(0,1)^2$ with
$\gamma(s)=\eta(t)$.
For tagged arcs $\gamma,\eta$ in minimal position, a pair
$(s,t)\in\{0,1\}^2$ with $\gamma(s)=\eta(t)\in\mathbf P$ is called a
\emph{tagged intersection} if the following conditions hold, see
\cite[Definition~3.3]{QZ}.
\begin{enumerate}
 \item $\kappa_\gamma(s)\neq\kappa_\eta(t)$.
 \item If the underlying arcs, oriented from $s$ to $1-s$ and from $t$ to
       $1-t$, respectively, are homotopic, then
       \[
        \gamma(1-s)=\eta(1-t)\in\mathbf P
        \qquad\text{and}\qquad \kappa_\gamma(1-s)\neq\kappa_\eta(1-t).
       \]
\end{enumerate}
Denote by $\cap(\gamma,\eta)$ the set of interior intersections and tagged
intersections, and by
$\operatorname{Int}(\gamma,\eta)=|\cap(\gamma,\eta)|$ their \emph{intersection
number}. Two tagged arcs $\gamma,\eta$ are called \emph{compatible} if
$\operatorname{Int}(\gamma,\eta)=0$. Two distinct compatible tagged arcs
with the same underlying arc are called a \emph{conjugate pair}. Their underlying
arc has distinct endpoints, and their taggings differ at exactly one puncture
end, see the third picture of Figure~\ref{fig:1}.
A \emph{tagged triangulation} is a maximal compatible collection of tagged arcs.
A \emph{partial tagged triangulation} is a subset of a tagged triangulation (see
\cite[Definition~3.6]{HZZ}).  A partial tagged triangulation is called \emph{punctured} if all endpoints of the arcs lie in $\mathbf P$.

\begin{remark}\label{rem:conjugate-endpoints}
Let $\{\alpha,\beta\}$ be a conjugate pair joining punctures $p,q$, with
common tagging $a$ at $p$ and opposite taggings at $q$.
If a third tagged arc $\gamma$ is compatible
with both $\alpha$ and $\beta$, then $\gamma$ has no endpoint at $q$.
If $\gamma$ meets the pair, then $\gamma$ meets both arcs at $p$ with tagging $a$.
Indeed, no third tagging of the underlying arc is compatible with both
$\alpha$ and $\beta$. For a different underlying arc, compatibility at
$q$ would require agreement with two opposite taggings. Compatibility at
$p$ requires tagging $a$.
\end{remark}

\begin{example}[{\cite[Example~3.4]{QZ}}]\label{res:1.4}
A loop $\gamma$ based at a puncture has
$\operatorname{Int}(\gamma,\rho(\gamma))=4$, since all four endpoint pairs
are tagged intersections, see the first picture of Figure~\ref{fig:1}.
An arc $\gamma$ joining distinct punctures has
$\operatorname{Int}(\gamma,\rho(\gamma))=2$ (see the second picture), and a
conjugate pair $\gamma,\eta$ has $\operatorname{Int}(\gamma,\eta)=0$ (see the
third picture). The endpoint marks
$?,!,\sharp$ denote tagging choices, with
$\{?,!\}=\{\varnothing,\times\}$ and $\sharp\in\{\varnothing,\times\}$.
Choices at distinct punctures are independent.
\end{example}

\begin{figure}[htbp]
 \centering
 \includegraphics[width=.75\linewidth]{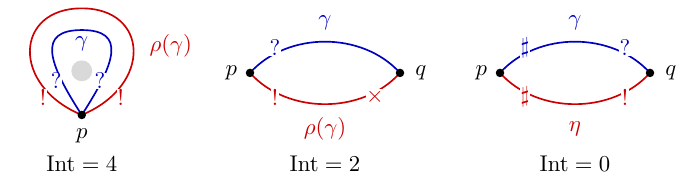}
 \caption{Endpoint intersections.}
 \label{fig:1}
\end{figure}

A \emph{connected component} of a partial tagged triangulation is a maximal
subcollection whose arcs are joined by chains of common endpoints. Such a
component is called a \emph{puncture component} if all endpoints belong to
$\mathbf P$. A partial tagged triangulation is called \emph{connected to the
boundary} if each connected component contains an arc incident with $\mathbf M$, see
\cite[Definition~5.12]{HZZ}.

Let $\mathbf U$ be a partial tagged triangulation, and denote by
$\mathbf P_{\mathbf U}\subseteq\mathbf P$ the set of punctures incident
with the arcs of $\mathbf U$. The \emph{weak signature} of $\mathbf U$ is the map
$\kappa_{\mathbf U}:\mathbf P_{\mathbf U}\longrightarrow\{-1,0,1\}$
whose value is $1$ when all incident taggings are plain, $-1$ when all are
notched, and $0$ when both taggings occur.

Denote by
$\mathbf U^\circ$ the associated partial ideal triangulation, obtained by
forgetting the taggings at punctures of nonzero weak signature and replacing
the two tagged arcs at each zero-signature puncture $p$ by a self-folded
triangle enclosing $p$, see \cite[Construction~3.8]{HZZ}.

For any tagged arc $\delta$, denote by $\delta^{\mathbf U}$ the tagged arc
obtained by reversing the taggings of $\delta$ at the punctures
$p\in\mathbf P_{\mathbf U}$ with $\kappa_{\mathbf U}(p)=-1$, see
\cite[Construction~3.9]{HZZ}. The arcs $\delta^{\mathbf U}$ and $\delta$
have the same underlying arc.

\subsection{Cluster categories and the arc--object correspondence}
The cluster category $\mathcal C(\mathbf S)$ is the generalized cluster
category of a surface quiver with potential whose Jacobian algebra is
finite-dimensional, see \cite{A,BZ,QZ} for the construction and geometric
models. We write $\mathcal C$ when the surface is fixed. The category
$\mathcal C(\mathbf S)$ is a Hom-finite $2$-Calabi--Yau triangulated category,
that is, there are functorial isomorphisms
\[
 D\operatorname{Hom}_{\mathcal C}(Y,Z)
 \cong\operatorname{Hom}_{\mathcal C}(Z,Y[2]),\quad
 \text{for all }Y,Z\in\mathcal C,
\]
where $[1]$ denotes the suspension functor.
An object $Y$ is called \emph{rigid} if $\operatorname{Hom}(Y,Y[1])=0$. A rigid
object $T$ is called \emph{cluster tilting} if
\[
 \operatorname{add}T=\{Y:\operatorname{Hom}(T,Y[1])=0\}.
\]
For a full additive subcategory $\mathcal E\subseteq\mathcal C(\mathbf S)$,
a \emph{rigid object in $\mathcal E$} means an object of $\mathcal E$ that is
rigid in $\mathcal C(\mathbf S)$. Such an object $Y$ is called
\emph{maximal rigid in $\mathcal E$} if every $Z\in\mathcal E$ for which
$Y\oplus Z$ is rigid belongs to $\operatorname{add}Y$. Following \cite[Definition~2.3]{HZZ}, denote by
$\operatorname{rigid}\text{-}\mathcal E$ and
$\operatorname{max\,rigid}\text{-}\mathcal E$ the sets of isomorphism
classes of basic rigid objects and basic maximal rigid objects in $\mathcal E$,
respectively. Denote by $\operatorname{ind\,rigid}\mathcal E$ the set of
isomorphism classes of indecomposable rigid objects in $\mathcal E$.

\begin{theorem}[{\cite[Theorems~5.2, 5.5 and Corollary~5.6]{QZ}}]\label{res:1.5}
There is a bijection
\[
 X:\mathcal A^\times(\mathbf S)\longrightarrow
 \operatorname{ind\,rigid}\mathcal C
\]
satisfying
\[
 X(\rho\gamma)\cong X(\gamma)[1]
 \qquad\text{and}\qquad
 \dim_k\operatorname{Hom}(X(\alpha),X(\beta)[1])
 =\operatorname{Int}(\alpha,\beta).
\]
The bijection $X$ extends to a bijection between partial tagged triangulations
$\mathbf R$ and isomorphism classes of basic rigid objects by
$X(\mathbf R)=\bigoplus_{\gamma\in\mathbf R}X(\gamma)$. Under $X$,
tagged triangulations correspond precisely to basic cluster tilting
objects.
\end{theorem}

\begin{definition}[{\cite[Definition~4.1]{HZZ}}]\label{res:1.6}
The \emph{surface rigid algebra} associated with a partial tagged triangulation
$\mathbf R$ is the basic algebra
\[
 \Lambda_{\mathbf R}=\operatorname{End}_{\mathcal C(\mathbf S)}X(\mathbf R).
\]
\end{definition}

\subsection{Standard arcs and dissections}
The geometric model for $\Lambda_{\mathbf R}$ restricts the arc--object
correspondence to $\mathbf R$-standard arcs. We recall their laminate description from
\cite[Definition~3.13 and Proposition~4.4]{HZZ}.

\begin{definition}\label{res:1.7}
The \emph{elementary laminate} $e(\delta)$ runs close to a tagged arc $\delta$. At a
boundary end, the endpoint of $e(\delta)$ is displaced slightly clockwise along
the boundary. At a puncture end, $e(\delta)$ spirals clockwise for a plain
tagging and anticlockwise for a notched tagging. The \emph{co-elementary laminate}
$e^{\operatorname{op}}(\delta)$ uses the opposite boundary displacements and
spiral directions.
\end{definition}

Let $\mathbf R$ be a partial tagged triangulation. For $L=e(\delta)$, let
$L^{\mathbf R}=e(\delta^{\mathbf R})$, and for
$L=e^{\operatorname{op}}(\delta)$, let
$L^{\mathbf R}=e^{\operatorname{op}}(\delta^{\mathbf R})$. Place
$L^{\mathbf R}$ in minimal position with $\mathbf R^\circ$. The arcs of
$\mathbf R^\circ$ cut $L^{\mathbf R}$ into segments. A segment $\eta$ is
said to \emph{cut out an angle} if, in the surface cut along $\mathbf R^\circ$, it
bounds a triangular disk together with parts of two sides meeting at a corner,
as in Figure~\ref{fig:2}. The interior of the triangular disk contains no marked point
or puncture. The two sides are occurrences of arcs in
$\mathbf R^\circ\cup\mathbf B$ and may come from the same arc.

\begin{figure}[htbp]
 \centering
 \includegraphics[width=100pt]{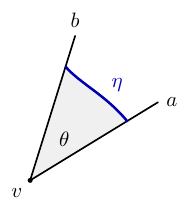}
 \caption{A segment $\eta$ cutting out the angle $\theta$ between $a$ and $b$.}
 \label{fig:2}
\end{figure}

\begin{definition}[{\cite[Definition~3.13 and Proposition~4.4]{HZZ}}]\label{res:1.8}
The laminate $L$ \emph{shears} $\mathbf R$ if every segment of $L^{\mathbf R}$ cut
by $\mathbf R^\circ$ cuts out an angle. A tagged arc $\delta$ is
\emph{$\mathbf R$-standard} if $e(\delta)$ shears $\mathbf R$.
An \emph{$\mathbf R$-dissection} is a maximal compatible collection of
$\mathbf R$-standard tagged arcs, see \cite[Definition~4.6]{HZZ}.
Write $\mathcal A^\times_{\mathbf R\text{-st}}$ for the set of standard arcs
and $\mathcal D(\mathbf R)$ for the set of dissections.
Replacing $e$ by $e^{\operatorname{op}}$ gives the notions of
\emph{$\mathbf R$-co-standard arcs} and \emph{$\mathbf R$-co-dissections}.
\end{definition}

Compatible collections of $\mathbf R$-standard arcs correspond to
$\tau$-rigid pairs over $\Lambda_{\mathbf R}$. We recall this correspondence
through extension subcategories of the cluster category.
For full subcategories $\mathcal X,\mathcal Y$ of a triangulated category,
$\mathcal X*\mathcal Y$ consists of the objects $W$ occurring in a triangle
$X\longrightarrow W\longrightarrow Y\longrightarrow X[1]$ with
$X\in\mathcal X$ and $Y\in\mathcal Y$.
The algebraic correspondence $\Phi$ below is recalled from
\cite[Theorem~3.2]{CZZ}.

\begin{theorem}[{\cite[Theorem~4.7 and Corollary~4.8]{HZZ}}]\label{res:1.9}
Let $R=X(\mathbf R)$, $A=\Lambda_{\mathbf R}$,
$\mathcal E_R=\operatorname{add}R*\operatorname{add}R[1]$, and
$H=\operatorname{Hom}_{\mathcal C}(R,-)$. The correspondence $X$ restricts
to a bijection
\[
 \mathcal A^\times_{\mathbf R\text{-st}}
 \longleftrightarrow\operatorname{ind\,rigid}\mathcal E_R.
\]
Compatible collections correspond to basic rigid objects by taking direct sums,
and dissections correspond to maximal rigid objects in $\mathcal E_R$.
There is a bijection
\[
 \Phi:\operatorname{rigid}\text{-}\mathcal E_R
 \longrightarrow\tau\text{-}\operatorname{rigidp}A,
 \qquad W=W_0\oplus W_1\longmapsto
 \bigl(H(W_0),H(W_1[-1])\bigr),
\]
where $W_1$ is the maximal summand of $W$ in $\operatorname{add}R[1]$.
The composite $\Phi\circ X$ identifies $\mathcal D(\mathbf R)$ with
$s\tau\text{-}\operatorname{tilt}A$. A compatible collection of
$\mathbf R$-standard arcs is a dissection exactly when it has $|\mathbf R|$
arcs. Co-standard arcs and co-dissections correspond to indecomposable and
maximal rigid objects in $\mathcal E_R[-1]$, respectively, with the same
compatibility and cardinality statements.
\end{theorem}

\subsection{Indices and $g$-vectors}
An object $T$ of $K^b(\operatorname{proj}A)$ is called \emph{presilting} if
$\operatorname{Hom}(T,T[i])=0$ for all $i>0$, and \emph{silting} if, in addition,
the smallest thick subcategory containing $T$ is $K^b(\operatorname{proj}A)$.
A silting object $T$ is called \emph{tilting} if
$\operatorname{Hom}(T,T[i])=0$ also for all $i<0$.
The object $T$ is called \emph{two-term} if $T$ is isomorphic to a complex
concentrated in degrees $-1$ and $0$.
Denote by $\twopsilt A$ and $\indtwopsilt A$ the sets of isomorphism
classes of basic and indecomposable two-term presilting objects in
$K^b(\operatorname{proj}A)$, respectively.

By \cite{AIR}, basic two-term presilting complexes correspond to basic
$\tau$-rigid pairs, and basic two-term silting complexes correspond to basic
$\tau$-tilting pairs.

Let $K_A=K_0^{\mathrm{sp}}(\operatorname{proj}A)$ be the split Grothendieck group, and  let
$K_{A,\mathbb R}=K_A\otimes_{\mathbb Z}\mathbb R$. We identify
$K_A$ with the Grothendieck group $K_0(K^b(\operatorname{proj}A))$ by sending the class of a
projective module to the class of the corresponding stalk complex.
The classes $[P_1],\ldots,[P_n]$ form a basis of $K_A$, and a two-term
complex $(P^{-1}\longrightarrow P^0)$ has class $[P^0]-[P^{-1}]$.
For a basic two-term silting complex $T=T_1\oplus\cdots\oplus T_n$ with
indecomposable summands $T_i$, the \emph{$g$-matrix} $G_T$ has as columns
the coordinate vectors of $[T_i]$ in this basis.
For a $\tau$-rigid pair $u=(M,P)$, choose a minimal projective presentation
$P_M^{-1}\longrightarrow P_M^0\longrightarrow M\longrightarrow0$. The
\emph{$g$-vector} of $u$ is defined to be
\[
 g_A(u)=[P_M^0]-[P_M^{-1}]-[P]\in K_A,
\]
see \cite[Section~5.1]{AIR}. For a basic $\tau$-rigid pair
$u=u_1\oplus\cdots\oplus u_s$ with indecomposable summands $u_i$, define
the cone associated with $u$ and the \emph{$g$-fan} of $A$ to be
\[
 \sigma_u=\operatorname{cone}\{g_A(u_1),\ldots,g_A(u_s)\}
 \qquad\text{and}\qquad
 \Sigma_g(A)=\{\sigma_v:v\in\tau\text{-}\operatorname{rigidp}A\},
\]
respectively. The collection $\Sigma_g(A)$ is a nonsingular fan in
$K_{A,\mathbb R}$. The primitive ray generators of $\Sigma_g(A)$ are precisely
the $g$-vectors of indecomposable $\tau$-rigid pairs, see \cite[Section~5]{AIR}
and \cite[Definition~4.1 and Proposition~4.2]{AHIKM}.

\begin{definition}\label{def:tau-complex}
The \emph{support $\tau$-tilting complex} $\Delta_\tau(A)$ has the
isomorphism classes of indecomposable $\tau$-rigid pairs as vertices.
A finite set forms a simplex if the direct sum of representative pairs
is $\tau$-rigid. The complex $\Delta_{\mathbf R}$ has the
$\mathbf R$-standard arcs as vertices and compatible collections as
simplices. We give their geometric realizations the weak topology of
simplicial complexes.
\end{definition}
Theorem~\ref{res:1.9} identifies $\Delta_{\mathbf R}$ with
$\Delta_\tau(A)$. Under the correspondence of \cite[Theorem~3.2 and
Lemmas~3.4--3.5]{AIR}, the vertices of $\Delta_\tau(A)$ are also the
elements of $\indtwopsilt A$, and the simplices of $\Delta_\tau(A)$ correspond to
 the elements of $\twopsilt A$ by taking direct sums.
Every simplex extends to a maximal simplex with $|A|$ vertices by
\cite[Proposition~3.3]{AIR}. Thus $\Delta_\tau(A)$ is pure of dimension
$|A|-1$ when $A\ne0$.

\Needspace*{8\baselineskip}
We express the same $g$-vectors in the cluster category by means of relative
indices. Let $W\in\mathcal E_R$, and choose an $R$-presentation
$R^1\longrightarrow R^0\longrightarrow W\longrightarrow R^1[1]$ with
$R^0,R^1\in\operatorname{add}R$. The \emph{relative index} of $W$ is
\[
 \operatorname{ind}_R W=[R^0]-[R^1]\in K_A,
\]
where we identify $K_0^{\mathrm{sp}}(\operatorname{add}R)$ with $K_A$ via
$\operatorname{Hom}(R,-)$. Removing split summands gives the presentation
associated with a minimal right $\operatorname{add}R$-approximation of $W$.
Uniqueness of the minimal approximation makes the index independent of the
presentation. When $W$ is rigid, $\operatorname{ind}_R W$ is
the $g$-vector of the $\tau$-rigid pair $\Phi(W)$ from
Theorem~\ref{res:1.9}, see \cite[Section~5.1]{AIR}. For an
$\mathbf R$-standard tagged arc $\gamma$, write
$g_{\mathbf R}(\gamma)=\operatorname{ind}_R X(\gamma)$.
Theorem~\ref{res:1.9} therefore identifies $\mathbf R$-standard arcs with rays
and compatible collections of $\mathbf R$-standard arcs with cones of $\Sigma_g(A)$.
Under this identification, the exchange graph is the
\emph{chamber graph} of $\Sigma_g(A)$:  the vertices are the
full-dimensional cones, and the edges join cones sharing a facet.
Write $|\Sigma_g(A)|=\bigcup_{\sigma\in\Sigma_g(A)}\sigma$.
The fan $\Sigma_g(A)$ is \emph{complete} if
$|\Sigma_g(A)|=K_{A,\mathbb R}$. By \cite[Proposition~4.2]{AHIKM}, the fan is
complete if and only if $A$ is $\tau$-tilting finite.


\Needspace*{5\baselineskip}
\section{Cutting and connected completions}
\label{sec:geometric-reduction}\label{sec:2}

Let $\mathbf R$ be a partial tagged triangulation, and  let
$R=X(\mathbf R)$ and $A=\operatorname{End}_{\mathcal C}(R)$.
Fix a basic $\tau$-rigid pair $u=(M,P)$ over $A$.
Write $U=\Phi^{-1}(u)$ and $\mathbf N_u=X^{-1}(U)$.
The Bongartz completion of $U$ gives an $\mathbf R$-dissection containing $\mathbf N_u$.
Cutting along $\mathbf N_u$ leaves a partial tagged triangulation whose
surface rigid algebra is the reduction algebra at $u$.
When $A$ is connected and $u\neq(0,0)$, every component of this cut model
meets the boundary, so \cite[Theorem~5.16]{HZZ} gives connectedness of
the completion graph $\mathcal F_A(u)$.

\subsection{Bongartz completion and the reduction algebra}
 Let $\mathcal E_R=\operatorname{add}R*\operatorname{add}R[1]$.
For an additive subcategory $\mathcal W$, write $[\mathcal W]$ for the
ideal of morphisms factoring through $\mathcal W$.
The category $\mathcal C/[\mathcal W]$ is the corresponding additive quotient.
By \cite[Proposition~6.2(3)]{IY}, the functor $\operatorname{Hom}(R,-)$
induces an equivalence
\begin{equation}\label{eq:2.1}
 \mathcal E_R/[\operatorname{add}R[1]]
 \overset{\sim}{\longrightarrow}\operatorname{mod}A.
\end{equation}

We order $\tau$-tilting pairs by
$(M,P)\leq(N,Q)$ if $\operatorname{Fac}M\subseteq\operatorname{Fac}N$,
where $\operatorname{Fac}M$ consists of the factor modules of finite direct
sums of $M$.
The \emph{Bongartz completion} is the greatest $\tau$-tilting pair containing $u$
under the order defined above. The following approximation triangle
constructs the corresponding rigid object.

\begin{proposition}[{\cite[Section~3.2]{HZZrel}}]
\label{res:2.1}
Let $f\in\operatorname{Hom}(U',R[1])$ be a right
$\operatorname{add}U$-approximation of $R[1]$, and complete it to a triangle
\begin{equation}\label{eq:2.2}
 R\overset{a}{\longrightarrow}V\overset{b}{\longrightarrow}U'\overset{f}{\longrightarrow}R[1].
\end{equation}
Let $B_U$ be the basic object determined by
$\operatorname{add}B_U=\operatorname{add}(U\oplus V)$.
Then $B_U$ is maximal rigid in $\mathcal E_R$, and $\Phi(B_U)$ is the
Bongartz completion of $u$. More precisely, $\Phi(B_U)=(T,P)$, where
$T=\operatorname{Hom}(R,B_U)$.
\end{proposition}

\begin{proof}
The rigidity of $R$ and associativity of extension products give
\[
 V\in\operatorname{add}R*\operatorname{add}U
 \subseteq\operatorname{add}R*
   (\operatorname{add}R*\operatorname{add}R[1])=\mathcal E_R.
\]
Applying $\operatorname{Hom}(U,-)$ to \eqref{eq:2.2} and using the
approximation property and rigidity of $U$ gives
$\operatorname{Hom}(U,V[1])=0$.
The $2$-Calabi--Yau property then gives $\operatorname{Hom}(V,U[1])=0$.

Let $t\in\operatorname{Hom}(V,V[1])$.
Since $\operatorname{Hom}(V,U'[1])=0$, the triangle gives
$t=a[1]c$ for some $c\in\operatorname{Hom}(V,R[1])$.
The equality $\operatorname{Hom}(R,R[1])=0$ then gives
$c=db$ for some $d\in\operatorname{Hom}(U',R[1])$.
The approximation property gives $d=fh$ with
$h\in\operatorname{End}(U')$. Hence
\[
 t=a[1]fhb=0.
\]
Thus $U\oplus V$, and therefore $B_U$, is rigid.

Rotating \eqref{eq:2.2} gives
$R\in\operatorname{add}B_U[-1]*\operatorname{add}B_U$.
The rank bound in \cite[Proposition~2.7]{HZZ}, applied first with respect to
$R$ and then with respect to $B_U[-1]$, gives
\[
 |B_U|\leq|R|\leq|B_U[-1]|=|B_U|.
\]
The equality criterion in the same proposition shows that $B_U$ is maximal
rigid in $\mathcal E_R$.

For $R'\in\operatorname{add}R$, applying
$\operatorname{Hom}(-,R'[1])$ to \eqref{eq:2.2} shows that every morphism
from $V$ to $R'[1]$ factors through $b$. Therefore
\begin{equation}\label{eq:2.3}
 [\operatorname{add}R[1]](B_U,B_U)
 \subseteq[\operatorname{add}U](B_U,B_U).
\end{equation}
If an indecomposable summand $Z$ of $B_U$ belongs to
$\operatorname{add}R[1]$, then \eqref{eq:2.3} shows that $1_Z$ factors
through $\operatorname{add}U$. Thus $Z$ is a summand of $U$.
By the definition of $\Phi$, the projective part of $\Phi(B_U)$ is
therefore $P$. Write $\Phi(B_U)=(T,P)$.

Let $Y$ be any maximal rigid object of $\mathcal E_R$ containing $U$.
Since $\operatorname{Hom}(U'[-1],Y)=0$, every morphism $R\to Y$
factors through $a$. Choose a $k$-basis of size $m$ for
$\operatorname{Hom}(R,Y)$ and factor each basis element through $a$.
The resulting morphism $V^m\to Y$ induces an epimorphism
$\operatorname{Hom}(R,V)^m\twoheadrightarrow\operatorname{Hom}(R,Y)$.
Consequently,
\[
 \operatorname{Fac}\operatorname{Hom}(R,Y)
 \subseteq\operatorname{Fac}\operatorname{Hom}(R,V)
 \subseteq\operatorname{Fac}T.
\]
Theorem~\ref{res:1.9} identifies the maximal rigid objects containing $U$
with the $\tau$-tilting pairs containing $u$.
The displayed inclusions show that $(T,P)$ is the greatest pair containing $u$.
\end{proof}

Denote by $L$ the complement of $U$ in $B_U$, so that $B_U=U\oplus L$.
Theorem~\ref{res:1.9} gives an $\mathbf R$-dissection
$\mathbf U_u=\mathbf N_u\sqcup\mathbf L_u$ with
$B_U=X(\mathbf U_u)$ and $L=X(\mathbf L_u)$.
Define the \emph{reduction algebra} $A_u$ by
\begin{equation}\label{eq:2.4}
 A_u=\operatorname{End}_A(T)/[\operatorname{add}M](T,T).
\end{equation}

The \emph{Iyama--Yoshino reduction} of $\mathcal C$ at $U$ is
$\mathcal C_U={}^\perp U[1]/[\operatorname{add}U]$.
Since $\operatorname{add}U$ is functorially finite and rigid, $\mathcal C_U$
is a $2$-Calabi--Yau triangulated category, see
\cite[Theorems~4.2 and~4.7]{IY}. By rigidity, $L\in{}^\perp U[1]$.
Write $\overline L$ for the image of $L$ in $\mathcal C_U$.

\begin{proposition}\label{res:2.2}
There are algebra isomorphisms
\[
 A_u\cong\operatorname{End}_{\mathcal C}(B_U)/
    [\operatorname{add}U](B_U,B_U)
 \cong\operatorname{End}_{\mathcal C_U}(\overline L).
\]
Moreover, $\tau$-tilting reduction identifies the $\tau$-tilting pairs over $A$
containing $u$ with $s\tau\text{-}\operatorname{tilt}A_u$ as a poset.
\end{proposition}

\begin{proof}
The equivalence \eqref{eq:2.1} induces a surjective homomorphism
$\psi:\operatorname{End}_{\mathcal C}(B_U)\to\operatorname{End}_A(T)$ with
kernel $[\operatorname{add}R[1]](B_U,B_U)$.
Since $\operatorname{Hom}(R,U)\cong M$, fullness of \eqref{eq:2.1} gives
\[
 \psi\bigl([\operatorname{add}U](B_U,B_U)\bigr)
 =[\operatorname{add}M](T,T).
\]
The kernel of $\psi$ is contained in $[\operatorname{add}U](B_U,B_U)$
by \eqref{eq:2.3}. Taking quotients gives the first isomorphism.
The second follows from $B_U=U\oplus L$ and $U=0$ in $\mathcal C_U$.

For the reduction of a pair with projective part $P$, choose an idempotent
$e$ with $P\cong eA$ and  let $A'=A/AeA$.
Passage to $A'$ identifies the $\tau$-tilting pairs over $A$ containing
$(M,P)$ with the support $\tau$-tilting $A'$-modules containing $M$,
preserving the $\operatorname{Fac}$ order, see
\cite[Lemma~2.1(b) and Proposition~2.3(a)]{AIR}.
Thus $T$ is the Bongartz completion of $M$ over $A'$.
Since $\operatorname{End}_{A'}(T)=\operatorname{End}_A(T)$,
\eqref{eq:2.4} is the reduction algebra over $A'$, and the poset
identification follows from \cite[Theorem~3.15]{J}.
\end{proof}

The compatibility of categorical reduction with $\tau$-tilting reduction is
treated more generally by Pan and Zhu in the setting of $0$-Auslander
extriangulated categories \cite{PZ}.

\subsection{Tagged cutting}

Cutting along the arcs $\mathbf N_u$ representing $u$ realizes $A_u$ as a
surface rigid algebra on the cut surface.

\begin{definition}[{\cite[Section~5.1]{HZZ}; cf.\ \cite[Section~5.2]{QZ}}]
\label{res:2.3}
For a tagged arc $\eta$, let
\[
 \mathcal A^\times(\mathbf S)_{\{\eta\}}
 =\{\gamma\in\mathcal A^\times(\mathbf S)\setminus\{\eta\}:
       \operatorname{Int}(\gamma,\eta)=0\}.
\]
The cut surface $\mathbf S/\{\eta\}$ is obtained from $\mathbf S$ by
cutting along $\eta$. The \emph{tagged cutting map} is the bijection
\[
 F_{\{\eta\}}:\mathcal A^\times(\mathbf S)_{\{\eta\}}
 \longrightarrow\mathcal A^\times(\mathbf S/\{\eta\}).
\]
If the underlying arcs of $\gamma$ and $\eta$ are not homotopic,
$F_{\{\eta\}}(\gamma)$ follows the same curve, with the taggings forgotten
at the puncture ends of $\eta$, which become boundary marked points.
When the underlying arcs of $\gamma$ and $\eta$ are homotopic, the images
are as in Figure~\ref{fig:3}, reproduced from \cite[Figure~10]{HZZ}.

For a partial tagged triangulation $\mathbf N$, let
\[
 \mathcal A^\times(\mathbf S)_{\mathbf N}
 =\{\gamma\in\mathcal A^\times(\mathbf S)\setminus\mathbf N:
   \operatorname{Int}(\gamma,\eta)=0\text{ for all }\eta\in\mathbf N\}.
\]
Iterating the construction gives a bijection
\[
 F_{\mathbf N}:\mathcal A^\times(\mathbf S)_{\mathbf N}
 \longrightarrow\mathcal A^\times(\mathbf S/\mathbf N)
\]
that preserves compatibility in both directions.
The same cut surface $\mathbf S/\mathbf N$ is obtained by cutting along
$\mathbf N^\circ$, see \cite[Remark~5.2]{HZZ}.
\end{definition}

\begin{figure}[tbp]
 \centering
 \includegraphics[width=.96\linewidth]{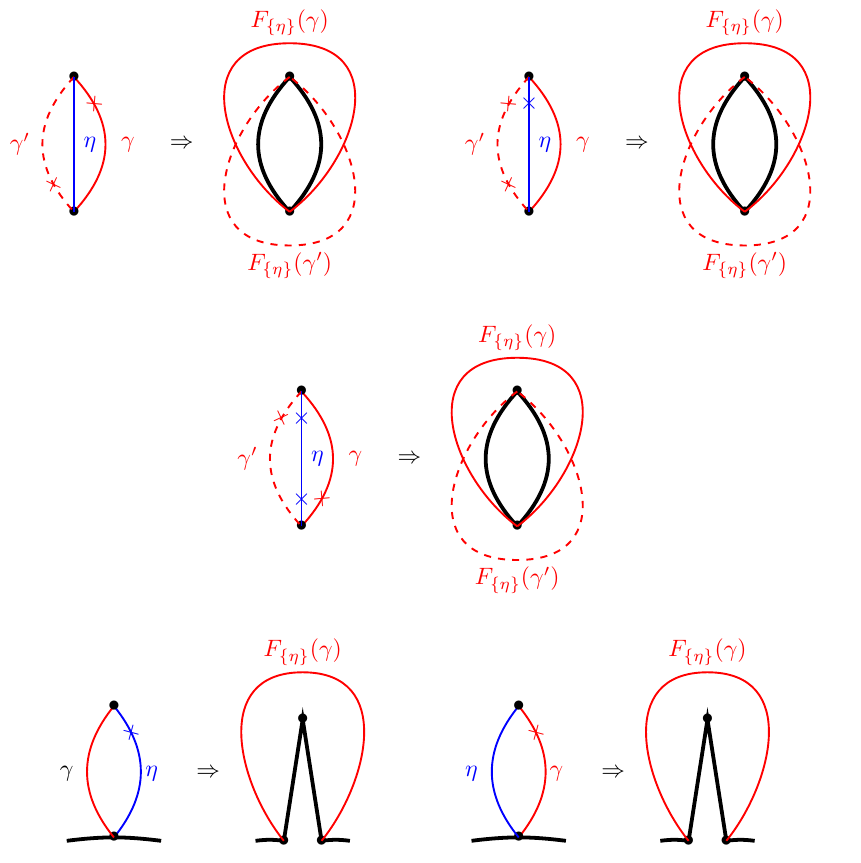}
 \caption{Tagged cutting.}
 \label{fig:3}
\end{figure}

Tagged cutting realizes Iyama--Yoshino reduction as follows.

\begin{proposition}[{\cite[Lemma~5.1]{HZZ}}]\label{res:2.4}
For every partial tagged triangulation $\mathbf N$, there is a triangle
equivalence
\[
 \xi_{\mathbf N}:\mathcal C_{X(\mathbf N)}
 \overset{\sim}{\longrightarrow}\mathcal C(\mathbf S/\mathbf N),\qquad
 \overline{X(\gamma)}\longmapsto X_{\mathbf S/\mathbf N}
                    \bigl(F_{\mathbf N}(\gamma)\bigr)
\]
for each $\gamma\in\mathcal A^\times(\mathbf S)_{\mathbf N}$.
\end{proposition}

For the Bongartz dissection $\mathbf U_u=\mathbf N_u\sqcup\mathbf L_u$,
set $\mathbf R_u=F_{\mathbf N_u}(\mathbf L_u)$.

\Needspace*{7\baselineskip}
\begin{theorem}\label{res:2.5}
The reduction algebra $A_u$ is isomorphic to the surface rigid algebra
\[
 A_u\cong\operatorname{End}_{\mathcal C(\mathbf S/\mathbf N_u)}
             X_{\mathbf S/\mathbf N_u}(\mathbf R_u)
 =\Lambda_{\mathbf R_u}.
\]
In particular, surface rigid algebras are closed under $\tau$-tilting reduction.
\end{theorem}
\begin{proof}
Proposition~\ref{res:2.2} gives
$A_u\cong\operatorname{End}_{\mathcal C_U}(\overline L)$, and
Proposition~\ref{res:2.4} sends $\overline L$ to
$X_{\mathbf S/\mathbf N_u}(\mathbf R_u)$.
Taking endomorphism algebras proves the assertion.
\end{proof}

To pass from the reduction poset to the geometric exchange graph, we use
the flip--mutation correspondence. A \emph{flip} of an $\mathbf R$-dissection replaces
one arc by the unique other $\mathbf R$-standard arc giving an
$\mathbf R$-dissection, see \cite[Proposition~5.8]{HZZ}.
Denote by $\operatorname{EG}^\times(\mathbf R)$ the graph of
$\mathbf R$-dissections and flips.
\begin{proposition}[{\cite[Proposition~5.8 and Lemma~5.10]{HZZ}}]
\label{res:2.6}
Under the correspondence $\Phi\circ X$, flips of $\mathbf R$-dissections
correspond to mutations of $\tau$-tilting pairs over $\Lambda_{\mathbf R}$.
In particular,
\[
 \operatorname{EG}^\times(\mathbf R)
 \cong\operatorname{EG}(s\tau\text{-}\operatorname{tilt}\Lambda_{\mathbf R}).
\]
\end{proposition}

Combining the flip--mutation correspondence with $\tau$-tilting reduction
gives a geometric model of the completion graph $\mathcal F_A(u)$.

\begin{corollary}\label{res:2.7}
For every basic $\tau$-rigid pair $u$,
\[
 \mathcal F_A(u)\cong
 \operatorname{EG}(s\tau\text{-}\operatorname{tilt}A_u)
 \cong\operatorname{EG}^\times(\mathbf R_u).
\]
\end{corollary}
\begin{proof}
The first isomorphism follows from the mutation compatibility of
$\tau$-tilting reduction, with the projective part included as in
Proposition~\ref{res:2.2}, see
\cite[Theorem~3.15 and Corollary~3.16]{J}.
The algebra isomorphism in Theorem~\ref{res:2.5} and the flip--mutation
correspondence in Proposition~\ref{res:2.6} give the second isomorphism.
\end{proof}

For a disconnected cut surface, categories and algebras are interpreted
componentwise. Unpunctured triangular disk components have zero cluster
category. Empty residual triangulations give the zero algebra and a
dissection graph with one vertex.


\subsection{Connected completions}\label{sec:3}
To apply \cite[Theorem~5.16]{HZZ} to the cut model, we show that cutting
along a nonempty subcollection of a dissection leaves every component
connected to the boundary when $\Lambda_{\mathbf R}$ is connected.
We first relate puncture components to direct product decompositions.
Let $\mathbf U$ be an $\mathbf R$-dissection.

\Needspace*{6\baselineskip}
\begin{lemma}\label{res:3.1}
Let $\mathbf R_1$ be a puncture component of $\mathbf R$, and let
$\mathbf R_2=\mathbf R\setminus\mathbf R_1$. Then
\[
 \operatorname{Hom}\bigl(X(\mathbf R_1),X(\mathbf R_2)[m]\bigr)
 =0=
 \operatorname{Hom}\bigl(X(\mathbf R_2),X(\mathbf R_1)[m]\bigr),
 \quad m\in\mathbb Z.
\]
In particular, $\Lambda_{\mathbf R}\cong
\Lambda_{\mathbf R_1}\times\Lambda_{\mathbf R_2}$.
\end{lemma}

\begin{proof}
Take $\gamma\in\mathbf R_1$ and $\eta\in\mathbf R_2$. The arcs $\gamma$ and
$\eta$ have no common endpoints or interior intersections. Since both endpoints
of $\gamma$ are punctures, $\rho^{-1}\gamma$ has the same underlying arc as $\gamma$, with both taggings
reversed. Consequently,
\[
 \operatorname{Int}(\eta,\gamma)=0=\operatorname{Int}(\eta,\rho^{-1}\gamma).
\]
Theorem~\ref{res:1.5} gives
\[
 \operatorname{Hom}\bigl(X(\eta),X(\gamma)[1]\bigr)=0
 \qquad\text{and}\qquad
 \operatorname{Hom}\bigl(X(\eta),X(\gamma)\bigr)=0.
\]
Moreover, Theorem~\ref{res:1.5} and $\rho^2\gamma=\gamma$ give
$X(\gamma)[2]\cong X(\gamma)$.
Hence $\operatorname{Hom}(X(\eta),X(\gamma)[m])=0$ for every $m\in\mathbb Z$.
The $2$-Calabi--Yau duality gives
\[
 \operatorname{Hom}\bigl(X(\gamma),X(\eta)[m]\bigr)
 \cong D\operatorname{Hom}\bigl(X(\eta),X(\gamma)[2-m]\bigr)=0.
\]
Taking direct sums proves the asserted vanishing. The algebra decomposition
follows by taking endomorphisms in degree zero.
\end{proof}

To pass from a puncture component of $\mathbf U$ to a decomposition of
$\Lambda_{\mathbf R}$, we use that $\mathbf R$ is a $\mathbf U$-co-dissection
by \cite[Corollary~4.9]{HZZ}. Thus $X(\mathbf R)$ lies in
$\operatorname{add}X(\mathbf U)[-1]*\operatorname{add}X(\mathbf U)$.
The following lemma extends the orthogonality to the extension category.

\begin{lemma}\label{res:3.2}
Let $V_1,V_2$ be rigid objects such that
\[
 \operatorname{Hom}(V_1,V_2[m])=0=
 \operatorname{Hom}(V_2,V_1[m]),
 \quad m\in\mathbb Z.
\]
Let $V=V_1\oplus V_2$, and define
$\mathcal X_i=\operatorname{add}V_i[-1]*\operatorname{add}V_i$ for $i=1,2$.
Then
\[
 \operatorname{Hom}(\mathcal X_1,\mathcal X_2[m])=0=
 \operatorname{Hom}(\mathcal X_2,\mathcal X_1[m]),
 \quad m\in\mathbb Z,
\]
and the direct sum functor induces an equivalence
\[
 \mathcal X_1\times\mathcal X_2
 \overset{\sim}{\longrightarrow}\operatorname{add}V[-1]*\operatorname{add}V.
\]
\end{lemma}

\begin{proof}
Take $Y_i\in\mathcal X_i$ and choose triangles
\[
 A_i[-1]\longrightarrow Y_i\longrightarrow B_i\longrightarrow A_i,
 \qquad A_i,B_i\in\operatorname{add}V_i,\quad i=1,2.
\]
Fix $m\in\mathbb Z$. Applying
$\operatorname{Hom}(A_1[-1],-)$ and $\operatorname{Hom}(B_1,-)$
to the shifted triangle for $Y_2$, the assumed orthogonality gives
\[
 \operatorname{Hom}(A_1[-1],Y_2[m])=0
 =\operatorname{Hom}(B_1,Y_2[m]).
\]
Applying $\operatorname{Hom}(-,Y_2[m])$ to the triangle for $Y_1$
then gives $\operatorname{Hom}(Y_1,Y_2[m])=0$.
Interchanging the indices proves the reverse vanishing.

For an object $Y\in\operatorname{add}V[-1]*\operatorname{add}V$, choose
a triangle
\[
 A[-1]\longrightarrow Y\longrightarrow B\overset{\delta}{\longrightarrow}A,
\]
with $A,B\in\operatorname{add}V$.
Write $A=A_1\oplus A_2$ and $B=B_1\oplus B_2$, with
$A_i,B_i\in\operatorname{add}V_i$. The orthogonality of $V_1$ and $V_2$ makes $\delta$ diagonal.
The triangle therefore decomposes as a direct sum of triangles, giving
$Y\cong Y_1\oplus Y_2$ with $Y_i\in\mathcal X_i$.
Thus the direct sum functor is essentially surjective. The degree-zero
orthogonality of $\mathcal X_1$ and $\mathcal X_2$ makes the functor fully
faithful, proving the equivalence.
\end{proof}

\begin{proposition}\label{res:3.3}
A nonempty proper puncture component of $\mathbf U$ induces a nontrivial
direct product decomposition of $\Lambda_{\mathbf R}$. In particular, if
$\Lambda_{\mathbf R}$ is connected, then $\mathbf U$ is connected to the
boundary or consists of one puncture component.
\end{proposition}

\begin{proof}
Suppose $\mathbf U=\mathbf U_1\sqcup\mathbf U_2$, where $\mathbf U_1$ is a
nonempty proper puncture component. Let $U_i=X(\mathbf U_i)$ for $i=1,2$
and $U=X(\mathbf U)=U_1\oplus U_2$.
Lemma~\ref{res:3.1} shows that $U_1$ and $U_2$ are orthogonal in every degree.
By \cite[Corollary~4.9]{HZZ}, $\mathbf R$ is a $\mathbf U$-co-dissection.
Thus $X(\mathbf R)\in\operatorname{add}U[-1]*\operatorname{add}U$.
By Lemma~\ref{res:3.2}, we have
$X(\mathbf R)\cong T_1\oplus T_2$ with
$T_i\in\operatorname{add}U_i[-1]*\operatorname{add}U_i$ for $i=1,2$.
 If $T_1=0$, then
$X(\mathbf R)\cong T_2\in\operatorname{add}U_2[-1]*\operatorname{add}U_2$.
The orthogonality of $U_1$ and $U_2$, together with Lemma~\ref{res:3.2}, gives
\[
 \operatorname{Hom}(U_1,X(\mathbf R)[m])=0,
 \quad m\in\mathbb Z.
\]
Since $U_1\in\mathcal E_R$, choose an $R$-presentation
$R_1\longrightarrow R_0\longrightarrow U_1\longrightarrow R_1[1]$.
Applying $\operatorname{Hom}(U_1,-)$ gives
$\operatorname{End}(U_1)=0$, contrary to $U_1\neq0$.
The same argument proves $T_2\neq0$. Thus both $T_i$ are nonzero.

Lemma~\ref{res:3.2} also gives zero Hom-spaces between $T_1$ and $T_2$.
Hence
\[
 \Lambda_{\mathbf R}\cong\operatorname{End}(T_1)\times\operatorname{End}(T_2)
\]
is a nontrivial direct product. Finally, if $\mathbf U$ is not connected to
the boundary, then $\mathbf U$ has a puncture component. The component is all of
$\mathbf U$ when $\Lambda_{\mathbf R}$ is connected.
\end{proof}

For connected $\Lambda_{\mathbf R}$, Proposition~\ref{res:3.3} leaves only
dissections connected to the boundary and single puncture components.
The first case remains connected to the boundary after cutting by
\cite[Lemma~5.13]{HZZ}. The next lemma shows that one cut also brings
the second case into this setting.

\begin{lemma}\label{res:3.4}
Let $\eta\in\mathbf R$. If $\mathbf R$ consists of one puncture component, then
\[
 \mathbf R/\{\eta\}=F_{\{\eta\}}\bigl(\mathbf R\setminus\{\eta\}\bigr)
\]
is connected to the boundary. In particular, if $\Lambda_{\mathbf R}$ is
connected, then $F_{\mathbf N}(\mathbf U\setminus\mathbf N)$ is connected to
the boundary for every nonempty $\mathbf N\subseteq\mathbf U$.
\end{lemma}

\begin{proof}
For any $\gamma\in\mathbf R\setminus\{\eta\}$, choose a shortest chain
\[
 \gamma=\gamma_0,\gamma_1,\ldots,\gamma_{r-1},\gamma_r=\eta
\]
in $\mathbf R$ such that $\gamma_j$ and $\gamma_{j+1}$ share an endpoint
for $0\leq j<r$.
For $j<r-1$, $\gamma_j$ has no endpoint in common with $\eta$ by the minimality
of the chain.
Thus the incidences in $\gamma_0,\ldots,\gamma_{r-1}$ are preserved by cutting.
The last arc $\gamma_{r-1}$ shares an endpoint with $\eta$. Hence $F_{\{\eta\}}(\gamma_{r-1})$
meets a marked point on the new boundary, including when
$\gamma_{r-1}$ and $\eta$ have the same underlying arc and different taggings,
as in Definition~\ref{res:2.3} and Figure~\ref{fig:3}.
The cut chain therefore joins $F_{\{\eta\}}(\gamma)$ to the boundary.

For the final assertion, Proposition~\ref{res:3.3} shows that $\mathbf U$
is connected to the boundary or consists of one puncture component.
In the first case, successive cuts preserve boundary-connectedness by
\cite[Lemma~5.13]{HZZ}. In the second case, cut an arc of $\mathbf N$ first.
The first assertion shows that the resulting partial tagged triangulation
is connected to the boundary. Subsequent cuts preserve connectedness to
the boundary by
\cite[Lemma~5.13]{HZZ}.
\end{proof}

Lemma~\ref{res:3.4} supplies the boundary condition needed to apply
\cite[Theorem~5.16]{HZZ} to the cut model of a completion graph.
Together with Corollary~\ref{res:2.7}, we obtain the following result.

\begin{theorem}\label{res:3.5}
Let $A=\Lambda_{\mathbf R}$ be connected. For every nonzero basic $\tau$-rigid
pair $u$, the completion graph $\mathcal F_A(u)$ is connected. Equivalently,
$\operatorname{EG}(s\tau\text{-}\operatorname{tilt}A_u)$ is connected.
In particular, $A$ has the reachable-in-face property of \cite{FGLZ}.
\end{theorem}

\begin{proof}
The Bongartz dissection $\mathbf U_u=\mathbf N_u\sqcup\mathbf L_u$ has
$\mathbf N_u\neq\varnothing$. By Lemma~\ref{res:3.4}, the partial tagged
triangulation $\mathbf R_u=F_{\mathbf N_u}(\mathbf L_u)$ is connected to the
boundary. The dissection graph of $\mathbf R_u$ is therefore connected by
\cite[Theorem~5.16]{HZZ}. Corollary~\ref{res:2.7} identifies the dissection
graph with $\mathcal F_A(u)$ and with
$\operatorname{EG}(s\tau\text{-}\operatorname{tilt}A_u)$.

For the reachable-in-face assertion, let $T,T'$ be basic $\tau$-tilting pairs
joined by a mutation path and having a common direct summand $v$.
If $v\neq(0,0)$, the connected graph $\mathcal F_A(v)$ joins $T$ and $T'$
while keeping $v$ fixed. If $v=(0,0)$, any mutation path from $T$ to $T'$
already keeps $v$ fixed. Thus $A$ has the reachable-in-face property
of \cite[Definition~1.4]{FGLZ}.
\end{proof}


\section{Endpoint forms and finiteness}
\label{sec:4}
Throughout this section, let $\mathbf R=\{\gamma_1,\ldots,\gamma_n\}$ be a nonempty punctured partial tagged triangulation. Let $A=\Lambda_{\mathbf R}$ be the corresponding surface rigid algebra, with Cartan matrix $C=C_A$.

\subsection{Endpoint coordinates}
For a tagged arc $\gamma$ with endpoints in $\mathbf P_{\mathbf R}$, define the \emph{endpoint vector} by
\[
 \eta(\gamma)_p=
 \sum_{\substack{t\in\{0,1\}\\\gamma(t)=p}}\kappa_\gamma(t),
 \qquad p\in\mathbf P_{\mathbf R}.
\]
The \emph{signed endpoint matrix} of $\mathbf R$ is $B=(\eta(\gamma_1),\ldots,\eta(\gamma_n))$. Thus an arc from $p$ to $q\neq p$ with taggings $a,b$ has endpoint vector $ae_p+be_q$, and a loop at $p$ with tagging $a$ has endpoint vector $2ae_p$. Since every $\mathbf R$-standard arc has endpoints in $\mathbf P_{\mathbf R}$ by \cite[Definition~4.2]{HZZ},  $\eta$ is defined on all $\mathbf R$-standard arcs.

Write $R=\bigoplus_{i=1}^nR_i$ with $R_i=X(\gamma_i)$.

\begin{proposition}\label{res:4.1}
The surface rigid algebra $A$ is weakly symmetric and satisfies
\[
 C=B^tB.
\]
In particular, the principal submatrix of $C$ indexed by a conjugate pair in $\mathbf R$ is $2I_2$.
\end{proposition}
\begin{proof}
Since the endpoints of $\gamma_i$ are punctures, $R_i[2]\cong R_i$ by
Theorem~\ref{res:1.5}. The $2$-Calabi--Yau duality gives
\[
 D\operatorname{Hom}(R_i,R)
 \cong\operatorname{Hom}(R,R_i[2])
 \cong\operatorname{Hom}(R,R_i).
\]
Thus the indecomposable projective $\operatorname{Hom}(R,R_i)$ is also
the indecomposable injective corresponding to $R_i$. Hence $A$ is weakly symmetric.

By Theorem~\ref{res:1.5}, the Cartan entries satisfy $c_{ij}=\operatorname{Int}(\gamma_i,\rho^{-1}\gamma_j)$. There are the following cases.
\begin{enumerate}[label=(\alph*),labelindent=0pt,leftmargin=*,labelsep=0.5em,align=left]
	\item[(I)] Suppose $i\neq j$.\nopagebreak[4]
	\begin{itemize}[beginpenalty=10000,label=\textendash,labelindent=0pt,leftmargin=*,labelsep=0.5em,align=left]
		\item If $\gamma_i$ and $\gamma_j$ have distinct underlying arcs then, by compatibility, $\gamma_i$ and $\gamma_j$ have no
		interior intersections and the same tagging at each common endpoint. Since $\rho^{-1}(\gamma_j)$ has the same underlying arc as $\gamma_j$ and the opposite tagging at each endpoint,
		each pair of common ends contributes one to $\operatorname{Int}(\gamma_i,\rho^{-1}\gamma_j)$. The scalar product $\eta(\gamma_i)^t\eta(\gamma_j)$ counts the same pairs of ends, since their original taggings agree. Hence
		$c_{ij}=\eta(\gamma_i)^t\eta(\gamma_j)$.

		\item If $\{\gamma_i,\gamma_j\}$ is a conjugate pair, then so is $\{\gamma_i,\rho^{-1}\gamma_j\}$. So $\operatorname{Int}(\gamma_i,\rho^{-1}\gamma_j)=0$. On the other hand, the two common endpoint contributions to $\eta(\gamma_i)^t\eta(\gamma_j)$ are $1$ and $-1$, respectively. So $c_{ij}=\eta(\gamma_i)^t\eta(\gamma_j)=0$.
	\end{itemize}
	\item[(II)] Suppose $i=j$. Then by Example~\ref{res:1.4}, we have $c_{ii}=2$ for $\gamma_i(0)\neq\gamma_i(1)$ and $c_{ii}=4$ otherwise. In both cases, we have $c_{ii}=\|\eta(\gamma_i)\|^2$.
\end{enumerate}
Thus $C=B^tB$, and the principal submatrix for a conjugate pair is $2I_2$.
\end{proof}

\begin{proposition}\label{prop:endpoint-formula}\label{res:7.1}
For each $\mathbf R$-standard tagged arc $\gamma$, we have
\[
 Bg_{\mathbf R}(\gamma)=\eta(\gamma).
\]
\end{proposition}
\begin{proof}
Let $\gamma$ be an $\mathbf R$-standard arc. Let $Y=X(\gamma)$, $g=\operatorname{ind}_R Y$ and $\eta=\eta(\gamma)$. Choose an $R$-presentation
	\begin{equation}\label{eq:3.1.1}
		R^1\longrightarrow R^0\longrightarrow Y\longrightarrow R^1[1]
	\end{equation}
	of $Y$ with $R^0,R^1\in\operatorname{add}R$.  For $Z\in\mathcal C$, let $h(Z)=(\dim_k\operatorname{Hom}(R_i,Z))_i$. Since $R$ is rigid and $R[-1]\cong R[1]$, applying
	$\operatorname{Hom}(R,-)$ to \eqref{eq:3.1.1} gives the exact sequence
	\[	0\longrightarrow\operatorname{Hom}(R,Y[-1])\longrightarrow\operatorname{Hom}(R,R^1)	\longrightarrow\operatorname{Hom}(R,R^0)\longrightarrow\operatorname{Hom}(R,Y)\longrightarrow0.\]
	Since $g=[R^0]-[R^1]$, taking dimension vectors gives
	\begin{equation}\label{eq:4.1}
		Cg=h(Y)-h(Y[-1]).
	\end{equation}
	By Theorem~\ref{res:1.5}, the $i$th coordinate $d_i$ of $Cg$ is
	\[d_i=\operatorname{Int}(\gamma_i,\rho^{-1}\gamma)-\operatorname{Int}(\gamma_i,\gamma).\]
	Since $\rho^{-1}\gamma$ is obtained from $\gamma$ by changing only the taggings, the interior contributions cancel. So we have the following two cases.
	\begin{enumerate}[label=(\Roman*),labelindent=0pt,leftmargin=*,labelsep=0.5em,align=left]
		\item Suppose $\gamma_i$ and $\gamma$ have distinct underlying arcs. Let $(s,t)\in\{0,1\}^2$ with $\gamma_i(s)=\gamma(t)$. By the definition of tagged intersections, $(s,t)$ contributes $(1-\kappa_{\gamma_i}(s)\kappa_\gamma(t))/2$ to $\operatorname{Int}(\gamma_i,\gamma)$. Rotation replaces $\kappa_\gamma(t)$ by $-\kappa_\gamma(t)$, so $(s,t)$ contributes $(1+\kappa_{\gamma_i}(s)\kappa_\gamma(t))/2$ to $\operatorname{Int}(\gamma_i,\rho^{-1}\gamma)$. Hence
		\[
		d_i=\sum_{\substack{s,t\in\{0,1\}\\\gamma_i(s)=\gamma(t)}}
		\kappa_{\gamma_i}(s)\kappa_\gamma(t)=\eta(\gamma_i)^t\eta.
		\]
		\item Suppose $\gamma_i$ and $\gamma$ have the same underlying arc. Choose parametrizations such that $\gamma_i(t)=\gamma(t)$ for $t=0,1$.
		If $\gamma(0)\neq \gamma(1)$, then the definition of tagged intersection gives
		\[
		d_i=
		\begin{cases}
			2,&\kappa_{\gamma_i}=\kappa_\gamma,\\
			-2,&\kappa_{\gamma_i}=-\kappa_\gamma,\\
			0,&\kappa_{\gamma_i}\notin\{\kappa_\gamma,-\kappa_\gamma\}.
		\end{cases}
		\]
		If $\gamma(0)=\gamma(1)$, then both tagging maps are constant and
		\[
		d_i=
		\begin{cases}
			4,&\kappa_{\gamma_i}=\kappa_\gamma,\\
			-4,&\kappa_{\gamma_i}=-\kappa_\gamma.
		\end{cases}
		\]
		In both cases, $d_i=\eta(\gamma_i)^t\eta$.
	\end{enumerate}
	Thus $Cg=B^t\eta$.
	Since $C=B^tB$ by Proposition~\ref{res:4.1}, we have $Bg-\eta\in\ker B^t$. To prove $Bg=\eta$, we compute $\|Bg-\eta\|^2$. Since the endpoints of $\gamma$ are punctures, $Y[2]\cong Y$ by Theorem~\ref{res:1.5}. Since $Y$ is rigid, we have $\operatorname{Hom}(Y[1],Y)=0=\operatorname{Hom}(Y[-1],Y)$. Applying $\operatorname{Hom}(-,Y)$ to \eqref{eq:3.1.1} gives the exact sequence
	\begin{multline*}
		0\longrightarrow\operatorname{Hom}(R^0,Y[-1])
		\longrightarrow\operatorname{Hom}(R^1,Y[-1])
		\longrightarrow\operatorname{End}(Y)\\
		\longrightarrow\operatorname{Hom}(R^0,Y)
		\longrightarrow\operatorname{Hom}(R^1,Y)\longrightarrow0.
	\end{multline*}
	Taking the alternating sum of dimensions and using \eqref{eq:4.1}, we obtain
	\[
	\dim_k\operatorname{End}(Y)=g^t\bigl(h(Y)-h(Y[-1])\bigr)=g^tCg.
	\]
	By Theorem~\ref{res:1.5} and Example~\ref{res:1.4}, we have $\dim_k\operatorname{End}(Y)=\|\eta\|^2$. Combining these equalities with $Cg=B^t\eta$ yields
	\[
	\|Bg-\eta\|^2=g^tCg-2g^tB^t\eta+\|\eta\|^2=\|\eta\|^2-g^tCg=0.
	\]
	Thus $Bg=\eta$.
\end{proof}

Since $\ker(B^tB)=\ker B$ over $\mathbb R$, we have $\operatorname{rank}_{\mathbb Q}C=\operatorname{rank}_{\mathbb Q}B$. If $\mathbf R$ has no conjugate pair, compatibility gives a common
tagging $\kappa_{\mathbf R}(p)$ at each incident puncture. Let $D=(d_{pi})$, where $d_{pi}=|\{t\in\{0,1\}:\gamma_i(t)=p\}|$ is the number of ends of $\gamma_i$ at $p$. Then
\[
 B=\operatorname{diag}(\kappa_{\mathbf R}(p))D
 \qquad\text{and}\qquad C=D^tD.
\]
The rank-one case will be used to classify disconnected exchange graphs in Section~\ref{sec:5}.

\begin{corollary}\label{res:4.2}
Suppose $n\geq2$. Then $\operatorname{rank}C_A=1$ if and only if the underlying arcs are pairwise distinct and all have the same endpoint multiset. Moreover, in this case, $C_A=2J_n$ if $|{\mathbf P_\mathbf R}|=2$ and $C_A=4J_n$ if $|{\mathbf P_\mathbf R}|=1$, where $J_n$ is the $n\times n$ all-ones matrix.
\end{corollary}
\begin{proof}
Suppose $\operatorname{rank}C=1$. A conjugate pair would give a principal submatrix $2I_2$, so the underlying arcs of $\mathbf R$ are pairwise distinct. Hence $C=D^tD$ and $\operatorname{rank}D=1$. So the columns of $D$ are proportional. Since each column of $D$ has coordinate sum two, all columns of $D$ are equal. So the arcs have the same endpoint multiset.

Conversely, suppose the underlying arcs are pairwise distinct and have the same endpoint multiset. The columns of $D$ are all $e_p+e_q$ for fixed distinct punctures $p,q$, or all $2e_p$ for a fixed puncture $p$. Thus $C=2J_n$ or $4J_n$, respectively, and $\operatorname{rank}C=1$.
\end{proof}

\subsection{Finiteness}

 Proposition~\ref{prop:endpoint-formula} makes the Cartan form positive on every nonzero vector of the $g$-fan and bounds its value on primitive ray generators. These two facts give the finiteness criterion below.

\begin{corollary}\label{res:4.4}\label{res:4.5}
Let $X,Y\in\mathcal E_R$ with $X\oplus Y$ rigid. Then
\[
 (\operatorname{ind}_R X)^tC\operatorname{ind}_R Y
 =\dim_k\operatorname{Hom}_{\mathcal C(\mathbf S)}(X,Y)\geq0.
\]
In particular, for an indecomposable $\tau$-rigid pair $u$ represented by $\gamma$,
\[
 g(u)^tCg(u)=
 \begin{cases}
 2,&\gamma(0)\neq\gamma(1),\\
 4,&\gamma(0)=\gamma(1).
 \end{cases}
\]
\end{corollary}
\begin{proof}
For indecomposable $X=X(\alpha)$ and $Y=X(\beta)$, by Propositions~\ref{res:4.1} and~\ref{prop:endpoint-formula}, we have
\[
 (\operatorname{ind}_R X)^tC\operatorname{ind}_R Y
 =\eta(\alpha)^t\eta(\beta).
\]
Since $X\oplus Y$ is rigid, $\alpha$ and $\beta$ are compatible. As in the proof of Proposition~\ref{res:4.1}, we have

\[
 \eta(\alpha)^t\eta(\beta)
 =\operatorname{Int}(\alpha,\rho^{-1}\beta)
 =\dim_k\operatorname{Hom}(X,Y),
\]
where the second equality follows from Theorem~\ref{res:1.5}. Additivity gives the first assertion for direct sums. The second assertion follows by taking $X=Y=X(\gamma)$.
\end{proof}

To decide when $C=B^tB$ is positive definite, we determine when the columns of $B$ are linearly independent. We express the rank calculation in terms of the following signed graph.

The \emph{signed endpoint graph} $\Gamma_{\mathbf R}$ has vertex set $\mathbf P_{\mathbf R}$ and one edge for each arc of $\mathbf R$. Loops and parallel edges are retained, including conjugate pairs. An edge $e$ with endpoint taggings $a,b$ has sign $\operatorname{sign}(e)=-ab$, so every loop has sign $-1$. A cycle of $\Gamma_{\mathbf R}$ is called \emph{balanced} when the product of the edge signs is $1$, and \emph{unbalanced} otherwise. We regard loops and pairs of parallel edges as cycles of lengths one and two, respectively. A connected graph is called \emph{unicyclic} if it has exactly one cycle. A signed graph is balanced if every cycle is balanced, and unbalanced otherwise. Note that connectedness of $\Gamma_{\mathbf R}$ need not imply connectedness of the algebra $\Lambda_{\mathbf R}$. Indeed, if $\mathbf R$ consists only of a conjugate pair then $\Gamma_{\mathbf R}$ is connected. But $C_{\Lambda_{\mathbf R}}=2I_2$, which implies that $\Lambda_{\mathbf R}$ is a product of two local algebras.

\begin{lemma}\label{res:4.6}
Let $\Gamma$ be a connected component of $\Gamma_{\mathbf R}$ with $v$ vertices and $m$ edges, and let $B_\Gamma$ be the corresponding $v\times m$ endpoint matrix. Then
\[
 \operatorname{rank}_{\mathbb Q}B_\Gamma=
 \begin{cases}
 v-1,&\text{ $\Gamma$ is balanced},\\
 v,&\text{ $\Gamma$ is unbalanced}.
 \end{cases}
\]
The columns of $B_\Gamma$ are independent precisely when $\Gamma$  is a tree or is unbalanced and unicyclic.
\end{lemma}
\begin{proof}
To compute $\ker B_\Gamma^t$, write $x=(x_p)_{p\in V(\Gamma)}\in\mathbb Q^v$. The column of an edge $e$ from $p$ to $q\neq p$ with taggings $a,b$ gives the equation
\[
 f_e:\quad ax_p+bx_q=0.
\]
So $x_q=\operatorname{sign}(e)x_p$. Since $\Gamma$ is connected, $x$ is determined by its value at one vertex. If $\Gamma$ has a loop $l$ at $p$ with tagging $a\in\{\pm1\}$, then
\[
 f_l:\quad 2ax_p=0.
\]
So $x_p=0$ and hence $x=0$.

Assume $\Gamma$ has no loops. Choose a spanning tree $T$ of $\Gamma$. Fix $p_0\in V(\Gamma)$. For each $t\in\mathbb Q$, the equations $f_e$ for $e\in E(T)$ uniquely determine all coordinates from $x_{p_0}=t$. Let $l\in E(\Gamma)\setminus E(T)$ have endpoints $p,q$, and let $e_1,\ldots,e_s$ be the edges of the unique path from $p$ to $q$ in $T$. The tree equations give
\[
 x_q=\left(\prod_{i=1}^s\operatorname{sign}(e_i)\right)x_p.
\]
Thus $f_l$ holds if and only if
\[
 x_p=\operatorname{sign}(l)
 \left(\prod_{i=1}^s\operatorname{sign}(e_i)\right)x_p.
\]
 If $\Gamma$ is balanced, the sign product on the right is $1$. So the coordinates determined by any $t\in\mathbb Q$ satisfy all edge equations, and $\dim_{\mathbb Q}\ker B_\Gamma^t=1$.

If $\Gamma$ has an unbalanced cycle $Z$, its edge equations give 
\[
 x_p=\left(\prod_{e\in E(Z)}\operatorname{sign}(e)\right)x_p=-x_p
\]
at a vertex $p$ of $Z$. So $x_p=0$ and hence $x=0$ by connectedness. Together with the loop case, we have $\ker B_\Gamma^t=0$ whenever $\Gamma$ is unbalanced.
 The rank formula follows from
\[
 \operatorname{rank}_{\mathbb Q}B_\Gamma
 =\operatorname{rank}_{\mathbb Q}B_\Gamma^t
 =v-\dim_{\mathbb Q}\ker B_\Gamma^t.
\]

The columns of $B_\Gamma$ are independent if and only if $\operatorname{rank}_{\mathbb Q}B_\Gamma=m$. Full column rank is therefore equivalent to $m=v-1$ in the balanced case and $m=v$ in the unbalanced case.  Since $\Gamma$ is connected, $m=v-1$ is equivalent to $\Gamma$ being a tree, and $m=v$ is equivalent to $\Gamma$ being unicyclic. This proves the last assertion.
\end{proof}

The norm bound and the rank formula give the following equivalent finiteness criteria.

\begin{theorem}\label{res:4.7}
Let $\mathbf R$ be a nonempty punctured partial tagged triangulation, and let $A=\Lambda_{\mathbf R}$. The following conditions are equivalent:
\begin{enumerate}[label=(\roman*)]
 \item $A$ is $\tau$-tilting finite;
 \item $C_A$ is positive definite;
 \item $C_A$ is nonsingular;
 \item  every connected component of $\Gamma_{\mathbf R}$ is a tree or is unbalanced and unicyclic;
 \item the $g$-fan $\Sigma_g(A)$ is complete.
\end{enumerate}
Moreover,
\[
 |\Sigma_g(A)|\cap\ker C_A=\{0\}.
\]
If the underlying arcs of $\mathbf R$ are pairwise distinct, condition \textup{(iv)} is equivalent to requiring that every connected component of $\Gamma_{\mathbf R}$ is a tree or is unicyclic with a cycle of odd length.
\end{theorem}
\begin{proof}
By Proposition~\ref{res:4.1}, $C=B^tB$ is positive semidefinite. Thus \textup{(ii)} is equivalent to \textup{(iii)}. Grouping the rows and columns of $B$ by the connected components of $\Gamma_{\mathbf R}$ gives the blocks $B_\Gamma$. Since $\operatorname{rank}C=\operatorname{rank}B$, the matrix $C$ is nonsingular if and only if every $B_\Gamma$ has full column rank. Lemma~\ref{res:4.6} therefore gives \textup{(iii)}$\Longleftrightarrow$\textup{(iv)}. Equivalence \textup{(i)}$\Longleftrightarrow$\textup{(v)} follows from \cite[Section~5]{DIJ} and \cite[Proposition~4.2]{AHIKM}.

For $0\neq x\in|\Sigma_g(A)|$, choose a cone containing $x$ and write $x=\sum_i a_i g_i$, where $a_i\geq0$ and the $g_i$ correspond to compatible indecomposable pairs. Corollary~\ref{res:4.4} gives
\[
 x^tCx
 =\sum_i a_i^2g_i^tCg_i+
 2\sum_{i<j}a_i a_jg_i^tCg_j
 \geq2\sum_i a_i^2>0.
\]
Hence $x\notin\ker C$. If $\Sigma_g(A)$ is complete, the same inequality holds for every nonzero $x\in\mathbb R^n$. So $C$ is positive definite. Hence \textup{(v)}$\Longrightarrow$\textup{(ii)}.

Suppose $C$ is positive definite. The set $\{z\in\mathbb Z^n:0<z^tCz\leq4\}$ is finite and contains the $g$-vector of every indecomposable $\tau$-rigid pair by Corollary~\ref{res:4.5}. Since indecomposable $\tau$-rigid pairs are uniquely determined by their $g$-vectors (see \cite[Theorems~5.1 and~5.5]{AIR}), there are only finitely many indecomposable $\tau$-rigid pairs, and hence finitely many basic $\tau$-tilting pairs. Thus $A$ is $\tau$-tilting finite, proving \textup{(ii)}$\Longrightarrow$\textup{(i)}.

When the underlying arcs are pairwise distinct, compatibility gives a common tagging at every incident puncture. The product of edge signs around a cycle of length $\ell$ is therefore $(-1)^\ell$. The sign product formula includes loops and parallel pairs, so unbalanced cycles are exactly the odd cycles.
\end{proof}

The implication from $\tau$-tilting finiteness to positive definiteness holds for every weakly symmetric algebra \cite[Proposition~3.2(a)]{H}. The converse in Theorem~\ref{res:4.7} uses the uniform norm bound in Corollary~\ref{res:4.5}, as in the proof of \cite[Theorem~4.1]{Ai}. For Brauer graph algebras, the corresponding graph criterion is \cite[Theorem~6.7]{AAC}.


\section{Homotopy types of \texorpdfstring{$\tau$}{tau}-tilting complexes}
\label{sec:topology}

For any finite-dimensional algebra, the realization of the support $\tau$-tilting complex is homotopy equivalent to the complement of Asai's purely non-rigid region in $K_{A,\mathbb R}$. We then identify this region with the Cartan kernel for tame weakly symmetric algebras satisfying a boundedness condition.  The identity $Bg_{\mathbf R}(\gamma)=\eta(\gamma)$ in Proposition~\ref{prop:endpoint-formula} verifies the boundedness condition for surface rigid algebras defined by punctured partial tagged triangulations.

\subsection{Presilting neighbourhoods}
\label{sec:presilting-neighborhoods}\label{top:neighbourhood-section}

Let $A$ be an arbitrary finite-dimensional basic algebra, and let $\mathcal K=K^b(\operatorname{proj}A)$ and $V=K_{A,\mathbb R}$. For a two-term complex $X=(P^{-1}\xrightarrow{d_X}P^0)$, write
\[
 g(X)=[X]=[P^0]-[P^{-1}],\quad
 M_X=H^0(X)\quad\text{and}\quad K_X=H^{-1}(\nu_A X).
\]
We regard $g(X)$ as an integral vector with respect to the projective basis $[P_1],\ldots,[P_n]$ of $K_0(\mathcal K)$. If $X$ is presilting and corresponds to the $\tau$-rigid pair $u_X$, then $g(X)=g_A(u_X)$. For a module $L$ and $\theta=\sum_i\theta_i[P_i]\in V$, let
\[
 \theta(L)=\sum_i\theta_i\dim_k\operatorname{Hom}_A(P_i,L).
\]

We shall construct open subsets of $V$ whose intersections record the faces of $\Delta_\tau(A)$.

\begin{lemma}\label{top:euler}
For two-term complexes $X,Y\in\mathcal K$ and any $A$-module $L$, there is a natural isomorphism
\begin{equation}\label{top:dual-hom}
 D\operatorname{Hom}_{\mathcal K}(X,Y[1])\cong\operatorname{Hom}_A(M_Y,K_X).
\end{equation}
Moreover, we have
\begin{equation}\label{top:euler-identity}
 g(X)(L)=\dim_k\operatorname{Hom}_A(M_X,L)-\dim_k\operatorname{Hom}_A(L,K_X).
\end{equation}
\end{lemma}
\begin{proof}
Write $Y=(Q^{-1}\xrightarrow{d_Y}Q^0)$. Chain maps and homotopies give
\[
 \operatorname{Hom}_{\mathcal K}(X,Y[1])
 \cong
 \frac{\operatorname{Hom}_A(P^{-1},Q^0)}
 {\operatorname{Hom}_A(P^0,Q^0)d_X+
       d_Y\operatorname{Hom}_A(P^{-1},Q^{-1})}.
\]
Under Nakayama duality, the dual of this quotient consists of the maps
$f:Q^0\to\nu_A P^{-1}$ satisfying $fd_Y=0$ and $(\nu_A d_X)f=0$.
These are precisely the maps $M_Y\to K_X$.

The map $\operatorname{Hom}_A(d_X,L)$ has kernel $\operatorname{Hom}_A(M_X,L)$. Nakayama duality identifies the dual of the cokernel with $\operatorname{Hom}_A(L,K_X)$. Thus we have an exact sequence
\[
 0\longrightarrow\operatorname{Hom}_A(M_X,L)
 \longrightarrow\operatorname{Hom}_A(P^0,L)
 \longrightarrow\operatorname{Hom}_A(P^{-1},L)
 \longrightarrow D\operatorname{Hom}_A(L,K_X)
 \longrightarrow0.
\]
Taking dimensions proves \eqref{top:euler-identity}.
\end{proof}

For a nonzero two-term presilting object $U$, we use strict inequalities to define an open neighbourhood and weak inequalities to define a closed cone in $V$. We use the neighbourhoods to characterize compatibility and the closed cones to take limits in Section~\ref{top:sec-cartan-kernel}. Let
\begin{equation}\label{top:neighbourhood-definition}
 N_U=\left\{\theta\in V:\begin{array}{ll}
 \theta(Q)>0&\text{for every nonzero quotient $Q$ of $M_U$,}\\
 \theta(L)<0&\text{for every nonzero submodule $L$ of $K_U$}
 \end{array}\right\}
\end{equation}
and
\begin{equation}\label{top:closedcone-definition}
 D_U=\left\{\theta\in V:\begin{array}{ll}
 \theta(Q)\geq0&\text{for every quotient $Q$ of $M_U$,}\\
 \theta(L)\leq0&\text{for every submodule $L$ of $K_U$}
 \end{array}\right\}.
\end{equation}
By filtrations and additivity, the inequalities defining $N_U$ also hold for nonzero quotients of finite sums of $M_U$ and submodules of finite sums of $K_U$. Thus $N_U$ is the presilting neighbourhood of \cite[Definition~4.1]{Asai}. Only finitely many dimension vectors occur among the quotients of $M_U$ and the submodules of $K_U$. Thus $N_U$ is open and convex, and $D_U$ is a closed convex cone.

\begin{lemma}\label{top:closedcone}
Let $U$ be a nonzero two-term presilting object in $\mathcal K$.
\begin{enumerate}[label=\textup{(\arabic*)},leftmargin=*,labelsep=.5em,itemsep=.3em]
\item $g(U)\in N_U$. In particular, for every $c>0$,
\begin{equation}\label{top:strict-plus-closed}
 c g(U)+D_U\subseteq N_U.
\end{equation}
\item If a two-term complex $Y\in\mathcal K$ satisfies
\[
 \operatorname{Hom}_{\mathcal K}(U,Y[1])=0
   =\operatorname{Hom}_{\mathcal K}(Y,U[1]),
\]
then $g(Y)\in D_U$.
\end{enumerate}
\end{lemma}
\begin{proof}
(1) By \eqref{top:dual-hom}, $\operatorname{Hom}_A(M_U,K_U)=0$. So for any nonzero quotient $Q$ of $M_U$, we have $\operatorname{Hom}_A(Q,K_U)=0$ and $\operatorname{Hom}_A(M_U,Q)\ne0$. Thus \eqref{top:euler-identity} gives $g(U)(Q)>0$. Dually, for any nonzero submodule $L$ of $K_U$, we have $\operatorname{Hom}_A(M_U,L)=0$ and $\operatorname{Hom}_A(L,K_U)\ne0$. So $g(U)(L)<0$, and hence $g(U)\in N_U$. For $c>0$ and $\theta\in D_U$, the defining inequalities give $(c g(U)+\theta)(Q)>0$ and $(c g(U)+\theta)(L)<0$. This proves \eqref{top:strict-plus-closed}.

(2) For $Y$ as in the statement, \eqref{top:dual-hom} gives $\operatorname{Hom}_A(M_U,K_Y)=0=\operatorname{Hom}_A(M_Y,K_U)$. So $\operatorname{Hom}_A(Q,K_Y)=0$ for every quotient $Q$ of $M_U$, and $\operatorname{Hom}_A(M_Y,L)=0$ for every submodule $L$ of $K_U$. Thus \eqref{top:euler-identity} gives
\[
 g(Y)(Q)=\dim_k\operatorname{Hom}_A(M_Y,Q)\geq0
 \quad\text{and}\quad
 g(Y)(L)=-\dim_k\operatorname{Hom}_A(L,K_Y)\leq0.
\]
Hence $g(Y)\in D_U$.
\end{proof}

\begin{lemma}[{\cite[Lemma~4.2(3)]{Asai}}]\label{top:intersection}
For nonzero two-term presilting objects $X_1,\ldots,X_r$,  let $U=\bigoplus_iX_i$. Then
\[
 \bigcap_iN_{X_i}\ne\varnothing
 \quad\Longleftrightarrow\quad U\text{ is presilting}.
\]
In this case the intersection equals $N_U$.
\end{lemma}
\begin{proof}
If $\theta\in\bigcap_iN_{X_i}$, a nonzero map
$f\in\Hom_A(M_{X_j},K_{X_i})$ would give
$0<\theta(\operatorname{im}f)<0$. Hence
$\Hom_A(M_{X_j},K_{X_i})=0$ for all $i,j$, and
\eqref{top:dual-hom} shows that $U$ is presilting.

Conversely, suppose that $U$ is presilting. Every quotient of $M_U=\bigoplus_i M_{X_i}$ has a filtration whose factors are quotients of the $M_{X_i}$; the analogous assertion holds for submodules of $K_U=\bigoplus_iK_{X_i}$. So additivity of weights gives $N_U=\bigcap_iN_{X_i}$. This set contains $g(U)$ by Lemma~\ref{top:closedcone}(1).
\end{proof}

Following \cite[Definition~4.6]{Asai}, let
\begin{equation}\label{top:r0-definition}
 R_0(A)=V\setminus\bigcup_{X\in\indtwopsilt A}N_X.
\end{equation}
 By Lemma~\ref{top:intersection}, the nerve of this cover is identified with $\Delta_\tau(A)$ under the bijection between $\twopsilt A$ and $\tau\text{-}\operatorname{rigidp}A$ in \cite{AIR}. The following proposition gives a homotopy equivalence in $g$-vector coordinates.

\begin{proposition}\label{top:nerve}
The map
\begin{equation}\label{top:g-map}
 G_A:|\Delta_\tau(A)|\longrightarrow V\setminus R_0(A), \qquad\sum_i t_i u_i\longmapsto\sum_i t_i g_A(u_i)
\end{equation}
is a homotopy equivalence. Here the $u_i$ are vertices of a simplex and $t_i\geq0$, $\sum_i t_i=1$, are the barycentric coordinates.
\end{proposition}
\begin{proof}
Let $W=V\setminus R_0(A)$. By Lemma~\ref{top:intersection}, $\Delta_\tau(A)$ is the nerve of the convex open cover $(N_X)_{X\in\indtwopsilt A}$. We use a partition of unity to construct a homotopy inverse of $G_A$, as in the nerve theorem \cite[Corollary~4G.3]{Hatcher}.

Let $x=\sum_i t_i u_{X_i}\in|\Delta_\tau(A)|$ with all $t_i>0$, and let $U=\bigoplus_iX_i$ and $c=\min_i t_i$. Since $U$ is presilting, Lemma~\ref{top:closedcone}(2) gives $g(X_i)\in D_U$ for every $i$. By Lemma~\ref{top:closedcone}(1), we have
\begin{equation}\label{top:g-map-neighbourhood}
 G_A(x)=c g(U)+\sum_i(t_i-c)g(X_i)\in N_U=\bigcap_iN_{X_i}.
\end{equation}
Thus $G_A$ takes values in $W$. The restriction to each simplex is linear, so $G_A$ is continuous in the weak topology.

Since $W$ is a metric space and hence paracompact, choose a locally finite partition of unity $(\varphi_X)$ subordinate to $(N_X)$ and let
\[
 F:W\longrightarrow|\Delta_\tau(A)|, \qquad\theta\longmapsto\sum_X\varphi_X(\theta)u_X.
\]
For $\theta\in W$, the set $I_\theta=\{X:\varphi_X(\theta)>0\}$ is finite, and
\[
 \theta\in\bigcap_{X\in I_\theta}N_X.
\]
So $\{u_X:X\in I_\theta\}$ spans a simplex of $\Delta_\tau(A)$. Since $\sum_X\varphi_X(\theta)=1$, the point $F(\theta)$ lies in this simplex. Thus $F$ is well-defined. Local finiteness implies that, near each $\theta$, only finitely many coefficients can be nonzero. Hence $F$ locally takes values in a finite subcomplex and is continuous.

 By \eqref{top:g-map-neighbourhood}, both $\theta$ and $G_AF(\theta)$ lie in the convex set $\bigcap_{X\in I_\theta}N_X$. Therefore
\[
 H(\theta,t)=(1-t)\theta+tG_AF(\theta),\qquad t\in[0,1],
\]
gives a homotopy $\mathrm{id}_W\simeq G_AF$.

For $x=\sum_i t_i u_{X_i}\in|\Delta_\tau(A)|$ with all $t_i>0$,  \eqref{top:g-map-neighbourhood} gives $G_A(x)\in\bigcap_iN_{X_i}$. Thus every neighbourhood indexed by $\operatorname{supp}(x)=\{u_{X_i}\}$ contains $G_A(x)$. If $u_X\in\operatorname{supp}(FG_A(x))$, then $\varphi_X(G_A(x))>0$, so $G_A(x)\in N_X$. Hence the neighbourhoods indexed by $\operatorname{supp}(x)\cup\operatorname{supp}(FG_A(x))$ have a common point. This union therefore spans a simplex of $\Delta_\tau(A)$.

The barycentric segment
\[
 H'(x,s)=(1-s)x+sFG_A(x),\qquad s\in[0,1],
\]
lies in that simplex, with $H'(x,0)=x$ and $H'(x,1)=FG_A(x)$.
For each simplex $\sigma$, the image $G_A(\sigma)$ is compact and the partition of unity is locally finite.  Thus only finitely many $\varphi_X$ satisfy $\varphi_X|_{G_A(\sigma)}\ne0$. So the restriction $H'|_{\sigma\times[0,1]}$ takes values in a finite subcomplex. Since its barycentric coordinate functions are continuous, the restriction is continuous. Since $[0,1]$ is a finite CW complex, the product $|\Delta_\tau(A)|\times[0,1]$ has the weak topology with respect to the products $\sigma\times[0,1]$ by \cite[Theorem~A.6]{Hatcher}. Hence $H'$ is continuous and $\mathrm{id}_{|\Delta_\tau(A)|}\simeq FG_A$.
\end{proof}



\subsection{\texorpdfstring{The Cartan kernel}{The Cartan kernel}}
\label{top:sec-kernel}\label{top:sec-cartan-kernel}

In this subsection, let $A$ be a nonzero tame weakly symmetric algebra, and write $C=C_A$. 
We first use generic decompositions to write each rational weight as a nonnegative rational combination of presilting $g$-vectors plus a vector in $\ker C$. Under a boundedness assumption on $Cg(X)$, we then pass to real weights and prove $R_0(A)=\ker C$.

For any $A$-module $M$, write $\mathbf dM=(\dim_k\Hom_A(P_i,M))_{i=1}^n$. The Nakayama functor $\nu_A$ is an exact autoequivalence with $\nu_A P_i\cong P_i$. It fixes the simple modules up to isomorphism and hence preserves dimension vectors. Nakayama duality gives
\[
 D\Hom_A(P_i,P_j)\cong\Hom_A(P_j,\nu_A P_i).
\]
Since $\nu_A P_i\cong P_i$, it follows that $C=C^t$. For a two-term complex $Y=(P_Y^{-1}\xrightarrow{d_Y}P_Y^0)\in\mathcal K$, exactness of $\nu_A$ gives
\[
 0\longrightarrow K_Y\longrightarrow\nu_A P_Y^{-1}
 \xrightarrow{\nu_A d_Y}\nu_A P_Y^0
 \longrightarrow\nu_A M_Y\longrightarrow0.
\]
Applying $\Hom_A(P_i,-)$ and taking dimensions, we obtain
\begin{equation}\label{top:cartan-cohomology}
 \begin{aligned}
 Cg(Y)&=\mathbf dP_Y^0-\mathbf dP_Y^{-1}\\
      &=\mathbf d(\nu_A P_Y^0)-\mathbf d(\nu_A P_Y^{-1})\\
      &=\mathbf d(\nu_A M_Y)-\mathbf dK_Y
       =\mathbf dM_Y-\mathbf dK_Y,
 \end{aligned}
\end{equation}
where the last equality holds because $\nu_A$ preserves dimension vectors.

For $z\in\mathbb Z^n$, define
\begin{equation}\label{top:presentation-space}
 P^-(z)=\bigoplus_{z_i<0}P_i^{\oplus(-z_i)},\quad
 P^+(z)=\bigoplus_{z_i>0}P_i^{\oplus z_i}\quad\text{and}\quad
 E_z=\Hom_A(P^-(z),P^+(z)).
\end{equation}
We regard $E_z$ as an affine space. Each $d\in E_z$ defines a minimal two-term complex $T_d=(P^-(z)\xrightarrow dP^+(z))$ of class $z$, called a \emph{presentation} of $z$. A property holds for a \emph{general presentation} of $z$ if it holds on a dense open subset of $E_z$. A nonzero vector $z$ is \emph{generically indecomposable} if a general presentation of $z$ is indecomposable. For $z,w\in\mathbb Z^n$, let
\[
 e(z,w)=\min_{d\in E_z,\ d'\in E_w} \dim_k\Hom_{\K}(T_d,T_{d'}[1]).
\]
The function under the minimum is upper semicontinuous. Since $E_z\times E_w$ is irreducible, the minimum is attained on a dense open subset.

The canonical decomposition theorem \cite[Theorem~4.4]{DF}, in the form of \cite[Theorem~2.3]{PY}, writes each $z$ as a sum of generically indecomposable vectors
\[
 z=z_1+\cdots+z_r \text{ such that } e(z_\ell,z_m)=0=e(z_m,z_\ell) \text{ for any } \ell\neq m. 
\]
A general presentation of $z$ is isomorphic to a direct sum of general presentations of these vectors. For tame algebras, \cite[Theorem~3.8]{PY} separates these summands into two types:
\begin{equation}\label{top:generic-decomposition}
 z=\sum_{i=1}^s m_i g(X_i)+\sum_{j=1}^t b_jh_j, \qquad m_i,b_j\in\mathbb Z_{>0}.
\end{equation}
Here $X_1,\ldots,X_s\in\indtwopsilt A$ are pairwise distinct. For each $h_j$, a general presentation $Y_j$ has indecomposable cokernel $M_j$ with $\End_A(M_j)=k$ and $\tau M_j\cong M_j$. Repeated summand vectors have been grouped in \eqref{top:generic-decomposition}. The $X_i$ are mutually compatible, so $\bigoplus_iX_i$ is presilting and $s\leq n$ by \cite[Proposition~3.3]{AIR}.

The following lemma gives the required decomposition and bounds the coefficients of the presilting terms.
For $x=(x_1,\ldots,x_n)\in\mathbb R^n$, write $\|x\|_1=\sum_{k=1}^n|x_k|$.

\begin{lemma}\label{top:rational-decomposition}
For each $\theta\in\mathbb Q^n$, there are pairwise distinct $X_1,\ldots,X_s\in\indtwopsilt A$ with $\bigoplus_iX_i$ presilting and $0\leq s\leq n$, and $a_1,\ldots,a_s\in\mathbb Q_{>0}$ such that
\begin{equation}\label{top:rational-expansion}
 \theta=\sum_{i=1}^s a_i g(X_i)+\delta \qquad\text{with}\qquad\delta\in\ker C.
\end{equation}
For every $i$,
\begin{equation}\label{top:rational-control}
 a_i\|g(X_i)\|_1\leq\|\theta\|_1 \qquad\text{and}\qquad\theta-a_i g(X_i)\in D_{X_i}.
\end{equation}

\end{lemma}

\begin{proof}
Choose $d\in\mathbb Z_{>0}$ with $z=d\theta\in\mathbb Z^n$ and use \eqref{top:generic-decomposition}. A general presentation $Y_j$ of $h_j$ is indecomposable with nonzero cokernel $M_j$. So it has no shifted projective summand and is a minimal projective presentation of $M_j$. Consequently
\[
 K_{Y_j}\cong\tau M_j\cong M_j \qquad\text{and hence}\qquad Ch_j=0
\]
by \eqref{top:cartan-cohomology}. Set $a_i=m_i/d$ and $\delta=d^{-1}\sum_jb_jh_j$. This gives \eqref{top:rational-expansion}.

By definition, $P^-(z)$ and $P^+(z)$ have no common indecomposable direct summand. So the $g$-vector of each direct summand of a presentation in $E_z$ has the same weak sign as $z$ in each coordinate. There is no coordinate cancellation, so
\[
 \|z\|_1=\sum_i m_i\|g(X_i)\|_1+\sum_j b_j\|h_j\|_1.
\]
Division by $d$ proves the first assertion in \eqref{top:rational-control}.

Fix $i$ and let $w\in\{g(X_\ell):\ell\ne i\}\cup\{h_1,\ldots,h_t\}$. Since $X_i$ is presilting, upper semicontinuity gives a nonempty open presilting locus $\mathcal O_i\subseteq E_{g(X_i)}$. By \cite[Theorem~5.5]{AIR}, every point of $\mathcal O_i$ represents $X_i$.

Since \(e(g(X_i),w)=e(w,g(X_i))=0\), the set
\[
U=\left\{(f,f')\in E_{g(X_i)}\times E_w:
 \Hom_{\mathcal K}(T_f,T_{f'}[1])=0=
 \Hom_{\mathcal K}(T_{f'},T_f[1])\right\}
\]
is dense and open. Thus
\[
(\mathcal O_i\times E_w)\cap U\ne\varnothing.
\]
Choose \((f,f')\) in this intersection and  let $Y=T_{f'}$. Since $T_f\cong X_i$, we have
\[
\operatorname{Hom}_{\mathcal K}(X_i,Y[1])=\operatorname{Hom}_{\mathcal K}(Y,X_i[1])=0.
\]
Hence $w=g(Y)\in D_{X_i}$ by Lemma~\ref{top:closedcone}. Since $a_\ell,b_j/d>0$ and $D_{X_i}$ is a convex cone, we have
\begin{align*}
 \theta-a_i g(X_i)
 &=\sum_{\ell\ne i}a_\ell g(X_\ell)+\delta\\
 &=\sum_{\ell\ne i}a_\ell g(X_\ell)
       +\frac1d\sum_jb_jh_j\in D_{X_i}.\qedhere
\end{align*}
\end{proof}

Recall that for a nonzero linear map $L:V\to W$ to a Euclidean space $W$, the map
\begin{equation}\label{top:linear-normalization}
 V\setminus\ker L\longrightarrow S(\operatorname{im}L),
 \qquad v\longmapsto\frac{Lv}{\|Lv\|},
\end{equation}
is a homotopy equivalence, where $S(E)$ denotes the unit sphere in a Euclidean space $E$. Indeed, a linear complement to $\ker L$ identifies the domain with $\ker L\times(\operatorname{im}L\setminus\{0\})$, and the map becomes projection followed by radial normalization.

\begin{theorem}\label{top:kernel}
Let $A$ be a nonzero tame weakly symmetric algebra such that $\{Cg(X):X\in\indtwopsilt A\}$ is bounded in $\mathbb R^n$. Then $R_0(A)=\ker C$. Moreover, the map
\begin{equation}\label{top:normalized-cartan-map}
 |\Delta_\tau(A)|\longrightarrow S(\operatorname{im}_{\mathbb R}C),\qquad x\longmapsto\frac{CG_A(x)}{\|CG_A(x)\|},
\end{equation}
is a homotopy equivalence. In particular,
\[
 |\Delta_\tau(A)|\simeq S^{\operatorname{rank}C-1}.
\]
\end{theorem}

\begin{proof}
First let $\theta\in\ker C$. For each nonzero two-term presilting object $X$, symmetry of $C$ and \eqref{top:cartan-cohomology} give
\[
 \theta^t\mathbf dM_X-\theta^t\mathbf dK_X
       =\theta^tCg(X)=0.
\]
Since $\nu_A$ is an exact autoequivalence, $K_X\cong\nu_AH^{-1}(X)$. If $M_X=K_X=0$, then $H^{-1}(X)=H^0(X)=0$. So $X$ is acyclic and hence contractible, contrary to $X\ne0$. If $\theta\in N_X$, then $\theta(M_X)-\theta(K_X)>0$, contradicting the displayed equality. Thus $\ker C\subseteq R_0(A)$.

For the reverse inclusion, fix $\theta\notin\ker C$ and choose $\theta_m\in\mathbb Q^n$ converging to $\theta$.  Since $C\theta_m\to C\theta$ and $C\theta\ne0$, after omitting finitely many terms we may choose $b>0$ such that $\|C\theta_m\|\geq b$ for all $m$. By the boundedness assumption, choose $L>0$ such that $\|Cg(X)\|\leq L$ for every $X\in\indtwopsilt A$. Applying Lemma~\ref{top:rational-decomposition} to each $\theta_m$, we obtain
\[
 \theta_m=\sum_{i=1}^{s_m}a_{mi}g(X_{mi})+\delta_m \qquad\text{with}\qquad s_m\leq n\quad\text{and}\quad C\delta_m=0.
\]
Then
\[
 b\leq\|C\theta_m\| \leq L\sum_{i=1}^{s_m}a_{mi}.
\]
Since $s_m\leq n$, for each $m$ there is a coefficient $a_m\geq c:=b/(Ln)>0$. Denote the corresponding object by $U_m$. By \eqref{top:rational-control}, we have
\begin{equation}\label{top:subsequence-bound}
 c\|g(U_m)\|_1\leq a_m\|g(U_m)\|_1 \leq\|\theta_m\|_1.
\end{equation}
Since $(\|\theta_m\|_1)$ is bounded and $g(U_m)\in\mathbb Z^n$,  \eqref{top:subsequence-bound} implies that only finitely many vectors $g(U_m)$ occur. At least one of these vectors occurs infinitely often. So there exist indices $m_1<m_2<\cdots$ such that
\[
g(U_{m_k})=g(U_{m_1}) \qquad (k\ge1).
\]
By \cite[Theorem~5.5]{AIR}, we have $U_{m_k}\cong U_{m_1}=:U$ for all $k\ge1$. Replacing $(\theta_m,a_m,U_m)$ by $(\theta_{m_k},a_{m_k},U_{m_k})$ and relabelling the indices, we may assume that $U_m\cong U$ for all $m$. 

By the second assertion in \eqref{top:rational-control}, $\theta_m-a_mg(U)\in D_U$. Since $a_m\geq c$ and $g(U)\in D_U$, we have
\[
 \theta_m-cg(U)=(\theta_m-a_mg(U))+(a_m-c)g(U)\in D_U.
\]
Since $D_U$ is closed and $\theta_m\to\theta$, we have $\theta-cg(U)\in D_U$. Lemma~\ref{top:closedcone}(1) then gives
\[
 \theta\in cg(U)+D_U\subseteq N_U.
\]
So $\theta\notin R_0(A)$, and $R_0(A)=\ker C$ follows.

Since $A\ne0$ and $c_{ii}\geq1$, we have $\operatorname{rank}C\geq1$. By Proposition~\ref{top:nerve} and \eqref{top:linear-normalization} with $L=C$, we have the homotopy equivalences
\[
 |\Delta_\tau(A)|\xrightarrow{G_A}V\setminus\ker C\longrightarrow S(\operatorname{im}_{\mathbb R}C),
\]
whose composition is the map in \eqref{top:normalized-cartan-map}.
\end{proof}


\subsection{The endpoint homotopy equivalence}
\label{top:endpoint-homotopy}

We return to a nonempty punctured partial tagged triangulation $\mathbf R$, and  let $A=\Lambda_{\mathbf R}$ and $n=|\mathbf R|$. Write $B$ for the signed endpoint matrix and $C=C_A=B^tB$. Proposition~\ref{prop:endpoint-formula} gives $Bg_{\mathbf R}(\gamma)=\eta(\gamma)$ for every $\mathbf R$-standard arc~$\gamma$. 

\begin{theorem}\label{top:surface}
 Let $r=\operatorname{rank}C$. Then
\[
 R_0(A)=\ker C=\ker B \qquad\text{and}\qquad |\Delta_\tau(A)|\simeq S^{r-1}.
\]
More precisely, under the identification $\Delta_{\mathbf R}\cong\Delta_\tau(A)$, the map
\begin{equation}\label{top:endpoint-map}
 \Psi_{\mathbf R}:|\Delta_{\mathbf R}| \longrightarrow S(\operatorname{im}_{\mathbb R}B),
 \qquad\sum_i t_i\alpha_i\longmapsto \frac{\sum_i t_i\eta(\alpha_i)}{\bigl\|\sum_i t_i\eta(\alpha_i)\bigr\|}
\end{equation}
is a homotopy equivalence. Here $\{\alpha_1,\ldots,\alpha_s\}$ is a simplex, $t_i\geq0$, and $\sum_i t_i=1$.
\end{theorem}

\begin{proof}
By Proposition~\ref{res:4.1} and Lemma~\ref{top:surface-tame}, $A$ is weakly symmetric and tame. Since $C=B^tB$, Corollary~\ref{res:4.4} gives $\|Bg(U)\|^2=g(U)^tCg(U)\leq4$. Thus
\[
 \|Bg(U)\|\leq2 \qquad\text{and hence}\qquad \|Cg(U)\|\leq2\|B^t\|_{\mathrm{op}}
\]
for every $U\in\indtwopsilt A$. Thus Theorem~\ref{top:kernel} gives $R_0(A)=\ker C$. Since $C=B^tB$, we have $\ker C=\ker B$. By Proposition~\ref{top:nerve} and \eqref{top:linear-normalization} with $L=B$, the composite
\[
 |\Delta_\tau(A)|\xrightarrow{G_A}\mathbb R^n\setminus\ker B \longrightarrow S(\operatorname{im}_{\mathbb R}B)
\]
is a homotopy equivalence. Since $Bg_{\mathbf R}(\alpha)=\eta(\alpha)$, the composition is precisely $\Psi_{\mathbf R}$.
\end{proof}

The complex $\Delta_\tau(A)$ has pure dimension $n-1$, while the sphere has dimension $r-1$. In the $\tau$-tilting finite case, $r=n$ by Theorem~\ref{res:4.7} and the $g$-fan is complete
\cite[Proposition~4.2]{AHIKM}. Since $B$ is then injective, radial normalization of this fan shows that $\Psi_{\mathbf R}$ is a homeomorphism. When $A$ is $\tau$-tilting infinite, the vertex set is an infinite closed discrete subspace of $|\Delta_\tau(A)|$. So $|\Delta_\tau(A)|$ is not compact and cannot be homeomorphic to the sphere.

The same argument applies to Brauer graph algebras with arbitrary positive vertex multiplicities.

\begin{corollary}\label{top:brauer}
Let $G$ be a finite connected Brauer graph with at least one edge, allowing
loops and multiple edges, and let $A_G$ be the Brauer graph algebra of $G$ with
arbitrary positive integer vertex multiplicities. Then
\[
 |\Delta_\tau(A_G)|\simeq
 \begin{cases}
  S^{|V(G)|-2},&G\text{ is bipartite},\\
  S^{|V(G)|-1},&G\text{ is not bipartite}.
 \end{cases}
\]
\end{corollary}

\begin{proof}
Let $B_G$ be the unsigned incidence matrix, with column $e_p+e_q$ for
an edge with endpoints $p,q$ and column $2e_p$ for a loop at $p$.
By \cite[Corollary~6.9(1)]{Asai}, we have $R_0(A_G)=\ker B_G$.
Therefore Proposition~\ref{top:nerve} and \eqref{top:linear-normalization} with $L=B_G$ give
\[
 |\Delta_\tau(A_G)|\simeq
 \mathbb R^{E(G)}\setminus\ker B_G
 \simeq S^{\operatorname{rank}B_G-1}.
\]
Over $\mathbb R$, the equation $B_G^tx=0$ is equivalent to
$x_p+x_q=0$ for each edge with endpoints $p,q$, including $2x_p=0$
for a loop at $p$.  Since $G$ is connected, $x$ is determined by one coordinate. The solutions form a one-dimensional
space when $G$ is bipartite; otherwise an odd cycle forces $x=0$.
Thus $\operatorname{rank}B_G=|V(G)|-1$ if $G$ is bipartite and $\operatorname{rank}B_G=|V(G)|$ otherwise.
\end{proof}

Thus the homotopy type depends only on the number of vertices and
bipartiteness, independently of multiplicities and cyclic orderings.
For a four-cycle with edges listed in cyclic order,
$\ker B_G=\mathbb R(1,-1,1,-1)$ and $\operatorname{rank}B_G=3$.
Since $G$ has four edges, $\Delta_\tau(A_G)$ has pure dimension $3$. By \cite[Corollary~6.9(2)]{Asai}, the complex is infinite. The corollary gives $|\Delta_\tau(A_G)|\simeq S^2$.



\section{Exchange and completion components}\label{sec:5}

 In this section, we classify the exchange components for connected surface rigid algebras. For punctured $\mathbf R$, we use Theorem~\ref{top:surface} and connectedness of nonzero completion graphs; the boundary case follows from \cite[Theorem~5.16]{HZZ}. The formulas for arbitrary $A$ then follow by taking products over the blocks of $A$.

\subsection[The rank-one case]{ The rank-one case}

The full tilting (respectively, silting) mutation graph has isomorphism classes of basic tilting (respectively, silting) complexes as vertices and irreducible mutations as edges.

\Needspace*{9\baselineskip}
\begin{proposition}\label{res:5.2}
Let $A$ be a basic finite-dimensional algebra with $n=|A|\geq2$ and $\operatorname{rank}_{\mathbb Q}C_A=1$.
\begin{enumerate}
    \item Every basic tilting complex $T$ has a sign $\varepsilon(T)\in\{+1,-1\}$ preserved by irreducible tilting mutation, with $\varepsilon(A)=+1$ and $\varepsilon(A[1])=-1$. Thus $A$ and $A[1]$ lie in distinct components of the full tilting mutation graph.
    \item If $A$ is weakly symmetric, then $A$ and $A[1]$ lie in distinct components of the full silting mutation graph. Every basic $\tau$-tilting pair $(M,P)$ has a sign $\varepsilon_A(M,P)\in\{+1,-1\}$ preserved by mutation, with $\varepsilon_A(A,0)=+1$ and $\varepsilon_A(0,A)=-1$. In particular, $(A,0)$ and $(0,A)$ lie in distinct components of
      $\operatorname{EG}(s\tau\text{-}\operatorname{tilt}A)$.
\end{enumerate}
\end{proposition}
\begin{proof}
Fix indecomposable projectives $P_1,\ldots,P_n$ and write $C_A=(c_{ij})$. The entries of $C_A$ are nonnegative and $c_{ii}>0$.  Since $\operatorname{rank}_{\mathbb Q}C_A=1$, $c_{ij}c_{ji}=c_{ii}c_{jj}>0$. So every entry is positive. Set $a_i=c_{i1}$ and $b_j=c_{1j}/c_{11}$. Then $a=(a_i)_{i=1}^n,b=(b_j)_{j=1}^n\in\mathbb Q_{>0}^n$ and $C_A=ab^t$.

(1) Let $T=\bigoplus_{i=1}^nT_i$ be a basic tilting complex, and write $g_i=[T_i]$ in the projective basis $[P_1],\ldots,[P_n]$. The Euler form has matrix $C_A$ in the projective basis. Since $T$ is tilting, $\operatorname{Hom}_{K^b(\operatorname{proj}A)}(T_i,T_j[\ell])=0$ for $\ell\neq0$. Therefore
\[
 \begin{aligned}
 (g_i^ta)(b^tg_j)=g_i^tC_Ag_j
 &=\sum_{\ell\in\mathbb Z}(-1)^\ell
   \dim_k\operatorname{Hom}_{K^b(\operatorname{proj}A)}(T_i,T_j[\ell])\\
 &=\dim_k\operatorname{Hom}_{K^b(\operatorname{proj}A)}(T_i,T_j)
 \geq0.
 \end{aligned}
\]
Taking $i=j$, we have $(g_i^ta)(b^tg_i)>0$. So these two factors are nonzero and have the same sign. The inequalities for all $i,j$ then show that all $g_i^ta$ and $b^tg_i$ have a common sign, which defines $\varepsilon(T)$.

Adjacent basic tilting complexes share $n-1\geq1$ indecomposable summands. Any common summand determines the sign on both sides, so mutation preserves $\varepsilon$. The summands of $A$ and $A[1]$ have classes $e_i$ and $-e_i$, respectively. Since $a_i>0$, the corresponding signs are $+1$ and $-1$.

(2) Suppose $A$ is weakly symmetric. By \cite[Proposition~3.2]{AD}, every silting complex obtained from $A$ by iterated mutation is tilting. A silting mutation path from $A$ to $A[1]$ would therefore contradict (1). Moreover, every two-term silting complex is tilting by \cite[Proposition~3.6]{AD}. For the two-term silting complex $T$ corresponding to $(M,P)$ under \cite[Theorem~3.2]{AIR}, define $\varepsilon_A(M,P)=\varepsilon(T)$. The correspondence preserves mutation and sends $(A,0)$ and $(0,A)$ to $A$ and $A[1]$, respectively, which gives the remaining assertions.
\end{proof}

\subsection{From the complex to the exchange graph}

 Two maximal simplices may intersect without being adjacent in the exchange graph. Connectedness of the completion graphs at their common vertices gives the following comparison.

\begin{lemma}\label{top:exchange-components}
Let $A$ be a finite-dimensional algebra with $n=|A|\geq2$. Suppose that $\mathcal F_A(u)$ is connected for every indecomposable $\tau$-rigid pair $u$. Assigning a $\tau$-tilting pair to the component containing the corresponding maximal simplex induces a bijection
\[
 \pi_0\operatorname{EG}(s\tau\text{-}\operatorname{tilt}A)
 \cong\pi_0|\Delta_\tau(A)|.
\]
\end{lemma}
\begin{proof}
Every maximal simplex has $n$ vertices. Adjacent $\tau$-tilting pairs share $n-1\geq1$ indecomposable summands, so their maximal simplices intersect. The assignment is therefore constant along mutation paths. Every face extends to a maximal simplex by Bongartz completion, which gives surjectivity.

The components of the weak realization are determined by finite edge paths in the $1$-skeleton. Indeed, the union of the simplices in an edge-path component is path connected. Since each simplex lies in a single edge-path component, each such union is both closed and open in the weak topology. Thus distinct edge-path components are separated by open and closed subsets.

Suppose that maximal simplices $\sigma$ and $\sigma'$ lie in the same component. Choose an edge path $u_0,\ldots,u_m$ with $u_0\in\sigma$ and $u_m\in\sigma'$. If $m=0$, connectedness of
$\mathcal F_A(u_0)$ joins the two corresponding $\tau$-tilting pairs. For $m>0$, extend each edge $\{u_i,u_{i+1}\}$ to a maximal simplex $\sigma_i$.  Let $\sigma_{-1}=\sigma$ and $\sigma_m=\sigma'$. Then $u_i\in\sigma_{i-1}\cap\sigma_i$ for $0\leq i\leq m$. Connectedness of $\mathcal F_A(u_i)$ joins the pairs corresponding to $\sigma_{i-1}$ and $\sigma_i$ by a mutation path. Concatenating these paths proves injectivity.
\end{proof}

For a connected surface rigid algebra, Theorem~\ref{res:3.5} supplies the hypothesis of Lemma~\ref{top:exchange-components}.

\begin{remark}\label{rem:one-simple}
The restriction $n\geq2$ in Lemma~\ref{top:exchange-components} is necessary. When $|A|=1$, the zero pair is almost complete and has exactly two completions, related by mutation
\cite[Theorem~2.18 and Definition~2.19]{AIR}. Every indecomposable $\tau$-rigid pair is already a completion. Thus $\Delta_\tau(A)$ consists of two isolated vertices, while the exchange graph of $A$ is $K_2$, the graph with two vertices joined by an edge.
\end{remark}

\subsection[The connected case]{The connected case}

\Needspace*{7\baselineskip}
\begin{theorem}\label{res:5.8}
Let $A=\Lambda_{\mathbf R}$ be a connected surface rigid algebra.
 The exchange graph $\operatorname{EG}(s\tau\text{-}\operatorname{tilt}A)$ is disconnected if and only if $A$ is non-local and weakly symmetric with $\operatorname{rank}_{\mathbb Q}C_A=1$.
When $\operatorname{EG}(s\tau\text{-}\operatorname{tilt}A)$ is disconnected, there are exactly two components, containing $(A,0)$ and $(0,A)$, respectively.  In this case, the underlying arcs of $\mathbf R$ are pairwise distinct with the same endpoint multiset. Moreover, $C_A=2J_{|A|}$ if $|\mathbf P_{\mathbf R}|=2$, and $C_A=4J_{|A|}$ if $|\mathbf P_{\mathbf R}|=1$.

\end{theorem}
\begin{proof}
 If $|A|=1$, the exchange graph is connected by Remark~\ref{rem:one-simple}. Assume $|A|\geq2$. If $\mathbf R$ is connected to the boundary,  the exchange graph of $A$ is connected by \cite[Theorem~5.16]{HZZ}. Otherwise $\mathbf R$ has a puncture component, which is all of $\mathbf R$ by Lemma~\ref{res:3.1}.  Let $r=\operatorname{rank}_{\mathbb Q}C_A$. The Cartan matrix has positive diagonal, so $r\geq1$. Theorem~\ref{res:3.5} gives connectedness of every nonzero completion graph, so Lemma~\ref{top:exchange-components} identifies the exchange components with those of $|\Delta_\tau(A)|$. The homotopy equivalence $|\Delta_\tau(A)|\simeq S^{r-1}$ from
Theorem~\ref{top:surface} then gives
\[
 \#\pi_0\operatorname{EG}(s\tau\text{-}\operatorname{tilt}A)
 =\#\pi_0|\Delta_\tau(A)|
 =\#\pi_0S^{r-1}
 =\begin{cases}
  2,&r=1,\\
  1,&r\geq2.
 \end{cases}
\]
 If the exchange graph is disconnected, then $r=1$. Since all endpoints of $\mathbf R$ are punctures, Proposition~\ref{res:4.1} gives weak symmetry. So Proposition~\ref{res:5.2} places $(A,0)$ and $(0,A)$ in different components, and Corollary~\ref{res:4.2} gives the assertions about $\mathbf R$ and $C_A$. Conversely, if $A$ is non-local and weakly symmetric with $\operatorname{rank}_{\mathbb Q}C_A=1$, then Proposition~\ref{res:5.2} separates $(A,0)$ and $(0,A)$. Thus the exchange graph is disconnected.

\end{proof}

\subsection[Products and completion components]{Products and completion components}

\begin{lemma}\label{res:5.1}
 Let $A=\prod_{i=1}^{t}A_i$ and let $u=\bigoplus_{i=1}^{t}u_i$ be a basic $\tau$-rigid pair, with $u_i$ over $A_i$. Then
\[
 \mathcal F_A(u)\cong\mathop{\square}_{i=1}^{t}\mathcal F_{A_i}(u_i).
\]
In particular, taking $u=0$ gives
\[
 \operatorname{EG}(s\tau\text{-}\operatorname{tilt}A)
 \cong\mathop{\square}_{i=1}^{t}
       \operatorname{EG}(s\tau\text{-}\operatorname{tilt}A_i).
\]
 In both cases, the number of connected components is the product of the numbers for the factors.
\end{lemma}
\begin{proof}
 Every basic $\tau$-tilting pair over $A$ uniquely decomposes as $T=\bigoplus_{i=1}^t T_i$, where $T_i$ is a basic $\tau$-tilting pair over $A_i$.  A completion of $u$ is therefore a direct sum of completions of the $u_i$. Mutation preserving $u$ changes one summand in one factor and preserves the corresponding $u_i$. This gives the Cartesian product.
\end{proof}

 To state the product formulas, we use the following block notation.

\begin{definition}\label{res:5.3}
For a finite-dimensional algebra $A$, let $\mathcal B_C(A)$ be the set of blocks $B$ satisfying
\[
 |B|\geq2,\qquad B\text{ is weakly symmetric},\qquad\text{and}\qquad
 \operatorname{rank}_{\mathbb Q}C_B=1,
\]
 and let $\delta_C(A)=|\mathcal B_C(A)|$.
\end{definition}

The defining block conditions are Morita invariant. Proposition~\ref{res:5.2} and Lemma~\ref{res:5.1} give $\#\pi_0\operatorname{EG}(s\tau\text{-}\operatorname{tilt}A)\geq2^{\delta_C(A)}$ for every finite-dimensional algebra $A$.

 The Cartan matrix of each nonzero block has positive diagonal. So $\operatorname{rank}C_A=1$ implies that $A$ is connected.

 Let $A=\Lambda_{\mathbf R}$ and let $u=(M,P)$ be a basic $\tau$-rigid pair.  For a block $B=e_BA$, let $u_B=(Me_B,Pe_B)$. The central idempotent $e_B$ determines a subcollection $\mathbf R_B\subseteq\mathbf R$, and
\[
 B=e_BAe_B\cong\operatorname{End}_{\mathcal C(\mathbf S)}X(\mathbf R_B)
 =\Lambda_{\mathbf R_B}.
\]
Thus every block is a connected surface rigid algebra. Theorem~\ref{res:5.8} and Lemma~\ref{res:5.1} give the following formulas.

\begin{corollary}\label{res:5.9}\label{res:6.1}
Let $A$ be a surface rigid algebra and let $u$ be a basic $\tau$-rigid pair. Then
\[
 \#\pi_0\operatorname{EG}(s\tau\text{-}\operatorname{tilt}A)=2^{\delta_C(A)},
 \qquad
 \#\pi_0\mathcal F_A(u)=2^{\#\{B\in\mathcal B_C(A):u_B=(0,0)\}}.
\]
 The exchange components are uniquely labelled by $(\varepsilon_B)_{B\in\mathcal B_C(A)}$. For $\mathcal F_A(u)$, only the signs for $u_B=0$ are needed. Moreover, any two mutation-reachable completions of $u$ can be joined by mutations preserving $u$.
\end{corollary}
\begin{proof}
 Write $A=\prod_{i=1}^{t}A_i$ as a product of blocks, with $u_i=u_{A_i}$. By Lemma~\ref{res:5.1},
\[
 \mathcal F_A(u)\cong\mathop{\square}_{i=1}^{t}\mathcal F_{A_i}(u_i).
\]
 If $u_i\neq0$, then $\mathcal F_{A_i}(u_i)$ is connected by Theorem~\ref{res:3.5}. If $u_i=0$, then $\mathcal F_{A_i}(u_i)$ is the full exchange graph of $A_i$.  By Theorem~\ref{res:5.8}, the exchange graph of $A_i$ has two components, distinguished by $\varepsilon_{A_i}$, when $A_i\in\mathcal B_C(A)$, and is connected otherwise. Taking products proves the second formula and the sign description. Setting $u=0$ gives the first formula.

Now let $T,T^\prime$ be mutation-reachable completions of $u$. Write $T=\bigoplus_iT_i$ and $T^\prime=\bigoplus_iT_i^\prime$. Then $T_i$ and $T_i^\prime$ are mutation-reachable for every $i$. If $u_i\neq0$, connectedness of $\mathcal F_{A_i}(u_i)$ joins $T_i$ and $T_i^\prime$ while preserving $u_i$. If $u_i=0$, any mutation path in that block preserves $u_i$. Applying these paths one block at a time joins $T$ and $T^\prime$ while preserving $u$.
\end{proof}

\begin{remark}\label{rem:small-completions}
For any finite-dimensional algebra $A$, a basic $\tau$-rigid pair with $|A|$ summands is the unique vertex of the corresponding completion graph. A pair with $|A|-1$ summands has exactly two completions, related by mutation \cite[Theorem~2.18 and Definition~2.19]{AIR}.  Thus $\mathcal F_A(u)$ can be disconnected only if $|u|\leq|A|-2$. The following example shows that this bound is sharp.
\end{remark}

\begin{example}\label{res:6.6}\label{subsec:sparse-digon}
Let $\mathbf S_{\mathrm{dig}}$ be an annulus with one marked point on each boundary component and two punctures $p,q$. Choose two plain-tagged arcs $e_1,e_2$ from $p$ to $q$ which are disjoint away from their endpoints and whose union is an essential simple closed curve separating the boundary components.  Let $\mathbf R_{\mathrm{dig}}=\{e_1,e_2\}$ and
$A=\Lambda_{\mathbf R_{\mathrm{dig}}}$, see Figure~\ref{fig:6}. Proposition~\ref{res:4.1} gives
\[
 C_A
 =\begin{pmatrix}2&2\\2&2\end{pmatrix}.
\]
The nonzero off-diagonal entry shows that $A$ is connected. Since $\operatorname{rank}C_A=1$, Theorem~\ref{res:5.8} gives exactly two exchange components, containing $(A,0)$ and $(0,A)$, respectively. Moreover, Theorem~\ref{top:surface} gives $|\Delta_\tau(A)|\simeq S^0$, so both components of $|\Delta_\tau(A)|$ are contractible.  Taking $u=0$ gives a disconnected completion graph with $|u|=|A|-2$. The digon also gives a counterexample to the unrestricted first assertion of \cite[Corollary~5.18]{HZZ}.

\begin{figure}[htbp]
\centering
\includegraphics[width=100pt]{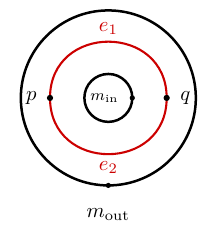}
\caption{The digon $\mathbf R_{\mathrm{dig}}$ with boundary marked points
$m_{\mathrm{in}}$ and $m_{\mathrm{out}}$.}
\label{fig:6}
\end{figure}
\end{example}


\Needspace*{10\baselineskip}
\section{Finite endpoint ray sets}\label{sec:7}

Let $\mathbf R=\{\gamma_1,\ldots,\gamma_n\}$ be a nonempty punctured partial tagged triangulation, and suppose that $A=\Lambda_{\mathbf R}$ is $\tau$-tilting finite. Write $B$ for the signed endpoint matrix and $C=B^tB$. By Proposition~\ref{prop:endpoint-formula}, we have $Bg_{\mathbf R}(\gamma)=\eta(\gamma)$ for every $\mathbf R$-standard arc $\gamma$. Thus every endpoint vector has one of the forms $ae_p+be_q$ for $p\neq q$ and $2ae_p$, with $a,b\in\{\pm1\}$. Vectors of the first form are called \emph{short}, and vectors of the second form are called \emph{long}.

\subsection{Short roots}
Let $V(A)$ be the finite set of isomorphism classes of indecomposable $\tau$-rigid pairs. By Theorem~\ref{res:4.7}, $C$ is positive definite. So $B$ has full column rank and hence identifies the projective lattice $\mathbb Z^n$ with $B\mathbb Z^n$. We determine the vectors $Bg(u)$ for $u\in V(A)$, with primitivity understood in this image lattice. Thus $x\in B\mathbb Z^n\setminus\{0\}$ is primitive if $x=my$ with $y\in B\mathbb Z^n$ and $m\in\mathbb Z_{>0}$ implies $m=1$.

Since $z^tCz=\|Bz\|^2$, the following lemma shows that every short vector in $B\mathbb Z^n$ occurs as $Bg(u)$.

\begin{lemma}\label{res:7.2}
Every $z\in\mathbb Z^n$ satisfying $z^tCz=2$ is the $g$-vector of a unique indecomposable $\tau$-rigid pair.
\end{lemma}
\begin{proof}
By Theorem~\ref{res:4.7}, the $g$-fan $\Sigma_g(A)$ is complete. Choose a maximal cone containing $z$, and let $u=u_1\oplus\cdots\oplus u_n$ be the associated basic $\tau$-tilting pair.  Let $g_i=g(u_i)$. Then
\[
 z=\sum_{i=1}^{n}a_i g_i\qquad\text{with }a_i\in\mathbb R_{\geq0}.
\]
By \cite[Theorem~5.1]{AIR}, the vectors $g_1,\ldots,g_n$ form a $\mathbb Z$-basis of $\mathbb Z^n$. Since $z\in\mathbb Z^n$, we have $a_i\in\mathbb Z_{\ge0}$ for every $i$. By Corollary~\ref{res:4.4}, we have
\[
 2=z^tCz
 =\sum_i a_i^2 g_i^tCg_i+2\sum_{i<j}a_i a_jg_i^tCg_j
 \geq2\sum_i a_i^2.
\]
Since $z\neq0$ and $a_i\in\mathbb Z_{\ge0}$, we have $\sum_i a_i^2=1$. So one $a_i$ equals $1$ and all the others equal $0$. Hence $z=g(u_i)$ for one $i$. Uniqueness follows from \cite[Theorem~5.5]{AIR}.
\end{proof}

For connected $\Gamma_{\mathbf R}$, Theorem~\ref{res:4.7} shows that $\Gamma_{\mathbf R}$ is either a tree or unbalanced and unicyclic. Fix an ordering of $\mathbf P_{\mathbf R}$ for the notation $p<q$.

\begin{theorem}\label{res:7.3}
Suppose $\Gamma_{\mathbf R}$ is a tree with $n$ edges. Choose $\varepsilon\in\{\pm1\}^{n+1}$ spanning $\ker B^t$. Then the map
\[
 V(A)\longrightarrow \{\,\pm(\varepsilon_pe_p-\varepsilon_qe_q):p<q\,\}, \qquad u\longmapsto Bg(u),
\]
is a bijection. Consequently, $|V(A)|=n(n+1)$, and every $\mathbf R$-standard arc has distinct endpoints.
\end{theorem}
\begin{proof}
For an $\mathbf R$-standard arc $\gamma$, Proposition~\ref{prop:endpoint-formula} gives $\eta(\gamma)=Bg_{\mathbf R}(\gamma)$, hence $\varepsilon^t\eta(\gamma)=0$. If $\eta(\gamma)=2ae_p$ is a long vector, then $\varepsilon^t\eta(\gamma)=2a\varepsilon_p\neq0$, a contradiction. Thus $\gamma$ has distinct endpoints. If $\eta(\gamma)=ae_p+be_q$ is a short vector, then $\varepsilon^t\eta(\gamma)=a\varepsilon_p+b\varepsilon_q=0$. So $\eta(\gamma)\in\{\pm(\varepsilon_pe_p-\varepsilon_qe_q):p<q\}$.

 Let $D=\operatorname{diag}(\varepsilon_1,\ldots,\varepsilon_{n+1})$. Since $B^t\varepsilon=0$, a column of $DB$ corresponding to an edge $uv$ is $\pm(e_u-e_v)$. Fix $p<q$. Along the unique path $p=v_0,\ldots,v_m=q$, choose signs $\sigma_j\in\{\pm1\}$ for the corresponding columns $b_j$ of $B$ so that
\[
 D\sum_{j=1}^m\sigma_jb_j
 =\sum_{j=1}^m(e_{v_{j-1}}-e_{v_j})=e_p-e_q.
\]
Since $D^2=I$, the integral combination $\sum_j\sigma_jb_j$ equals $\varepsilon_pe_p-\varepsilon_qe_q$.  Negating this combination gives the opposite vector. Thus $\pm(\varepsilon_pe_p-\varepsilon_qe_q)\in B\mathbb Z^n$. Since $B$ has independent columns, each integral preimage $z$ of $\pm(\varepsilon_pe_p-\varepsilon_qe_q)$ is unique. Since $z^tCz=\|Bz\|^2=2$, Lemma~\ref{res:7.2} gives $z=g(u)$ for a unique $u\in V(A)$. So the bijection follows. There are $\binom{n+1}{2}$ choices of $p<q$ and two signs for each choice. Thus $|V(A)|=2\binom{n+1}{2}=n(n+1)$.
\end{proof}

The unique cycle of a unicyclic graph is called the \emph{core}.
Let $L_{\mathbf R}$ be the set of long vectors occurring among $Bg(V(A))$.

\begin{theorem}\label{res:7.4}
Suppose $\Gamma_{\mathbf R}$ is unbalanced and unicyclic, with $n$ edges and $n$ vertices. Then $B$ is invertible and
\[
 B\mathbb Z^n=\{x\in\mathbb Z^n:\textstyle\sum_p x_p\text{ is even}\}.
\]
Moreover,
\[
 Bg(V(A))=\{ae_p+be_q:p<q,\ a,b\in\{\pm1\}\}\cup L_{\mathbf R}.
\]
\end{theorem}
\begin{proof}
The row of $B$ indexed by a leaf has a single nonzero entry, equal to $\pm1$. Expanding along these rows removes the attached trees without changing $|\det B|$. The remaining core is unbalanced, so the product of its edge signs is $-1$. Row and column sign changes therefore give
\[
 |\det B|=\left|1-\prod_{e\text{ in the core}}\operatorname{sign}(e)\right|=2.
\]
For a loop the remaining matrix is $(\pm2)$, giving the same value. Every column has even coordinate sum. The image $B\mathbb Z^n$ is therefore contained in the parity lattice in the statement. Since $[\mathbb Z^n:B\mathbb Z^n]=|\det B|=2$ and the parity lattice also has index $2$, the two lattices are equal.

Every short vector $x=ae_p+be_q$ belongs to $B\mathbb Z^n$ and has squared length two. The unique integral preimage $z$ of $x$ satisfies $z^tCz=\|x\|^2=2$. By Lemma~\ref{res:7.2}, $z=g(u)$ for a unique $u\in V(A)$. So $Bg(u)=x$. By Proposition~\ref{prop:endpoint-formula}, all remaining endpoint vectors are long.
\end{proof}

\subsection{The surface criterion}
\label{subsec:surface-charge-criterion}
 Suppose that $\Gamma_{\mathbf R}$ is unbalanced and unicyclic. The core is \emph{exceptional} if its arcs form a conjugate pair or bound an embedded unpunctured disk triangle.
 Let $D_n=\{\pm e_p\pm e_q:p<q\}$ and
$C_n=D_n\cup\{\pm2e_p:p\in\mathbf P_{\mathbf R}\}$.
We use the low-rank conventions $D_2=A_1\times A_1$ and $D_3=A_3$;
a single loop will give $C_1$.

By Theorem~\ref{res:7.4}, only the long vectors remain to be determined.
For $p\in\mathbf P_{\mathbf R}$ and $a\in\{\pm1\}$, Theorem~\ref{res:7.4} gives $z=B^{-1}(2ae_p)\in\mathbb Z^n$ and $z^tCz=4$. Hence Lemma~\ref{res:7.2}, which requires $z^tCz=2$, does not determine whether $2ae_p\in Bg(V(A))$.

\Needspace*{8\baselineskip}
\begin{theorem}\label{res:7.5}
Suppose $\Gamma_{\mathbf R}$ is unbalanced and unicyclic. Then
\[
 L_{\mathbf R}=
 \begin{cases}
 \varnothing,&\text{if the core is exceptional},\\
 \{\pm2e_p:p\in\mathbf P_{\mathbf R}\},&\text{otherwise}.
 \end{cases}
\]
Consequently, $Bg(V(A))=D_n$ if the core is exceptional and $Bg(V(A))=C_n$ otherwise. Thus $|V(A)|=2n(n-1)$ and $2n^2$, respectively.
\end{theorem}

Consider two embeddings of a three-edge cycle whose vertices are punctures and whose edges have plain tags. In the first, the cycle bounds an unpunctured disk triangle. In the second, it is essential in an annulus.  Choose a marked point on each boundary component.
Numbering the arcs in cyclic order gives, in both cases,
\[
 B=\begin{pmatrix}1&0&1\\1&1&0\\0&1&1\end{pmatrix},\qquad
 C=B^tB=I_3+J_3,\qquad B^{-1}(2e_1)=(1,-1,1)^t.
\]
Theorem~\ref{res:7.5} gives $Bg(V(A))=D_3$ for the disk triangle and $Bg(V(A))=C_3$ for the essential cycle, with $|V(A)|=12$ and $18$, respectively. Thus $z=(1,-1,1)^t$ satisfies $z^tCz=4$ in both cases, but $z\in g(V(A))$ only in the second. The endpoint graph does not retain this distinction between the two embeddings.

For a nonexceptional core, we use the ray count for Brauer graph algebras. For an exceptional core, we restrict to the core and one attached path, and show that the presentations of each long-vector index are generically decomposable. The open-orbit argument then excludes such an index from $g(V(A))$.

\subsection{Comparison with Brauer graph algebras}
\label{subsec:partial-triangulation-algebras}

Let $\mathbf Q$ be a nonempty punctured partial tagged triangulation with distinct underlying arcs. Let $G_{\mathbf Q}$ be the ribbon graph
whose vertices are the incident punctures and whose edges are the underlying
arcs, with cyclic orders induced by the surface. Write $A(G_{\mathbf Q})$
for  the multiplicity-one Brauer graph algebra of $G_{\mathbf Q}$.
We call $\mathbf Q^\circ$ \emph{sparse} if it contains no unpunctured disk
triangle \cite[Definition~6.1]{D}. Here the sides of a disk triangle may be identified, and  the image of the interior of the disk need not be a complementary region of $\mathbf Q^\circ$.
See Appendix~\ref{app:comparison-proof} for the polygon convention and the algebra presentations.

\begin{proposition}\label{res:4.9}
If $\mathbf Q^\circ$ is sparse, there is an isomorphism
\[
 \Lambda_{\mathbf Q}\cong A(G_{\mathbf Q})^{\mathrm{op}}
\]
preserving the idempotents indexed by $\mathbf Q$.
\end{proposition}
Appendix~\ref{app:comparison-proof} also gives the comparison with
Demonet's algebra in the nonsparse case, needed for the disk-triangle core.


\subsection{Cores and stems}

Throughout this subsection, assume that $\Gamma_{\mathbf R}$ is unbalanced and unicyclic. The next lemma determines when Proposition~\ref{res:4.9} applies and where trees can attach to an exceptional core.

\begin{lemma}\label{res:7.6}
The core and the attached trees satisfy the following properties.
\begin{enumerate}
\item If $\mathbf R$ contains a conjugate pair, the pair is the core. Every attached tree meets the core at the endpoint where the pair has the same tagging.
\item If $\mathbf R$ has no conjugate pair, the core has odd length, and $\mathbf R^\circ$ is sparse if and only if the core does not bound an embedded unpunctured disk triangle. If the core bounds such a disk, every attached tree lies outside the disk.
\end{enumerate}
\end{lemma}
\begin{proof}
(1) The assertion follows from unicyclicity and Remark~\ref{rem:conjugate-endpoints}. 

(2) Since the underlying arcs are pairwise distinct, compatibility gives a common tagging at each puncture. The sign product around a cycle of length $\ell$ is therefore $(-1)^\ell$, as in
Theorem~\ref{res:4.7}. Since the core is unbalanced, $(-1)^\ell=-1$. Hence $\ell$ is odd. 

By the definition of sparsity, it remains to determine the unpunctured disk triangles with sides in $\mathbf R^\circ$. Let $P:\Delta\to S$ be such a triangle. By \cite[Proposition~2.10]{D}, the images under $P$ of the three oriented sides of $\Delta$ are pairwise distinct. Since each arc in $\mathbf R^\circ$ has two orientations, these images involve two or three underlying arcs.

Suppose the three sides of $\Delta$ map to two distinct underlying arcs $\epsilon$ and $\delta$, with $\epsilon$ occurring twice. Since the oriented images are distinct, the two occurrences of $\epsilon$ have opposite orientations. After cyclically relabelling the sides, the boundary walk is $\delta,\epsilon,\epsilon^{-1}$, so $\delta$ is a loop. By unicyclicity, $\epsilon$ is not a loop, so it joins the basepoint of $\delta$ to another puncture $q$. Identifying the two copies of $\epsilon$ gives a disk bounded by $\delta$, with $q$ in the interior, see Figure~\ref{fig:core-stem}(a). Since $P(\operatorname{int}\Delta)$ contains no puncture, no other puncture lies in the disk. Thus $\delta$ cuts out a once-punctured monogon, which is excluded for the underlying arc of a tagged arc.

Suppose the three sides of $\Delta$ map to three distinct underlying arcs. Then they form a closed walk of length three in $\Gamma_{\mathbf R}$. If exactly two of the three vertices coincide, the walk contains a loop and a two-edge cycle; if all three coincide, it contains three loops. Both contradict unicyclicity. Hence the vertices are distinct and the three edges form the core. No boundary sides or vertices of $\Delta$ are identified by $P$. Hence $P$ is an embedding, and its image is an unpunctured disk triangle bounded by the core. Conversely, such a disk violates sparsity.
All arc endpoints are punctures, and the disk has no puncture in its interior.  Suppose an attached tree has an edge inside the disk. Both endpoints of that edge lie on the core, giving a second cycle, a contradiction. Hence every attached tree lies outside the disk, as in Figure~\ref{fig:core-stem}(b).
\end{proof}

\begin{proof}[Proof of Theorem~\ref{res:7.5} for a nonexceptional core]
By Lemma~\ref{res:7.6}, a nonexceptional core has odd length, $\mathbf R$ contains no conjugate pair, and $\mathbf R^\circ$ is sparse. Proposition~\ref{res:4.9} identifies $A$, up to
opposite, with the multiplicity-one Brauer graph algebra of an odd unicyclic graph. Reversing all cyclic orders gives the opposite algebra and leaves the graph unchanged. By \cite[Theorem~8.4(b)]{AHIKM}, $|V(A)|=2n^2$. By Theorem~\ref{res:7.4}, we have $2n(n-1)$ short vectors. So $|L_{\mathbf R}|=2n$. Since there are exactly $2n$ vectors of the form $\pm2e_p$, all of them occur.
\end{proof}

\begin{figure}[htbp]
 \centering
 \includegraphics[width=.55\linewidth]{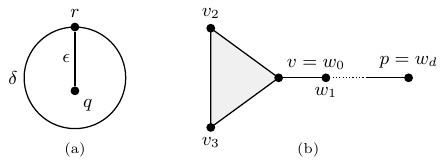}
 \caption{ (a) A forbidden monogon. (b) A triangle core with a stem.}
 \label{fig:core-stem}
\end{figure}

It remains to exclude long vectors for the exceptional cores. For $p\in\mathbf P_{\mathbf R}$, the unique path in $\Gamma_{\mathbf R}$ from $p$ to the core is called the \emph{stem}. It is empty when $p$ lies on the core. Let $\mathbf I_p\subseteq\mathbf R$ consist of the core arcs and the arcs corresponding to the edges of the stem. For $a\in\{\pm1\}$, we shall test whether $2ae_p\in L_{\mathbf R}$ using the algebra $\Lambda_{\mathbf I_p}$.

\begin{lemma}\label{res:7.7}
Let $p\in\mathbf P_{\mathbf R}$, $a\in\{\pm1\}$ and $z=B^{-1}(2ae_p)$.
\begin{enumerate}
\item  $\operatorname{supp}z=\mathbf I_p$. Moreover, $|z_\gamma|=1$ on the core and $|z_\gamma|=2$ on the stem.
\item We have
\[
 2ae_p\in L_{\mathbf R}
 \quad\Longleftrightarrow\quad 2ae_p\in L_{\mathbf I_p}.
\]
\end{enumerate}
\end{lemma}
\Needspace*{5\baselineskip}
\begin{proof}
For (1), let $q\neq p$ be a leaf with incident edge $\delta$. The $q$-th row of $B$ is $be_\delta^t$ for some $b\in\{\pm1\}$. Since $q\neq p$, the $q$-th coordinate of $Bz=2ae_p$ gives
\[
 bz_\delta=be_\delta^tz=0.
\]
So $z_\delta=0$. Successively deleting such leaves and their incident edges removes all edges outside $\mathbf I_p$. Thus $z_\gamma=0$ for $\gamma\notin\mathbf I_p$.

If the stem is nonempty, write its successive edges as $\delta_1,\ldots,\delta_s$, starting at $p$. Since all coordinates outside $\mathbf I_p$ are zero, the equation at $p$ gives $cz_{\delta_1}=2a$ for some $c\in\{\pm1\}$. Hence $|z_{\delta_1}|=2$. For $1\leq i<s$, the equation at the common vertex of $\delta_i$ and $\delta_{i+1}$ gives
\[
 c_1z_{\delta_i}+c_2z_{\delta_{i+1}}=0,
 \qquad c_1,c_2\in\{\pm1\}.
\]
So $|z_{\delta_{i+1}}|=|z_{\delta_i}|$. Thus $|z_{\delta_i}|=2$ for $1\leq i\leq s$.

Let $v$ be the attachment vertex, or $p$ if the stem is empty.  Define
\[
 \xi=\begin{cases}
 2a,&\text{if the stem is empty},\\
 -c_0z_{\delta_s},&\text{otherwise},
 \end{cases}
\]
where $c_0=B_{v,\delta_s}\in\{\pm1\}$ in the second case. Since $|z_{\delta_s}|=2$ for a nonempty stem, we have $\xi\in\{\pm2\}$.

Suppose first that the core has $\ell\geq2$ edges. Write them in cyclic order as $\epsilon_1,\ldots,\epsilon_\ell$, with $\epsilon_1$ and $\epsilon_\ell$ incident to $v$. For $1\leq j<\ell$, the equation at the common vertex of $\epsilon_j$ and $\epsilon_{j+1}$ other than $v$ gives
\[
 a_jz_{\epsilon_j}+b_jz_{\epsilon_{j+1}}=0,
 \qquad a_j,b_j\in\{\pm1\}.
\]
 Let $t=z_{\epsilon_1}$. These equations give $z_{\epsilon_j}=\sigma_jt$ with $\sigma_j\in\{\pm1\}$ and $\sigma_1=1$. The equation at $v$ is therefore
\[
 c_1z_{\epsilon_1}+c_2z_{\epsilon_\ell}=(c_1+c_2\sigma_\ell)t=\xi,
 \qquad c_1,c_2\in\{\pm1\}.
\]
Since $\xi\neq0$, we have $c_1+c_2\sigma_\ell\neq0$, so $|c_1+c_2\sigma_\ell|=2$. Hence $|t|=1$, and $|z_{\epsilon_j}|=1$ for every $j$. This includes a two-edge core. If the core is a loop $\epsilon$, its equation is $\pm2z_\epsilon=\xi$, which again gives $|z_\epsilon|=1$.

For (2),  let $e=\sum_{\gamma_i\in\mathbf I_p}e_i$.
By (1), $z$ is supported on $\mathbf I_p$ and has a coordinate of absolute value one. Hence $z$ is primitive in the projective lattice, as is its restriction $z|_{\mathbf I_p}$.
Under the identification $eAe\cong\Lambda_{\mathbf I_p}$,
\cite[Theorem~4.18]{AHIKM} identifies $\Sigma_g(eAe)$ with the subfan of
$\Sigma_g(A)$ in the subspace spanned by $[P_i]$ with
$\gamma_i\in\mathbf I_p$.
Thus $z$ generates a ray of $\Sigma_g(A)$ if and only if
$z|_{\mathbf I_p}$ generates a ray of $\Sigma_g(eAe)$.
Both vectors are primitive, so they generate rays precisely when they are $g$-vectors of indecomposable $\tau$-rigid pairs. By (1), $B_{\mathbf I_p}(z|_{\mathbf I_p})=2ae_p$. So Proposition~\ref{prop:endpoint-formula} gives (2).
\end{proof}

\Needspace*{9\baselineskip}
\subsection{Exceptional cores}
\label{subsec:local-splitting}
By Lemma~\ref{res:7.7}, it suffices to consider an exceptional core with one stem, possibly empty. By Proposition~\ref{prop:endpoint-formula} and the correspondence in \cite{AIR}, a long vector $2ae_p$ occurs if and only if there is $T\in\indtwopsilt A$ with $[T]=B^{-1}(2ae_p)$. We show that a nonempty open subset of the corresponding presentation space consists of decomposable complexes.

For $z\in\mathbb Z^n$, use the projectives $P^\pm(z)$, the presentation space $E_z$, and the complex $T_d$ from \eqref{top:presentation-space}.

Let $T=(T^{-1}\xrightarrow dT^0)$ be a minimal two-term presilting complex of class $z$. Since $T$ is minimal, $d\in\operatorname{rad}(T^{-1},T^0)$. Thus \cite[Proposition~2.5]{AIR} gives
\[
 \operatorname{Hom}_A(T^{-1},T^0)
 =\operatorname{End}_A(T^0)d+d\operatorname{End}_A(T^{-1})
 \subseteq\operatorname{rad}(T^{-1},T^0).
\]
 Suppose that $T^{-1}$ and $T^0$ have a common indecomposable projective summand. The identity on that summand gives a map outside $\operatorname{rad}(T^{-1},T^0)$, a contradiction. Hence the two terms have no common indecomposable summand. Since $[T]=z$, we have $T^{-1}\cong P^-(z)$ and $T^0\cong P^+(z)$.

The following lemma then excludes an indecomposable presilting complex of class $z$.

\Needspace*{7\baselineskip}
\begin{lemma}\label{res:7.8}
Let $G_z=\operatorname{Aut}_A(P^+(z))\times \operatorname{Aut}_A(P^-(z))$ act on $E_z$ by $(g^+,g^-)\cdot d=g^+d(g^-)^{-1}$.
\begin{enumerate}
\item If $T_d$ is presilting, the orbit $G_z\cdot d$ is open in $E_z$.
\item If every differential in a nonempty open subset $U\subseteq E_z$ defines a decomposable complex, then no $T\in\indtwopsilt A$ has class $z$.
\end{enumerate}
\end{lemma}
\begin{proof}
For (1), set $P^\pm=P^\pm(z)$ and define the linear map
\[
 \Phi_d:\operatorname{End}_A(P^+)\oplus\operatorname{End}_A(P^-)
 \longrightarrow E_z,\qquad (f^+,f^-)\longmapsto f^+d-df^-.
\]
The stabilizer of $d$ is $G_z\cap\ker\Phi_d$, which is an open subset of the vector space $\ker\Phi_d$. Thus it is smooth in every characteristic, and the orbit dimension formula gives \mbox{$\dim(G_z\cdot d)=\operatorname{rank}\Phi_d$}. The description of chain homotopies in \cite[Section~2 and Definition~4.6]{DF} gives
\[
 \operatorname{coker}\Phi_d
 =\frac{\operatorname{Hom}_A(P^-,P^+)}
 {\operatorname{End}_A(P^+)d+d\operatorname{End}_A(P^-)}
 \cong\operatorname{Hom}_{K^b}(T_d,T_d[1]).
\]
Since $T_d$ is presilting, $\operatorname{coker}\Phi_d=0$. So $\dim(G_z\cdot d)=\dim E_z$. The orbit is locally closed and $E_z$ is irreducible. Hence $G_z\cdot d$ is open in $E_z$.

For (2), suppose that $T\in\indtwopsilt A$ has class $z$.  Since $T$ is presilting, a minimal representative of $T$ has terms $P^-(z)$ and $P^+(z)$. Thus it is $T_d$ for some $d\in E_z$. By (1), $G_z\cdot d$ is a nonempty open subset of $E_z$. Since $E_z$ is irreducible, $(G_z\cdot d)\cap U\neq\varnothing$. A differential in this intersection defines both a complex isomorphic to $T_d$ and a decomposable complex, a contradiction.
\end{proof}

In the calculations below, the open conditions are the nonvanishing of specified scalars and determinants. In the local presentations of $A$, paths are read from left to right. An arrow $a:i\longrightarrow j$ acts as $f_a:P_j\longrightarrow P_i$, $x\longmapsto ax$. For arrows $a:i\longrightarrow j$ and $b:j\longrightarrow k$, the path $ab$ therefore acts as $f_{ab}=f_a\circ f_b:P_k\longrightarrow P_i$. We replace $B$ and each endpoint vector $x$ by $DB$ and $Dx$, where $D$ is diagonal with entries in $\{\pm1\}$. This is a change of endpoint coordinates, so $A$ and the projective classes are unchanged. Since $D^tD=I$, we also have $(DB)^t(DB)=C$. We choose $D$ so that all columns are $e_x+e_y$, except that a conjugate pair has columns $e_v+e_q$ and $e_v-e_q$.

\begin{lemma}\label{lem:conjugate-stem-splitting}
Suppose that $\mathbf R$ consists of a conjugate pair $a,b$, with columns $e_v+e_q,e_v-e_q$, and a stem $v=v_0,v_1,\ldots,v_d$. Let $P_a,P_b$ denote the core projectives and let $Q_i$ be the projective indexed by $v_{i-1}v_i$, with $B[Q_i]=e_{v_{i-1}}+e_{v_i}$. Set $S_d=\sum_{i=1}^d(-1)^{d-i}[Q_i]$. Then
\begin{equation}\label{eq:7.1}
 B^{-1}(2e_{v_d})=z_d=(-1)^d([P_a]+[P_b])+2S_d.
\end{equation}
For each $z\in\{z_d,-z_d\}$, a nonempty open subset of $E_z$ consists of differentials defining decomposable complexes. For $d=0$, the indices $\pm([P_a]-[P_b])$ also give only decomposable
minimal complexes.
\end{lemma}
\begin{proof}
Since $B([P_a]+[P_b])=2e_{v_0}$ and
\[
 BS_d=\sum_{i=1}^d(-1)^{d-i}(e_{v_{i-1}}+e_{v_i})
     =e_{v_d}-(-1)^de_{v_0},
\]
we have $B((-1)^d([P_a]+[P_b])+2S_d)=2e_{v_d}$, which proves
\eqref{eq:7.1}.

For $d\geq1$, Proposition~\ref{res:4.1} gives one-dimensional Hom-spaces in both directions along
\[
 P_a\longleftrightarrow Q_1\longleftrightarrow Q_2
 \longleftrightarrow\cdots\longleftrightarrow Q_d,
 \qquad P_b\longleftrightarrow Q_1,
\]
and all other off-diagonal Hom-spaces vanish. The core part of a differential of index $z_d$ is
\[
 \begin{cases}
  P_a\oplus P_b\longrightarrow Q_1^{\oplus2},&d\text{ odd},\\
  Q_1^{\oplus2}\longrightarrow P_a\oplus P_b,&d\text{ even}.
 \end{cases}
\]
 For $x\in\{a,b\}$, choose nonzero maps
$f_{x1}\in\Hom_A(P_x,Q_1)$ and $f_{1x}\in\Hom_A(Q_1,P_x)$.
For $1\leq i<d$, choose nonzero maps
$f_i\in\Hom_A(Q_i,Q_{i+1})$ and
$g_i\in\Hom_A(Q_{i+1},Q_i)$.
Since the coefficient of $[Q_i]$ in $z_d$ is $2(-1)^{d-i}$,
the block between $Q_i^{\oplus2}$ and $Q_{i+1}^{\oplus2}$ has the form
\[
 \begin{cases}
 \begin{pmatrix}
 \lambda^i_{11}f_i&\lambda^i_{12}f_i\\
 \lambda^i_{21}f_i&\lambda^i_{22}f_i
 \end{pmatrix}:Q_i^{\oplus2}\longrightarrow Q_{i+1}^{\oplus2},
 &d-i\text{ odd},\\[6pt]
 \begin{pmatrix}
 \lambda^i_{11}g_i&\lambda^i_{12}g_i\\
 \lambda^i_{21}g_i&\lambda^i_{22}g_i
 \end{pmatrix}:Q_{i+1}^{\oplus2}\longrightarrow Q_i^{\oplus2},
 &d-i\text{ even}.
 \end{cases}
\]
The core block has the form
\[
 \begin{cases}
 \begin{pmatrix}
 \lambda_{11}f_{a1}&\lambda_{12}f_{b1}\\
 \lambda_{21}f_{a1}&\lambda_{22}f_{b1}
 \end{pmatrix}:P_a\oplus P_b\longrightarrow Q_1^{\oplus2},
 &d\text{ odd},\\[6pt]
 \begin{pmatrix}
 \mu_{11}f_{1a}&\mu_{12}f_{1a}\\
 \mu_{21}f_{1b}&\mu_{22}f_{1b}
 \end{pmatrix}:Q_1^{\oplus2}\longrightarrow P_a\oplus P_b,
 &d\text{ even}.
 \end{cases}
\]
 Let
\[
 M_i=\begin{pmatrix}\lambda^i_{11}&\lambda^i_{12}\\
                      \lambda^i_{21}&\lambda^i_{22}\end{pmatrix}
 \quad(1\leq i<d),\qquad
 M_0=\begin{cases}
 (\lambda_{rs})_{r,s=1}^2,&d\text{ odd},\\
 (\mu_{rs})_{r,s=1}^2,&d\text{ even}.
 \end{cases}
\]
The entries of $M_0,\ldots,M_{d-1}$ are independent coordinates on
$E_{z_d}$. Taking $M_i=I_2$ for $0\leq i<d$ makes their determinants
simultaneously nonzero. Thus the conditions $\det M_i\neq0$,
$0\leq i<d$, define a nonempty open subset $U\subseteq E_{z_d}$.

For a differential in $U$, an element of $\operatorname{GL}_2(k)$ acting on the two copies of $Q_1$ makes $M_0=I_2$. For $1\leq i<d$, keep the basis of $Q_i^{\oplus2}$ fixed and change that of $Q_{i+1}^{\oplus2}$ to make $M_i=I_2$. The first copies of the $Q_i$ and $P_a$ now form one subcomplex; the second copies and $P_b$ form another. Hence the complex is a direct sum of two nonzero subcomplexes.
For $-z_d$, source and target are exchanged, and the same basis changes
apply to the reversed Hom-spaces.

For $d=0$, the indices $\pm z_0=\pm([P_a]+[P_b])$ give direct sums
of stalk complexes. The indices $\pm([P_a]-[P_b])$ do so as well,
since $\operatorname{Hom}(P_a,P_b)=\operatorname{Hom}(P_b,P_a)=0$
forces their differentials to vanish.
\end{proof}

For a disk-triangle core, we first treat the case with no stem. When the stem is nonempty, maps between the core projectives give a splitting at $Q_1^{\oplus2}$, which we then extend along the stem.

\begin{lemma}\label{lem:triangle-stem-splitting}
Suppose that $\mathbf R$ consists of an unpunctured disk triangle and a
stem $w_0=v_1,w_1,\ldots,w_d$. Label the core projectives so that
\[
 B[P_1]=e_{v_1}+e_{v_2},\quad
 B[P_2]=e_{v_2}+e_{v_3}\quad\text{and}\quad
 B[P_3]=e_{v_3}+e_{v_1},
\]
and let $Q_i$ be the projective indexed by $w_{i-1}w_i$, with
$B[Q_i]=e_{w_{i-1}}+e_{w_i}$.  Let
$S_d=\sum_{i=1}^d(-1)^{d-i}[Q_i]$ and
$z_0=[P_1]-[P_2]+[P_3]$. Then
\begin{equation}\label{eq:7.3}
 B^{-1}(2e_{w_d})=z_d=(-1)^dz_0+2S_d.
\end{equation}
For each $z\in\{z_d,-z_d\}$, a nonempty open subset of $E_z$ consists
of differentials defining decomposable complexes.
\end{lemma}
\begin{proof}
As in Lemma~\ref{lem:conjugate-stem-splitting},
$BS_d=e_{w_d}-(-1)^de_{w_0}$. Since $Bz_0=2e_{w_0}$,
we obtain $Bz_d=2e_{w_d}$, proving~\eqref{eq:7.3}.

For $d\geq1$, the presentation in Appendix~\ref{app:triangle-presentation} gives the following nonzero maps between projectives (products of maps denote composition from right to left):
\[
 P_2\overset s\longrightarrow P_1
 \overset u\longrightarrow P_3\overset t\longrightarrow P_2,
 \qquad P_3\overset h\longrightarrow Q_1
 \overset k\longrightarrow P_1,
 \quad\text{and}\quad st=kh.
\]
For $d=0$, use the same core generators $s,u,t$ and omit $Q_1,h,k$
and the relation $st=kh$.
Appendix~\ref{app:triangle-presentation} shows that the displayed composites are nonzero. By Proposition~\ref{res:4.1}, the nonzero off-diagonal Hom-spaces are one-dimensional. Thus the nonzero off-diagonal entries in the following table are basis maps. Sources index the rows and targets index the columns:
\begin{equation}\label{eq:7.2}
 \begin{array}{c|cccc}
       &P_1&P_2&P_3&Q_1\\\hline
 P_1   & *&tu&u&hu\\
 P_2   & s&*&us&0\\
 P_3   & st=kh&t&*&h\\
 Q_1   & k&0&uk&*
 \end{array}
\end{equation}
For $d\ge2$, the remaining nonzero off-diagonal Hom-spaces are $\operatorname{Hom}_A(Q_i,Q_{i+1})$ and $\operatorname{Hom}_A(Q_{i+1},Q_i)$, $1\le i<d$, each of dimension one.

The projective supports of $P^-(z_d)$ and $P^+(z_d)$ are
disjoint, so the diagonal Hom-spaces marked $*$ do not occur in $E_{z_d}$.

\emph{No stem.}
For $z_0$ and $-z_0$, the differentials have the respective forms
\[
 P_2\xrightarrow{(\alpha s,\,\beta us)^t}P_1\oplus P_3
 \quad\text{and}\quad
 P_1\oplus P_3\xrightarrow{(\alpha tu,\,\beta t)}P_2.
\]
Here $\alpha,\beta\in k$. For $z_0$, assume $\alpha\neq0$ and change the target by the first automorphism below. For $-z_0$, assume $\beta\neq0$ and change the source by the second:
\[
 \begin{pmatrix}1&0\\-(\beta/\alpha)u&1\end{pmatrix}
 \begin{pmatrix}\alpha s\\\beta us\end{pmatrix}
 =\begin{pmatrix}\alpha s\\0\end{pmatrix}
 \quad\text{and}\quad
 \begin{pmatrix}\alpha tu&\beta t\end{pmatrix}
 \begin{pmatrix}1&0\\-(\alpha/\beta)u&1\end{pmatrix}
 =\begin{pmatrix}0&\beta t\end{pmatrix}.
\]
Thus the complexes are isomorphic to
$(P_2\overset{\alpha s}\longrightarrow P_1)\oplus P_3$ and
$P_1[1]\oplus(P_3\overset{\beta t}\longrightarrow P_2)$, respectively,
and decompose on nonempty open subsets.

\emph{Nonempty stem.}
The signs in~\eqref{eq:7.3} determine the source and target of the differential. We treat the two parities separately.
For odd $d$, the core part of a differential of index $z_d$ is a map
\mbox{$P_1\oplus P_3\longrightarrow P_2\oplus Q_1^{\oplus2}$} with matrix
\begin{equation}\label{eq:7.4}
 \begin{pmatrix}\alpha tu&\beta t\\xhu&yh\end{pmatrix},
 \qquad x,y\in k^2.
\end{equation}
On $\beta\neq0$, multiply on the right by the displayed source
automorphism:
\[
 \begin{pmatrix}\alpha tu&\beta t\\xhu&yh\end{pmatrix}
 \begin{pmatrix}1&0\\-(\alpha/\beta)u&1\end{pmatrix}
 =\begin{pmatrix}0&\beta t\\x'hu&yh\end{pmatrix},
\]
where $x'=x-(\alpha/\beta)y$.
Since $\det(x',y)=\det(x,y)$, the condition $\det(x,y)\neq0$ makes $x',y$ a basis of $k^2$. We can therefore change the basis of $Q_1^{\oplus2}$ to send $x',y$ to the standard basis. The matrix becomes
\[
 \begin{pmatrix}
 0&\beta t\\hu&0\\0&h
 \end{pmatrix}.
\]
The two columns now lie in complementary target summands. They give
$P_1\overset{hu}\longrightarrow Q_1$ and
$P_3\xrightarrow{(\beta t,h)^t}P_2\oplus Q_1$.

For positive even $d$, the core part is a map
\mbox{$P_2\oplus Q_1^{\oplus2}\longrightarrow P_1\oplus P_3$} with matrix
\begin{equation}\label{eq:7.5}
 \begin{pmatrix}\alpha s&xk\\\beta us&yuk\end{pmatrix},
 \qquad x,y\in(k^2)^*.
\end{equation}
For $\alpha\neq0$,  let $y'=y-(\beta/\alpha)x$.
Multiplying on the left by the target automorphism gives
\[
 \begin{pmatrix}1&0\\-(\beta/\alpha)u&1\end{pmatrix}
 \begin{pmatrix}\alpha s&xk\\\beta us&yuk\end{pmatrix}
 =\begin{pmatrix}\alpha s&xk\\0&y'uk\end{pmatrix}.
\]
Similarly, $\det\binom{x}{y'}=\det\binom{x}{y}$. If this determinant is nonzero, a basis change in $Q_1^{\oplus2}$ gives
\[
 \begin{pmatrix}\alpha s&k&0\\0&0&uk\end{pmatrix}.
\]
The two rows use complementary source summands. They give
$P_2\oplus Q_1\xrightarrow{(\alpha s,k)}P_1$ and
$Q_1\overset{uk}\longrightarrow P_3$.

For odd $d$, require $\beta\det(x,y)\neq0$; for positive even $d$, require $\alpha\det\binom{x}{y}\neq0$. In both cases, also require every coefficient matrix between successive $Q_i^{\oplus2}$ to be invertible. These are polynomial nonvanishing conditions on $E_{z_d}$. They hold at a point: take $(\alpha,\beta)=(0,1)$ for odd $d$ and $(1,0)$ for even $d$, take $x,y$ to be the standard basis vectors or covectors, and take the stem coefficient matrices to be $I_2$. Hence they define a nonempty open subset. On this subset, keep the chosen basis of $Q_1^{\oplus2}$ and proceed along the stem as in Lemma~\ref{lem:conjugate-stem-splitting}. Each core subcomplex then uses one copy of every later $Q_i$, so the full complex has two nonzero direct summands.
For $-z_d$, the source and target are exchanged: odd $d$ gives~\eqref{eq:7.5}, and positive even $d$ gives~\eqref{eq:7.4}. The same calculations apply.
\end{proof}

For comparison, in characteristic different from two, skew-Brauer tree
algebras with at least three edges are trivial extensions of iterated tilted
algebras of type $D$ \cite[Section~6]{GEGV}. Lemmas~\ref{lem:conjugate-stem-splitting}
and~\ref{lem:triangle-stem-splitting} apply in every characteristic.

\begin{proof}[Proof of Theorem~\ref{res:7.5} for an exceptional core]
Suppose that the core is exceptional. Fix $p\in\mathbf P_{\mathbf R}$
and $a\in\{\pm1\}$. By Lemma~\ref{res:7.7}, the occurrence of $2ae_p$
can be checked over $\Lambda_{\mathbf I_p}$.
 Let $z=B_{\mathbf I_p}^{-1}(2ae_p)$.
For a triangle core, cyclically relabel the core vertices so that the stem is attached at $v_1$ (or $p=v_1$ if the stem is empty). For a conjugate core, use $\pm z_d$ at the common-tag endpoint and along the stem, and $\pm([P_a]-[P_b])$ at the opposite-tag endpoint. Lemmas~\ref{lem:conjugate-stem-splitting} and~\ref{lem:triangle-stem-splitting} therefore give a nonempty open subset of $E_z$ consisting of decomposable differentials.
Thus Lemma~\ref{res:7.8} excludes a complex $T\in\indtwopsilt\Lambda_{\mathbf I_p}$
with $[T]=z$ in every case. Hence $2ae_p\notin L_{\mathbf I_p}$,
and Lemma~\ref{res:7.7} gives $2ae_p\notin L_{\mathbf R}$.
Since $p$ and $a$ were arbitrary, $L_{\mathbf R}=\varnothing$.
\end{proof}

For connected $\Gamma_{\mathbf R}$ with $n$ edges,  Theorems~\ref{res:7.3} and~\ref{res:7.5} give
\[
 \begin{array}{c|cc}
 \text{configuration of }\mathbf R&\text{ray set}&|V(A)|\\\hline
 \Gamma_{\mathbf R}\text{ a tree}&A_n&n(n+1)\\
 \text{nonexceptional core}&C_n&2n^2\\
 \text{exceptional core}&D_n&2n(n-1).
 \end{array}
\]
In the tree case the $p$-coordinate is multiplied by $\varepsilon_p$.
Tagged compatibility specifies the cones on the root sets.
If $\Gamma_{\mathbf R}$ is disconnected, the Cartan entries between distinct components vanish. Hence $A$ is the product of the corresponding factors. Its ray set is the union of their ray sets in the respective coordinate subspaces, and $|V(A)|$ is the sum of their numbers of rays.

\subsection{The canonical $g$-polytope}

\begin{definition}[{\cite[Definitions~5.1--5.2]{AHIKM}}]\label{def:canonical-g-polytope}
The \emph{canonical $g$-polytope} of $A$ is
\[
 P(A)=\bigcup_{u=u_1\oplus\cdots\oplus u_n\in s\tau\text{-}\operatorname{tilt}A}
 \operatorname{conv}\{0,g(u_1),\ldots,g(u_n)\}.
\]
\end{definition}

If $\Gamma_{\mathbf R}$ is a tree, or is unicyclic with a nonexceptional core, Proposition~\ref{res:4.9} and \cite[Theorem~8.4]{AHIKM} imply that $P(A)$ is convex. For an exceptional core, the endpoint vectors still form a root system, but $P(A)$ need not be convex, as the following example shows.

\begin{example}\label{res:7.11}
Let $\mathbf R$ consist of a conjugate pair and a two-edge stem attached
at the common-tag endpoint. Then $Bg(V(\Lambda_{\mathbf R}))=D_4$
and $P(\Lambda_{\mathbf R})$ is nonconvex.
\end{example}
\begin{proof}
Write the core endpoints as $v,q$ and the stem as $v$--$r$--$s$.
Label the arcs by $a,b,c,d$, with columns
\[
 B_a=e_v+e_q,\qquad B_b=e_v-e_q,\qquad
 B_c=e_v+e_r\quad\text{and}\quad B_d=e_r+e_s.
\]
Theorem~\ref{res:7.5} gives $L_{\mathbf R}=\varnothing$, so
$Bg(V(A))=D_4$. The Cartan matrix is
\[
 C_A=B^tB=
 \begin{pmatrix}
 2&0&1&0\\
 0&2&1&0\\
 1&1&2&1\\
 0&0&1&2
 \end{pmatrix}.
\]
Let $P_i$ be the projective indexed by arc $i$, for $i=a,b,c,d$, and  let
$g_i=[P_i]$.
The Cartan matrix gives $\dim_k\operatorname{Hom}(P_c,P_i)=1$ for $i=a,b,d$ and $\operatorname{Hom}(P_i,P_j)=0$ for distinct $i,j\in\{a,b,d\}$. Choose a nonzero $f_i\in\operatorname{Hom}(P_c,P_i)$ for each $i$. Then
\[
 f=(f_a,f_b,f_d)^t:
 P_c\longrightarrow P_a\oplus P_b\oplus P_d
\]
is a left $\operatorname{add}(P_a\oplus P_b\oplus P_d)$-approximation.
To check minimality, let $h$ be an endomorphism of
$P_a\oplus P_b\oplus P_d$ with $hf=f$.
The vanishing Hom-spaces make $h$ diagonal, with entries $h_i$ satisfying
$h_if_i=f_i$. If $h_i$ were not invertible, then $1-h_i$ would be invertible because $\operatorname{End}(P_i)$ is local. But $(1-h_i)f_i=0$ would then imply $f_i=0$. So every $h_i$ is invertible. Thus $h$ is invertible and $f$ is minimal.

Two-term silting mutation \cite{AIR} gives the exchange triangle
\[
 P_c\overset f\longrightarrow P_a\oplus P_b\oplus P_d
 \longrightarrow P_c^*\longrightarrow P_c[1].
\]
Writing $g_c^*=[P_c^*]$, the exchange triangle gives
$g_c+g_c^*=g_a+g_b+g_d$.
The middle term has three indecomposable summands. By \cite[Definition~5.9 and Theorem~5.10(b)]{AHIKM}, $P(A)$ is therefore nonconvex. Since $B$ is invertible and the ray set is
$D_4=-D_4$, we obtain
\[
 \operatorname{conv}(BP(A))=\operatorname{conv}(D_4)
 \quad\text{and}\quad BP(A)\subsetneq\operatorname{conv}(D_4).\qedhere
\]
\end{proof}


\FloatBarrier
\appendix

\section{Cutting and index projections}
\label{sec:cutting-index-projections}\label{sec:appendix-cutting}

Theorem~\ref{res:2.5} identifies the reduction algebra $A_u$ with
$\Lambda_{\mathbf R_u}$. We show that tagged cutting induces the
corresponding isomorphisms of simplicial links and quotient $g$-fans,
and compute the map on $g$-vectors in the projective coordinates of $A_u$.

Fix $R=X_{\mathbf S}(\mathbf R)$,  let $A=\Lambda_{\mathbf R}$, and let
$u=(M,P)$ be a basic $\tau$-rigid pair.
Let $\mathbf N_u$ be the compatible arc collection representing $u$,  let $U=X_{\mathbf S}(\mathbf N_u)$, and let
$B_U=U\oplus L$ be the Bongartz completion from Proposition~\ref{res:2.1}. Write $\Phi(B_U)=(T,P)$.
 Let $\mathcal Z={}^{\perp}U[1]$, and write
$q:\mathcal Z\to\mathcal C_U$ for the quotient functor.
For $W\in\mathcal Z$,  let $\overline W=q(W)$ and denote the suspension
of $\mathcal C_U$ by $\langle1\rangle$.

For a simplex $\sigma$ of a simplicial complex $\Delta$,  the \emph{link} of $\sigma$ is
\[
 \operatorname{lk}_{\Delta}(\sigma)
 =\{\tau\in\Delta:\tau\cap\sigma=\varnothing,
                         \ \tau\cup\sigma\in\Delta\}.
\]
We use the complexes $\Delta_\tau(A)$ and $\Delta_{\mathbf R}$ from
Definition~\ref{def:tau-complex}.

Write $U=\bigoplus_{i=1}^{s}U_i$ and $L=\bigoplus_{j=1}^{r}L_j$,
and order the arcs of $\mathbf R_u$ by $L_1,\ldots,L_r$.
Since $\Phi(B_U)$ is a basic $\tau$-tilting pair and the simple
$A_u$-modules are indexed by $\mathbf R_u$, we have
\[
 n=|A|=s+r\quad\text{and}\quad r=|A|-|M|-|P|=|A_u|.
\]
By \cite[Theorem~5.1]{AIR}, the $g$-matrix of the Bongartz completion
$\Phi(B_U)$,
\[
 G_u=\bigl(\operatorname{ind}_R U_1,\ldots,
 \operatorname{ind}_R U_s\mid
 \operatorname{ind}_R L_1,\ldots,\operatorname{ind}_R L_r\bigr),
\]
belongs to $\mathrm{GL}_n(\mathbb Z)$.
\Needspace*{6\baselineskip}
Define the \emph{index projection} by
\begin{equation}\label{eq:2.6}
 p_u:\mathbb Z^n\longrightarrow\mathbb Z^r,
 \qquad x\longmapsto\operatorname{pr}_L(G_u^{-1}x),
\end{equation}
where $\operatorname{pr}_L$ is the projection onto the last $r$ coordinates.
Equivalently, $p_u(\operatorname{ind}_R U_i)=0$ and
$p_u(\operatorname{ind}_R L_j)=e_j$.
Thus $p_u$ induces a lattice isomorphism
\[
 \overline p_u:
 \mathbb Z^n\Big/\sum_{i=1}^{s}\mathbb Z\operatorname{ind}_R U_i
 \overset{\sim}{\longrightarrow}\mathbb Z^r.
\]

\begin{definition}
For $\sigma_u=\operatorname{cone}\{\operatorname{ind}_R U_i:1\leq i\leq s\}$,
let $\pi:\mathbb R^n\to\mathbb R^n/\operatorname{span}_{\mathbb R}\sigma_u$
be the quotient map. The \emph{quotient fan} is
\[
 \Sigma_g(A)/\sigma_u
 =\{\pi(\tau):\tau\in\Sigma_g(A),\ \sigma_u\subseteq\tau\}.
\]
\end{definition}

The link and quotient-fan isomorphisms follow from
\cite[Theorems~3.6 and~3.7]{BM} and \cite[Theorem~4.11]{AHIKM},
respectively. Related index formulas hold in $0$-Auslander extriangulated
categories \cite[Lemmas~2.7,~2.8 and~2.11]{BK}.
Here we identify these maps with tagged cutting and the projection $p_u$.

\begin{theorem}\label{thm:cutting-link-fan-v79}\label{res:2.9}
There are simplicial isomorphisms
\begin{equation}\label{eq:2.7}
 \operatorname{lk}_{\Delta_{\mathbf R}}(\mathbf N_u)
 \cong\operatorname{lk}_{\Delta_\tau(A)}(u)
 \cong\Delta_\tau(A_u)\cong\Delta_{\mathbf R_u},
\end{equation}
whose composite sends $\gamma$ to $F_{\mathbf N_u}(\gamma)$. Moreover,
\begin{equation}\label{eq:2.8}
 p_u\bigl(g_{\mathbf R}(\gamma)\bigr)
 =g_{\mathbf R_u}\bigl(F_{\mathbf N_u}(\gamma)\bigr)
\end{equation}
for every vertex $\gamma$ of
$\operatorname{lk}_{\Delta_{\mathbf R}}(\mathbf N_u)$.
The real extension of $\overline p_u$ induces an isomorphism of fans
\begin{equation}\label{eq:2.9}
 \Sigma_g(A)/\sigma_u
 \overset{\ \overline p_u\ }{\longrightarrow}
 \Sigma_g(A_u)\cong\Sigma_g(\Lambda_{\mathbf R_u}),
\end{equation}
sending the projected ray of $\gamma$ to the ray of $F_{\mathbf N_u}(\gamma)$.
\end{theorem}

We first compare the indices in $\mathcal C$ and $\mathcal C_U$.

\begin{lemma}\label{lem:quotient-presentations-v79}\label{res:2.10}
\label{lem:quotient-indices-v79}\label{res:2.11}
For every $W\in\mathcal E_R\cap\mathcal Z$, the object $\overline W$
admits an $\overline L$-presentation in $\mathcal C_U$ and
\begin{equation}\label{eq:2.11}
 \operatorname{ind}_{\overline L}\overline W
 =p_u(\operatorname{ind}_R W),
\end{equation}
where the target coordinates are $[\overline L_1],\ldots,[\overline L_r]$.
\end{lemma}
\begin{proof}
Let $Y\in(\operatorname{add}R*\operatorname{add}U)\cap\mathcal Z$,
and choose a triangle
$R_1\to Y\to U''\xrightarrow{d}R_1[1]$ with
$R_1\in\operatorname{add}R$ and $U''\in\operatorname{add}U$.
Since $U'_1\to R_1[1]$ is a right $\operatorname{add}U$-approximation
as in Proposition~\ref{res:2.1}, the map $d$ factors through it.
The homotopy pullback gives a triangle
\[
 Y\longrightarrow V_1\oplus U''\longrightarrow U'_1\longrightarrow Y[1],
\]
with $V_1\in\operatorname{add}B_U$.
Since
$\operatorname{Hom}(U'_1,Y[1])
\cong D\operatorname{Hom}(Y,U'_1[1])=0$, the last map vanishes.
Thus the triangle splits, so $Y$ is a direct summand of $V_1\oplus U''$. Since $V_1,U''\in\operatorname{add}B_U$, we obtain
$(\operatorname{add}R*\operatorname{add}U)\cap\mathcal Z
\subseteq\operatorname{add}B_U$.

Choose an $R$-presentation
$R_1\to R_0\xrightarrow{p}W\to R_1[1]$ and the approximation triangle
$R_0\xrightarrow{a_0}V_0\to U'_0\to R_0[1]$
of Proposition~\ref{res:2.1}.
Since $W\in\mathcal Z$, $2$-Calabi--Yau duality gives
$\operatorname{Hom}(U'_0[-1],W)=0$. Thus $p$ factors through $a_0$.
Choose a right $\operatorname{add}U$-approximation $U_0\to W$,
and complete the induced map $h:V_0\oplus U_0\to W$ to a triangle
\begin{equation}\label{eq:2.10}
 W_1\longrightarrow V_0\oplus U_0\xrightarrow{h}W\longrightarrow W_1[1].
\end{equation}
By construction, $\operatorname{Hom}(U,h)$ is surjective.
Since $V_0\oplus U_0\in\operatorname{add}B_U$ and $B_U$ is rigid,
applying $\operatorname{Hom}(U,-)$ to \eqref{eq:2.10} gives
$\operatorname{Hom}(U,W_1[1])=0$.
So $W_1\in\mathcal Z$ by $2$-Calabi--Yau duality.
The octahedral axiom for
$R_0\xrightarrow{(a_0,0)}V_0\oplus U_0\xrightarrow{h}W$
gives a triangle
\[
 R_1\longrightarrow W_1\longrightarrow U'_0\oplus U_0
 \longrightarrow R_1[1].
\]
Thus
$W_1\in(\operatorname{add}R*\operatorname{add}U)\cap\mathcal Z
\subseteq\operatorname{add}B_U$, by the inclusion proved at the start.
Moreover, $\operatorname{Hom}(W,U[1])=0$ implies that
$\operatorname{Hom}(V_0\oplus U_0,U)\to\operatorname{Hom}(W_1,U)$
is surjective. Hence \eqref{eq:2.10} descends to the triangle
\[
 \overline W_1\longrightarrow\overline V_0\longrightarrow\overline W
 \longrightarrow\overline W_1\langle1\rangle
\]
in $\mathcal C_U$ by \cite[Definition~4.1 and Theorem~4.2]{IY}.
Since $W_1,V_0\in\operatorname{add}B_U$, we have
$\overline W_1,\overline V_0\in\operatorname{add}\overline L$.
Thus this triangle is an $\overline L$-presentation of $\overline W$.

For a triangle $R'\to V\to Z\to R'[1]$ with
$R'\in\operatorname{add}R$ and $Z\in\mathcal E_R$, choose an
$R$-presentation $R_Z^1\to R_Z^0\to Z\to R_Z^1[1]$.
 Since $\operatorname{Hom}(R_Z^0,R'[1])=0$, the chosen presentation lifts to an
$R$-presentation $R_Z^1\to R'\oplus R_Z^0\to V\to R_Z^1[1]$. Hence
\begin{equation}\label{eq:2.12}
 \operatorname{ind}_R V=[R']+\operatorname{ind}_R Z.
\end{equation}

Apply \eqref{eq:2.12} to the triangles for $V_0$ and $W_1$. Then
\[
 \operatorname{ind}_R V_0=[R_0]+\operatorname{ind}_R U'_0
 \quad\text{and}\quad
 \operatorname{ind}_R W_1=[R_1]+\operatorname{ind}_R U'_0
                              +\operatorname{ind}_R U_0,
\]
and therefore
\begin{equation}\label{eq:quotient-index-difference}
 \operatorname{ind}_R(V_0\oplus U_0)-\operatorname{ind}_R W_1
 =[R_0]-[R_1]=\operatorname{ind}_R W.
\end{equation}
For $T\in\operatorname{add}B_U$, the vector
$p_u(\operatorname{ind}_R T)$ is $[\overline T]$ in the
$\overline L$-basis. Applying $p_u$ to  \eqref{eq:quotient-index-difference} gives
\[
 p_u(\operatorname{ind}_R W)
 =[\overline V_0]-[\overline W_1]
 =\operatorname{ind}_{\overline L}\overline W.
\qedhere
\]
\end{proof}

\begin{proof}[Proof of Theorem~\ref{thm:cutting-link-fan-v79}]
Let $\gamma$ be a vertex of
$\operatorname{lk}_{\Delta_{\mathbf R}}(\mathbf N_u)$ and  let
$W=X_{\mathbf S}(\gamma)$.
Since $\gamma$ is $\mathbf R$-standard, $W\in\mathcal E_R$. Its compatibility with $\mathbf N_u$ gives $W\in\mathcal Z$.
Lemma~\ref{lem:quotient-presentations-v79} gives an $\overline L$-presentation
of $\overline W$. By Proposition~\ref{res:2.4}, the equivalence $\xi_{\mathbf N_u}$ sends
$\overline L$ to $X_{\mathbf S/\mathbf N_u}(\mathbf R_u)$ and
$\overline W$ to $X_{\mathbf S/\mathbf N_u}(F_{\mathbf N_u}(\gamma))$.
Thus $F_{\mathbf N_u}(\gamma)$ is $\mathbf R_u$-standard, and
\[
 g_{\mathbf R_u}(F_{\mathbf N_u}(\gamma))
 =\operatorname{ind}_{\overline L}\overline W
 =p_u(\operatorname{ind}_R W)
 =p_u(g_{\mathbf R}(\gamma)).
\]
Equation~\eqref{eq:2.8} follows.
Tagged cutting is injective and preserves compatibility in both directions.
Hence the cutting map defines a full simplicial embedding
\[
 f:\operatorname{lk}_{\Delta_{\mathbf R}}(\mathbf N_u)
 \longrightarrow\Delta_{\mathbf R_u},
 \qquad \gamma\longmapsto F_{\mathbf N_u}(\gamma).
\]

We next express the quotient-fan map of \cite[Theorem~4.11]{AHIKM}
in the projective basis of $A_u$. Let $\widehat U$ and $\widehat L$ be the two-term presilting complexes
corresponding to $U$ and $L$. Their direct sum
$\widehat B_U=\widehat U\oplus\widehat L$ is the two-term silting
Bongartz completion of $\widehat U$.
By \cite[Lemma~2.1(c)]{BM}, $H^0$ is full on two-term projective
complexes, and its kernel consists of maps factoring through
$\operatorname{add}A[1]$. Consider the Bongartz triangle
$A\to\widehat V\to\widehat U'\to A[1]$.
Since $\widehat B_U$ is presilting and
$\widehat V\in\operatorname{add}\widehat B_U$,
$\operatorname{Hom}(\widehat B_U,\widehat V[1])=0$.
So every map $\widehat B_U\to A[1]$ factors through
$\widehat U'\in\operatorname{add}\widehat U$.
Thus the kernel of $H^0$ on $\operatorname{End}(\widehat B_U)$
is contained in $[\operatorname{add}\widehat U]$.
Since $H^0(\widehat U)\cong M$, fullness also gives
$H^0([\operatorname{add}\widehat U])=[\operatorname{add}M]$.
Taking quotients therefore gives the algebra isomorphism
\begin{equation}\label{eq:2.13}
 \operatorname{End}_{K^b(\operatorname{proj}A)}(\widehat B_U)
       /[\operatorname{add}\widehat U]
 \overset{\sim}{\longrightarrow}
 \operatorname{End}_A(T)/[\operatorname{add}M]=A_u.
\end{equation}

This argument also applies when $P\neq0$, and \eqref{eq:2.13}
preserves the summand idempotents.
By Proposition~\ref{res:2.1}, the projective part of $\Phi(B_U)$ is $P$.
Since $L$ has no summand in common with $U$, we have
$T_j=H^0(\widehat L_j)=\operatorname{Hom}(R,L_j)\neq0$ for every $j$.

Let $e_j\in A_u$ be the image of the projection onto the summand $T_j$,
and  let $Q_j=e_jA_u$.
The lattice map of \cite[Theorem~4.11]{AHIKM} sends
$[\widehat U_i]$ to zero and $[\widehat L_j]$ to $[Q_j]$.
Under the algebra isomorphisms of Proposition~\ref{res:2.2} and Theorem~\ref{res:2.5},
\[
 Q_j\cong
 \operatorname{Hom}_{\mathcal C_U}(\overline L,\overline L_j)
 \cong
 \operatorname{Hom}_{\mathcal C(\mathbf S/\mathbf N_u)}
 \bigl(X_{\mathbf S/\mathbf N_u}(\mathbf R_u),
       \xi_{\mathbf N_u}(\overline L_j)\bigr).
\]
The indecomposable projective module $Q_j$ therefore corresponds to the
cut arc $F_{\mathbf N_u}(X_{\mathbf S}^{-1}(L_j))$.
In the ordered projective basis of $A_u$, the class $[Q_j]$ is the
$j$th coordinate vector.
Since $[\widehat U_i]=\operatorname{ind}_R U_i$ and
$[\widehat L_j]=\operatorname{ind}_R L_j$, the canonical lattice map and $p_u$
agree on the columns of $G_u$. These columns form a $\mathbb Z$-basis,
so the two maps are equal. Passing to the quotient gives \eqref{eq:2.9}.

Finally, let $\delta$ be a vertex of $\Delta_{\mathbf R_u}$.
By \eqref{eq:2.9} and the definition of the quotient fan, the preimage
of $\mathbb R_{\geq0}g_{\mathbf R_u}(\delta)$ is $\pi(\tau)$
for some $\tau\in\Sigma_g(A)$ with $\sigma_u\subseteq\tau$.
Since $\dim\sigma_u=s$ and the projected cone is one-dimensional,
$\dim\tau=s+1$.
Since $\Sigma_g(A)$ is simplicial, $\tau$ is generated by the
vectors $\operatorname{ind}_R U_i$ and one additional vector $g_{\mathbf R}(\gamma)$,
where $\gamma$ is a vertex of
$\operatorname{lk}_{\Delta_{\mathbf R}}(\mathbf N_u)$.
By \eqref{eq:2.8},
$g_{\mathbf R_u}(f(\gamma))$ and $g_{\mathbf R_u}(\delta)$ generate the same
ray. Both are primitive vectors in $\mathbb Z^r$, so they are equal.
By \cite[Theorem~5.5]{AIR} and Theorem~\ref{res:1.9}, we obtain
$f(\gamma)=\delta$. Thus $f$ is surjective on vertices.
Since $f$ is a full simplicial embedding, it is a simplicial isomorphism.
The correspondences $X$ and $\Phi$ give \eqref{eq:2.7}.
\end{proof}


\section{Algebra presentations}
\label{app:algebra-presentations}\label{sec:appendix-algebra-presentations}

We compare surface rigid algebras with Demonet's partial-triangulation algebras,
prove tameness, and compute the projective maps for the triangle core.
The comparison assumes distinct underlying arcs; the tameness argument
also allows conjugate pairs.

\subsection{Comparison with partial-triangulation algebras}
\label{app:comparison-proof}

\begin{definition}\label{def:brauer-graph-algebra}
Let $G$ be a finite ribbon graph without isolated vertices, with half-edge
set $H$, fixed-point-free involution $h\mapsto\bar h$ and cyclic successor
$\sigma$. Write $E_h=\{h,\bar h\}$, let $p(h)$ be the vertex incident
with $h$, and let $d_h$ be the valency of $p(h)$.
Each loop contributes two half-edges.
The \emph{angle quiver} $Q_G$ is given by
\[
 (Q_G)_0=E(G)
 \quad\text{and}\quad
 (Q_G)_1=\{\alpha_h:E_h\longrightarrow E_{\sigma h}:h\in H\}.
\]
For $h\in H$, the \emph{endpoint cycle} at $h$ is
\[
 C_h=\alpha_h\alpha_{\sigma h}\cdots\alpha_{\sigma^{d_h-1}h}.
\]
Paths are read from left to right.
The \emph{multiplicity-one Brauer graph algebra} of $G$ is
\[
 A(G)=kQ_G/
 \bigl\langle C_h-C_{\bar h},\,
       \alpha_h\alpha_{\overline{\sigma h}}:h\in H\bigr\rangle.
\]
If $G=\bigsqcup_iG_i$, this presentation gives
$A(G)\cong\prod_iA(G_i)$. We set $A(\varnothing)=0$.
\end{definition}

If $d_h=1$, then $\sigma h=h$, so $\alpha_h$ is a loop at $E_h$.
Since $C_h=C_{\bar h}$ and
$C_{\bar h}\alpha_h$ contains a zero relation,
$C_h\alpha_h=C_{\bar h}\alpha_h=0$.
If $d_h=1<d_{\bar h}$, eliminating $\alpha_h=C_{\bar h}$ gives the
usual presentation \cite[Definition~3.4]{AAC}.
If $G$ consists of a single edge with distinct endpoints, the two
angle loops satisfy $x=y$ and $xy=yx=0$. Hence $A(G)\cong k[t]/(t^2)$.

\begin{definition}\label{def:demonet-algebra}
Let $\mathbf Q$ be a partial tagged triangulation with distinct underlying
arcs and all endpoints at punctures.  Let
$\sigma_{\mathbf Q}=\mathbf Q^\circ\cup\mathbf B$, where
$\mathbf Q^\circ$ is obtained by forgetting taggings and $\mathbf B$ is
the set of boundary segments.
Fix $m_p=1$ and $\lambda_p\in k^\times$ for every
$p\in\mathbf M\cup\mathbf P$.
Let $\Gamma^D_{\sigma_{\mathbf Q}}$ be Demonet's \emph{completed frozen algebra}
on the angle quiver of $\sigma_{\mathbf Q}$, with the weighted endpoint
relations
$\lambda_{p(h)}C_h=\lambda_{p(\bar h)}C_{\bar h}$
and the polygon relations of \cite[Definitions~2.5 and~2.12]{D}.
We use the \emph{boundary quotient}
\[
 \Delta^D_{\mathbf Q}
 =\Gamma^D_{\sigma_{\mathbf Q}}\big/
   \Gamma^D_{\sigma_{\mathbf Q}}e_\partial
                \Gamma^D_{\sigma_{\mathbf Q}},
\]
where $e_\partial=\sum_{b\in\mathbf B}e_b$,
see \cite[Definition~4.1]{D}.
In both $\Delta^D_{\mathbf Q}$ and $A(G_{\mathbf Q})$, the
vertex idempotents are indexed by the arcs of $\mathbf Q$.
\end{definition}

\begin{definition}\label{def:sparse-partial-triangulation}
Let $D$ be a closed disk with $n$ marked boundary points in cyclic order.
A \emph{disk $n$-gon} of $\sigma_{\mathbf Q}$ is an
orientation-preserving continuous map $f:D\longrightarrow S$ which is
injective on $\operatorname{int}D$ and on each open side, and maps each
side onto an oriented edge of $\sigma_{\mathbf Q}$,
see \cite[Definition~2.8]{D}. The restriction $f|_{\partial D}$ need not be injective: boundary vertices
may be identified, and distinct sides may map onto the same underlying
edge. By \cite[Proposition~2.10]{D}, the oriented edge images of the
sides are pairwise distinct.  Let $\mathbf P_f=\mathbf P\cap f(\operatorname{int}D)$.
Since all arc endpoints are punctures and $m_p=1$,
$\sigma_{\mathbf Q}$ is \emph{sparse} if and only if no disk $n$-gon satisfies
$(n,|\mathbf P_f|)=(3,0)$ or $(1,1)$
\cite[Definition~6.1]{D}.
We call $\mathbf Q^\circ$ sparse if $\sigma_{\mathbf Q}$ is sparse.
\end{definition}

Suppose $\mathbf Q\neq\varnothing$. Since no arc of
$\mathbf Q^\circ$ ends on $\partial S$, a disk polygon with one
boundary side has only boundary sides.
The set $f(\operatorname{int}D)=f(D)\cap\operatorname{int}S$ is
nonempty and open by invariance of domain. Since $f(D)$ is compact,  $f(\operatorname{int}D)$ is
also closed in $\operatorname{int}S$. Thus
$f(\operatorname{int}D)=\operatorname{int}S$, so $S$ is a disk and
$\mathbf P_f=\mathbf P$.
Such a polygon cannot be unpunctured when $\mathbf Q\neq\varnothing$,
and the once-punctured monogon surface is excluded by our conventions.
A once-punctured disk monogon with side in $\mathbf Q^\circ$ is
excluded by the definition of a tagged arc. Therefore, in this setting,
sparsity is equivalent to the absence of an unpunctured disk triangle
with sides in $\mathbf Q^\circ$.

The comparison with $\Delta^D_{\mathbf Q}$ below does not require
sparsity. Sparsity is used only to identify this algebra with
$A(G_{\mathbf Q})$.

\begin{proof}[Proof of Proposition~\ref{res:4.9}]

Since a disk with only one puncture has no essential arc with both
ends at punctures, the once-punctured digon excluded in \cite{LF}
cannot occur.
Since the underlying arcs of $\mathbf Q$ are distinct, compatibility
gives the same tagging at every common endpoint. By
\cite[Remark~3.10]{HZZ}, extend $\mathbf Q$ to a tagged triangulation
$\mathbf T$ with the same nonzero signatures on $\mathbf P_{\mathbf Q}$
and plain tags at the remaining punctures. Then $\mathbf T^\circ$
contains $\mathbf Q^\circ$ and has no self-folded triangle.
Write $J_{\mathbf T}=\operatorname{End}_{\mathcal C(\mathbf S)}X(\mathbf T)$.
For the nonzero surface-potential coefficients $x_p$, choose
$\lambda_p=-\kappa_{\mathbf T}(p)x_p$ at punctures and $\lambda_v=1$
at boundary marked points.

Demonet reads paths from left to right, whereas multiplication in
$J_{\mathbf T}$ is composition of maps. Thus the arrow identification
$F$ from the boundary quotient of Demonet's quiver to the unreduced
surface quiver reverses written products:
\[
 F(a_1\cdots a_r)=F(a_r)\cdots F(a_1).
\]
Under the arrow identification $F$, \cite[Corollary~5.2]{D} gives the unreduced
surface potential of \cite[Definitions~3.1 and~3.2]{LF}, with triangle
coefficients $1$ and puncture coefficients
$-\lambda_p=\kappa_{\mathbf T}(p)x_p$.
The splitting theorem for potentials
\cite[Proposition~4.5 and Theorem~4.6]{DWZ} removes the trivial summand
without changing the Jacobian algebra or the vertex idempotents.
Since $\partial S\neq\varnothing$, nonzero puncture coefficients give
right-equivalent potentials \cite[Proposition~10.2]{LF}.
Consequently, $J_{\mathbf T}\cong(\Delta^D_{\mathbf T})^{\mathrm{op}}$.

 Let $e=\sum_{\gamma\in\mathbf Q}e_\gamma$. Since  the isomorphism $J_{\mathbf T}\cong(\Delta^D_{\mathbf T})^{\mathrm{op}}$ preserves the arc idempotents, it restricts to the corner
algebras. By \cite[Corollary~4.2]{D}, $e\Delta^D_{\mathbf T}e\cong
\Delta^D_{\mathbf Q}$. Hence
\begin{equation}\label{eq:4.2}
 \Lambda_{\mathbf Q}
 \cong eJ_{\mathbf T}e
 \cong(e\Delta^D_{\mathbf T}e)^{\mathrm{op}}
 \cong(\Delta^D_{\mathbf Q})^{\mathrm{op}},
\end{equation}
and all isomorphisms preserve the arc idempotents.

If $\mathbf Q^\circ$ is sparse, \cite[Theorem~6.2]{D} gives the
presentation in Definition~\ref{def:brauer-graph-algebra}, with
$C_h=C_{\bar h}$ replaced by
$\lambda_{p(h)}C_h=\lambda_{p(\bar h)}C_{\bar h}$.
For each incident puncture $p$, choose a half-edge $h_p$ at $p$.
Set $t_{h_p}=\lambda_p$ and $t_h=1$ for all other half-edges. Then
$\prod_{p(h)=p}t_h=\lambda_p$.
Replacing $\alpha_h$ by $\alpha'_h=t_h\alpha_h$ gives
$C'_h=\lambda_{p(h)}C_h$ and preserves the zero relations.
The endpoint relations become $C'_h=C'_{\bar h}$. Hence
$\Delta^D_{\mathbf Q}\cong A(G_{\mathbf Q})$, and \eqref{eq:4.2} gives
$\Lambda_{\mathbf Q}\cong A(G_{\mathbf Q})^{\mathrm{op}}$.
If $G^{\mathrm{rev}}$ has successor $\sigma^{-1}$, arrow reversal fixes all edge idempotents and gives
\[
 A(G^{\mathrm{rev}})\cong A(G)^{\mathrm{op}},\qquad
 \alpha_h^{\mathrm{rev}}\longmapsto
 \alpha_{\sigma^{-1}h}^{\mathrm{op}}.\qedhere
\]
\end{proof}

Conjugate pairs are allowed in the following tameness statement.

\begin{lemma}\label{top:surface-tame}
Let $\mathbf R$ be a nonempty partial tagged triangulation with all endpoints
at punctures. Then $\Lambda_{\mathbf R}$ is tame.
\end{lemma}
\begin{proof}
Extend $\mathbf R$ to a full tagged triangulation $\mathbf T$ and write
$J_{\mathbf T}=\operatorname{End}_{\mathcal C(\mathbf S)}X(\mathbf T)$.
Let $e$ be the idempotent corresponding to the summand $X(\mathbf R)$
of $X(\mathbf T)$. Then $\Lambda_{\mathbf R}\cong eJ_{\mathbf T}e$.
The algebra $J_{\mathbf T}$ is the associated surface Jacobian algebra,
with the opposite convention for written products as above.
Take all multiplicities to be $1$. By \cite[Theorem~7.1]{D}, the
corresponding Demonet algebra is tame. By \cite[Corollary~5.3]{D}, it is
Morita equivalent to the surface Jacobian algebra. This comparison also
covers self-folded triangles \cite[Section~5, pp.~15--16]{D}, so it
applies when $\mathbf T$ contains conjugate arcs. Tameness is preserved by Morita equivalence and passage to the opposite
algebra. Thus $J_{\mathbf T}$ is tame. Since tameness is also preserved
by idempotent corners \cite[Proposition~4.3]{Iov},
$\Lambda_{\mathbf R}$ is tame.
\end{proof}

\subsection{The triangle-core presentation}
\label{app:triangle-presentation}

Use the projective labels of Lemma~\ref{lem:triangle-stem-splitting}.
As in Section~\ref{subsec:local-splitting}, paths are read from left to
right, and an arrow $a:i\to j$ defines a map $f_a:P_j\to P_i$.

For $d=1$, the comparison~\eqref{eq:4.2} and
\cite[Theorem~4.3]{D} give the presentation
\[
 \begin{gathered}
 a:1\longrightarrow2,\quad b:2\longrightarrow3,\quad
 c:3\longrightarrow1,\quad d_0:1\longrightarrow4,\quad
 e:4\longrightarrow3,\\
 ab=d_0e\quad\text{and}\quad bcd_0=eca=cabc=0.
 \end{gathered}
\]
The quiver vertices $1,2,3,4$ correspond to the arcs indexing
$P_1,P_2,P_3,Q_1$, respectively. Set
$(s,u,t,h,k)=(f_a,f_c,f_b,f_e,f_{d_0})$.  These are the maps used in Lemma~\ref{lem:triangle-stem-splitting}.
To check the relations, first retain the reverse angles $a^\vee:2\longrightarrow1$,
$b^\vee:3\longrightarrow2$ and the terminal angle loop
$\ell:4\longrightarrow4$. The polygon and endpoint relations give
\[
 a^\vee=bc,\qquad b^\vee=ca,\qquad \ell=ecd_0
 \quad\text{and}\quad ab=d_0e.
\]
Eliminating $a^\vee,b^\vee,\ell$ turns the switch relations into
\[
 a^\vee d_0=0\ \Longleftrightarrow\ bcd_0=0,\quad
 eb^\vee=0\ \Longleftrightarrow\ eca=0,\quad
 b^\vee a^\vee=0\ \Longleftrightarrow\ cabc=0.
\]
The remaining endpoint-cycle and termination relations follow from
$ab=d_0e$ and $bcd_0=eca=cabc=0$. For example,
\[
 abc=d_0ec,\quad cab=cd_0e,\quad
 d_0\ell=a(bcd_0)=0\quad\text{and}\quad \ell e=(eca)b=0.
\]
For $d\geq2$, vertex $5$ corresponds to $Q_2$. Replace the terminal
loop by the arrows $f:4\to5$ and $g:5\to4$, with
$fg=ecd_0$ and $d_0f=ge=0$. Along the rest of the stem, use the
Brauer graph relations of Definition~\ref{def:brauer-graph-algebra}.
For $d=0$, the algebra is $k\overrightarrow C_3/J^4$, where $J$ is
the arrow ideal. The nonzero puncture coefficients are normalized
as in the proof of Proposition~\ref{res:4.9}. If the opposite presentation occurs, exchange $P_1,P_3$ and
$v_2,v_3$ to obtain the displayed presentation.

For $d\geq1$, the presentation gives nonzero paths
$a,b,c,d_0,e,bc,ca,ec,cd_0$ and $ab=d_0e$. Since $ab=d_0e$, we obtain
$st=f_af_b=f_{ab}=f_{d_0e}=f_{d_0}f_e=kh$.
For $d=0$, the paths $a,b,c,bc,ca,ab$ are nonzero in
$k\overrightarrow C_3/J^4$.
Proposition~\ref{res:4.1} shows that each listed path spans the
corresponding Hom-space. The only remaining nonzero off-diagonal
Hom-spaces are those between $Q_i$ and $Q_{i+1}$, for $1\leq i<d$;
each is one-dimensional in both directions.



\FloatBarrier

\section*{Acknowledgments}
The research is supported by grants from the National Natural Science Foundation
of China (Nos.~1250011863 and~12371034).




\end{document}